\documentclass[11pt]{article}
\usepackage{fullpage}
\usepackage{authblk}
\usepackage{natbib}
\usepackage{algorithm}
\usepackage{algpseudocode}
\usepackage{amsmath,amsthm,amsfonts}
\usepackage{amsmath}
\usepackage[subnum]{cases}
\usepackage[most]{tcolorbox}
\usepackage{mathrsfs}
\usepackage{amssymb}
\usepackage{color}
\usepackage{mathrsfs}
\usepackage{empheq}
\usepackage{enumitem}
\usepackage{bm}
\usepackage{multirow}
\usepackage{booktabs}
\usepackage{makecell}
\usepackage{graphicx}
\usepackage{subcaption}
\usepackage{comment}
\usepackage{cases}
\usepackage{appendix}
\usepackage{tikz}
\usetikzlibrary{arrows,shapes}
\usepackage[colorlinks= true, linkcolor=brown, citecolor=blue, urlcolor=black]{hyperref}
\usepackage{cleveref}
\usepackage{marginnote}
\usepackage{diagbox} 
\numberwithin{equation}{section}

\theoremstyle{definition}

\newtheorem{theorem}{Theorem}[section]
\newtheorem{lemma}[theorem]{Lemma}

\newtheorem{condition}[theorem]{Condition}
\newtheorem{prop}[theorem]{Proposition}
\newtheorem{defn}[theorem]{Definition}

\newcommand{\defi}{\overset{\triangle}{=}}

\usepackage{booktabs}
\usepackage{pgfplots}
\usetikzlibrary{calc,intersections}
\usepackage{tikz}
\usetikzlibrary{patterns}

\usetikzlibrary{shapes.geometric, arrows, positioning}

\tikzset{
	mybox/.style  = {draw, rectangle, minimum width=2.0cm, minimum height=0.8cm, text centered, text width=2.0cm,   
		font=\normalsize,  align=left},
	box/.style  = {draw, rectangle, minimum width=2.0cm, minimum height=0.6cm, text centered, text width=7.4cm,   
		font=\normalsize,  align=left},		
	box1/.style  = {draw, rectangle, minimum width=2.0cm, minimum height=0.6cm, text centered, text width=5.6cm,   
		font=\normalsize},	
	box2/.style  = {draw, rectangle, minimum width=2.0cm, minimum height=0.6cm, text centered, text width=7.4cm,   
		font=\normalsize,  align=left},	
	myarrow/.style = {line width=0.2pt, draw=black, -triangle 60, postaction={draw, line width=0.2pt, shorten >=10pt,-}}	
}

\tikzstyle{arrow} = [->, >=stealth, -triangle 60]

\allowdisplaybreaks

\makeatletter
\newcommand{\leqnomode}{\tagsleft@true}
\newcommand{\reqnomode}{\tagsleft@false}
\makeatother

\begin{document}
%\maketitle

\title{Numerical Construction of Elliptic Lower-Dimensional Quasi-Periodic Solutions with a Priori Bound}
\author[1]{Mingwei Fu}
\author[2,3]{Bin Shi\thanks{Corresponding author: \url{binshi@fudan.edu.cn} } }
\affil[1]{School of Mathematical Sciences, University of Chinese Academy of Sciences, Beijing 100049, China}
\affil[2]{Center for Mathematics and Interdisciplinary Sciences, Fudan University, Shanghai 200433, China}
\affil[3]{Shanghai Institute for Mathematics and Interdisciplinary Sciences (SIMIS), Shanghai 200433, China}

\date\today

\maketitle

%\bin{Vanishing Viscosity should be included in the main part, the view is key important. At least, it is mentioned a little.}\wjs{OK.}
\begin{abstract}

A numerical framework for constructing full-dimensional quasi-periodic solutions 
in nearly integrable systems was recently developed by~\citet{fu2026numerical}. 
Based on an alternating scheme, this approach effectively overcomes the secular 
drift in angle variables, a fundamental limitation of classical symplectic 
integrators. 
However, in many applications, such as the restricted three-body problem, 
lower-dimensional quasi-periodic solutions hold greater significance. 
The construction of these solutions is considerably more challenging due to the 
presence of normal frequencies, which give rise to intricate resonance phenomena. 
Beyond the standard subspace resonance requirements, one must also account for 
the first and second Melnikov conditions to eliminate small divisors. 
In this study, we extend the proposed alternating numerical scheme to 
compute the elliptic lower-dimensional quasi-periodic solutions. 
Numerical experiments are presented for the H\'{e}non-Heiles model in galaxy 
scale and the Fermi--Pasta--Ulam (FPU) model in quantum scale, 
demonstrating the effectiveness of the proposed method. 
Furthermore, we emphasize that the perturbation Hamiltonian is not merely a 
polynomial with real coefficients but is a general real-valued function. 
As a result, the associated perturbation operator exhibits Gevrey-type decay 
without possessing a Hankel-like structure. 
Meanwhile, we further simplify the multi-scale analysis by exploiting the 
resolvent identity, showing that the global inverse can be expressed linearly in 
terms of local inverses via the gluing procedure. 
This representation reveals a regime-dependent interaction structure: weak 
interactions dominate at short range, while strong interactions emerge at long 
range. This balance ensures that the Gevrey decay of the inverse remains 
uniformly controlled. 
Moreover, within this linear representation, the inversion conditions provide a 
clearer characterization of the associated localization properties.

\end{abstract}

%\thispagestyle{empty}
%\setcounter{page}{0}
%
%\newpage
%\tableofcontents
%\newpage

\section{Introduction}
\label{sec: intro}

Nearly integrable systems constitute a fundamental class of Hamiltonian dynamics, 
serving as a bridge between perfectly predictable, ``solvable'' dynamics and the 
complex, chaotic phenomena observed in nature. 
They arise across a wide range of physical scales, from the macroscopic long-time 
stability of planetary orbits in our solar system to the microscopic vibrations 
of atoms in crystalline lattices. 
In action-angle coordinates, such systems are described by a Hamiltonian of the 
form: 
\[
H(I, \theta) = H_0(I) + \varepsilon H_1(I, \theta),
\]
where $H_0(I)$ is the integrable component and $\varepsilon H_1(I, \theta)$ is a 
small nonlinear perturbation with $0 < \varepsilon \ll 1$. 
The phase space is equipped with the standard symplectic form 
$\varpi = dI \wedge d\theta$.

In the absence of perturbation ($\varepsilon = 0$), the dynamics are governed 
exclusively by $H_0$. The resulting motion is explicitly given by: 
\begin{equation}
    \label{eqn: linear-near-soln}
    I(t) = I_0, \qquad \theta(t) = \theta_0 + \omega(I_0)t,
\end{equation}
where the frequency satisfies $\omega(I_0) = \nabla H_0(I_0)$. 
As a result, the trajectories are confined to $n$-dimensional invariant tori in 
phase space, exhibiting perfectly quasi-periodic behavior. 
To study the effect of the perturbations, we consider the dynamics in a 
neighborhood of the unperturbed torus $\{ I_0 \} \times \mathbb{T}^n$. 
By introducing the shifted action variable $J = I - I_0$, we translate this torus 
to $\{ 0 \} \times \mathbb{T}^n$. 
A Taylor expansion of $H_0(I)$ about $I_0$ allows us to reformulate the 
Hamiltonian as:
\begin{equation}
    \label{eqn: near-integrable-1}
    H(J, \theta; I_0) = \langle \omega(I_0), J \rangle + \varepsilon H_1(J, \theta; I_0),
\end{equation}
where the constant term $H_0(I_0)$ is omitted, as it does not affect the 
equations of motion. 
Note that the redefined perturbation term $H_1$ incorporates both the original 
perturbation and the higher-order terms (second-order and above) arising from the 
expansion of $H_0$. 
A rigorous derivation can be found in~\citet{poschel2001lecture}.

We restrict our attention to initial actions $I_0$ in a bounded domain 
$D_n \subseteq \mathbb{R}^n$, which necessitates an analysis of the frequency 
map $\omega: D_n \rightarrow \Omega_n \subseteq \mathbb{R}^n$. 
If this map is a diffeomorphism, the Hamiltonian~\eqref{eqn: near-integrable-1} 
can be reparameterized directly by the frequency $\omega \in \Omega_n$: 
\begin{equation}
    \label{eqn: near-integrable-2}
    H(J, \theta; \omega) = \langle \omega, J \rangle + \varepsilon H_1(J, \theta; \omega). 
\end{equation}
The validity of this reparameterization is ensured by the isoenergetic 
non-degeneracy condition. Specifically, if the following determinant remains 
non-zero for all $I_0 \in D_n$:
\[
\det 
\begin{pmatrix}  
    \nabla^2 H_0 & \nabla H_0 \\ 
    (\nabla H_0)^{\top}   & 0   
\end{pmatrix} 
\neq 0, 
\]
then the frequency map $\omega: I_0 \mapsto \omega(I_0)$ is guaranteed to be a 
diffeomorphism (see e.g.~\citet{arnol2013mathematical}). 
Throughout this work, we assume that all components of the frequency 
$\omega \in \Omega_n$ are non-zero, i.e., $\omega_j \neq 0$ for 
$j = 1, \ldots, n$, and that $H_1$ is at least quadratic in $J$. 
If $H_1$ contained linear terms in $J$, it would contribute to the effective 
frequency, necessitating a frequency drift correction to preserve the resonance 
conditions.

%%%%%%%%%%%%%%%%%%%%%%%%%%%%%%%%%%%%%%%%%%%%%%%%%%%%%%%%%%%%%%%%%%%%%%%%%%%%%%%%%%%%%%%%%%%%%%%%%
\subsection{Lower-dimensional tori}
\label{subsec: lower-dimensional-tori}

We now consider solutions of the nearly integrable 
system~\eqref{eqn: near-integrable-2}, with particular emphasis on perturbed 
trajectories near the unperturbed $n$-torus $\{0\} \times \mathbb{T}^n$. 
To characterize the domain $D_n$, let $\mathcal{H}_j $ denote the $j$-th 
coordinate hyperplane
\[
\mathcal{H}_j = \{ I_{0} \in \mathbb{R}^n | I_{0,j} = 0 \} 
\]
for $j=1,\ldots, n$. 
When the domain $D_n$ is set apart from these hyperplanes, i.e., 
$D_n \cap \mathcal{H}_j = \varnothing$ for all $j$, then the system operates in 
the regime of full-dimensional tori. 
In this setting, the classical KAM theorem, pioneered 
by~\citet{kolmogorov1954preservation, arnold1963proof} 
and~\citet{moser1962invariant}, serves as a fundamental landmark, which asserts 
that most invariant tori persist under sufficiently small perturbations, 
provided that they remain away from the resonant set:
\begin{equation}
    \label{eqn: diophantine}
    \langle k, \omega \rangle = 0,
\end{equation}
for all non-zero integer vectors $k \in \mathbb{Z}^n \setminus \{0\}$.

In many physical systems, however, some components of the initial action $I_0$ 
may vanish, i.e., $D_n \cap \mathcal{H}_j \neq \varnothing$ for some $j$, as 
noted by~\citet{moser1966theory}. 
A classic example is the restricted three-body problem 
(e.g., the Sun-Earth-Moon system), where one action variable, associated with the 
eccentricity of the primary (e.g., the Sun), vanishes in the integrable limit. 
In such cases, we may assume without loss of generality that the first $m$ components 
are non-zero while the remaining $n-m$ components vanish: 
$I_0 \in D_m \times \{0\} \subseteq \mathbb{R}^m \times \mathbb{R}^{n-m}$. 
The frequency then decomposes into tangential and normal components, 
$\omega = (\omega_T, \omega_N)$. 
The Hamiltonian~\eqref{eqn: near-integrable-2} can then be rewritten as
\begin{equation}
    \label{eqn: near-integrable-lower}
    H(J, \theta, z, \overline{z}; \omega) = \langle \omega_T, J \rangle + \langle \omega_{N}(\omega_T) z, \overline{z} \rangle + \varepsilon H_1(J, \theta; z, \overline{z}; \omega_T). 
\end{equation}
where $\omega_T(I_0) = \nabla H_0(I_0)$ satisfies the isoenergetic non-degeneracy 
condition. 
The phase space carries the symplectic form 
$\varpi = dJ \wedge d\theta + idz \wedge d\overline{z}$. 
Since only $m < n$ components are non-zero, the associated invariant manifolds are 
$m$-dimensional tori embedded in $2n$-dimensional space, referred to as 
lower-dimensional tori. 
We focus on the behavior of trajectories near the unperturbed torus 
$\{0\} \times \mathbb{T}^m \times \{ 0\} \times \{0\}$. 
When the normal frequencies are real, i.e., $\omega_N \in \mathbb{R}^{n-m}$, these 
lower-dimensional invariant tori are called elliptic. 
In this case, the resonant structure becomes more intricate. 
As first noted by~\citet{melnikov1965certain, melnikov1968certain}, in addition to 
the standard resonance condition~\eqref{eqn: diophantine} for all 
$k \in \mathbb{Z}^m \setminus \{0\}$, one must also consider the first Melnikov 
resonant condition: 
\begin{equation}
    \label{eqn: melnikov-1} 
    \langle k, \omega_T \rangle - \omega_{j} = 0,
\end{equation}
for all $ k \in \mathbb{Z}^m $ and $j = m+1, \ldots, n$, as well as the second 
Melnikov resonant condition: 
\begin{subequations}
    \label{eqn: melnikov-2}
    \begin{empheq}[left=\empheqlbrace]{align} 
    &  \langle k, \omega_T \rangle - \omega_{j_1} + \omega_{j_2}= 0,     \label{eqn: melnikov-2-1} \\
    &  \langle k, \omega_T \rangle - \omega_{j_1} - \omega_{j_2}= 0,     \label{eqn: melnikov-2-2} 
    \end{empheq}
\end{subequations}
for all $ k \in \mathbb{Z}^m $ and $j_1, j_2 = m+1, \ldots, n$. 
The KAM theorems established by~\citet{eliasson1988perturbations} 
and~\citet{poschel1989elliptic} show that most elliptic lower-dimensional invariant 
tori persist under sufficiently small perturbations, provided that these resonant 
conditions are avoided. 
In contrast, if some normal frequencies are complex, i.e., 
$\mathrm{Im}(\omega_j) \neq 0$ for some $j > m$, the corresponding lower-dimensional 
tori are termed hyperbolic, as first discussed by~\citet{moser1967convergent}. 
Unlike their elliptic counterparts, these hyperbolic tori are inherently unstable. 
Transversal intersections of their stable and unstable manifolds generate chains of 
homoclinic or heteroclinic orbits, which give rise to Arnold 
diffusion~\citep{arnold1964instability}, a phenomenon where trajectories undergo a 
slow, long-term drift in action variables across the phase space.

%%%%%%%%%%%%%%%%%%%%%%%%%%%%%%%%%%%%%%%%%%%%%%%%%%%%%%%%%%%%%%%%%%%%%%%%%%%%%%%%%%%%%%%%%%%%%%%%%
\subsection{Numerical scheme and iterative framework}
\label{subsec: numerical-scheme}

Unlike traditional normal form techniques in KAM theory, which employ successive 
symplectic transformations to separate perturbed dynamics from integrable motion, 
the Craig-Wayne-Bourgain (CWB) scheme adopts a more direct strategy. 
Originally established by~\citet{bourgain1997melnikov, bourgain2005green} for 
constructing elliptic lower-dimensional quasi-periodic solutions, this method avoids 
the explicit decomposition into integrable and perturbative components. 
Instead, it works directly with the Hamiltonian: 
\begin{equation}
    \label{eqn: near-integrable-3}
    H(I, \theta; \omega) = \langle \omega, I \rangle + \varepsilon H_1(I, \theta; \omega_T), 
\end{equation}
where the initial action is confined to a lower-dimensional subspace: 
$I_0 \in D_m \times \{ 0\} \subseteq \mathbb{R}^m \times \mathbb{R}^{n-m}$. 
Traditional numerical simulations typically rely on symplectic 
integrators~\citep{feng1995collected, hairer2006geometric, gauckler2018dynamics}. 
While these methods are highly effective at preserving underlying geometric structures, 
they suffer from an inherent limitation: the accumulation of phase errors, which 
leads to secular drift in the angle variables. 
This drift significantly degrades the long-time accuracy of quasi-periodic 
trajectories. 
In this paper, we generalize the numerical framework introduced 
by~\citet{fu2026numerical} to construct elliptic lower-dimensional quasi-periodic 
solutions with any prescribed level of accuracy, while effectively mitigating the 
drift issues that arise in standard time-integration methods.

By applying the complex canonical transformation, $z_j = \sqrt{I_j} e^{i\theta_j}$ 
for $j = 1, \ldots, n$, the Hamiltonian~\eqref{eqn: near-integrable-3} is rewritten as
\begin{equation}
    \label{eqn: near-integrable-use}
    H(z, \overline{z}) = \langle \omega \odot z, \overline{z} \rangle + \varepsilon H_1(z, \overline{z}), 
\end{equation}
where $\odot$ denotes componentwise multiplication. 
For simplicity, we omit the explicit dependence on the tangential frequencies 
$\omega_T$ within $H_1$, since all relevant properties of the perturbation hold 
uniformly across the domain with respect to $\omega_T$; this convention is adopted 
throughout the remainder of the paper. 
Under the symplectic form $dz \wedge d\overline{z} = - i dI \wedge d\theta$, the 
Hamilton's equations of motion are expressed as:
\begin{equation}
    \label{eqn: hamilton-system}
    \dot{z} = i \omega \odot z + i \varepsilon \frac{\partial H_1}{\partial \overline{z}}. 
\end{equation}
We now transform the Hamiltonian system~\eqref{eqn: hamilton-system} to an algebraic 
system available for numerical computation. 
Since the procedure closely follows~\citet{fu2026numerical}, we only provide a brief 
outline here. We seek quasi-periodic solutions of the form
\begin{equation}
    \label{eqn: quasi-periodic}
    z(\omega_T' t) = \sum_{k \in \mathbb{Z}^m} \hat{z}(k) e^{i \langle k, \omega_T' \rangle t}, 
\end{equation}
where $\hat{z}(k) \in \mathbb{R}^n$ are real Fourier coefficients, which constitute a 
key feature of the present approach. 
Substituting this ansatz~\eqref{eqn: quasi-periodic} into the Hamiltonian 
system~\eqref{eqn: hamilton-system}, we derive the following algebraic system as
\begin{equation}
    \label{eqn: lattice-algebraic}
    - \langle k, \omega'_T \rangle \hat{z} + \omega \odot \hat{z} + \varepsilon \hat{X} = 0, 
\end{equation}
where $\hat{X} = \{\hat{X}(k)\}_{k \in \mathbb{Z}^m}$ is the vector of Fourier 
coefficients corresponding to the perturbation 
$X = \partial H_1 / \partial \overline{z}$. 
We here focus on the elliptic lower-dimensional quasi-periodic solutions, with 
initial action $I_0 = a = (a_1, \ldots, a_m, 0, \ldots, 0 )$. 
Let $\mathcal{M} = \{ 1, \ldots, m \}$, and define the set of resonant indices 
\begin{equation}
    \label{eqn: resonant-set}
    \mathcal{S} = \left\{ (j, e_j) \big | j = 1, \ldots, m \right\},
\end{equation}
where $e_j \in \mathbb{Z}^m$ is the standard basis vector. 
Accordingly, we decompose the Fourier coefficients and the vector field into 
resonant and non-resonant parts: $\hat{z} = (\hat{z}_q, \hat{z}_p)^{\top}$ and 
$\hat{X} = (\hat{X}_q, \hat{X}_p)^{\top}$, where the resonant part is given by 
$\hat{z}_q = (\hat{z}_1(e_1), \ldots, \hat{z}_m(e_m))$ and 
$\hat{X}_q = (\hat{X}_1(e_1), \ldots, \hat{X}_m(e_m))$. 
For resonant indices $(j, k) \in \mathcal{S}$, the coefficients are fixed as 
$\hat{z}_j(e_j) = a_j$, so that $\hat{z}_q = a$. 
The tangential frequencies are then determined by the $Q$-equation: 
\begin{equation}
    \label{eqn: q-eqn}
    (- \omega_T' + \omega_T) \odot a + \varepsilon \hat{X}_q = 0. 
\end{equation}
For non-resonant indices $(j, k) \notin \mathcal{S}$, the non-resonant coefficients 
$\hat{z}_p$ are obtained from the $P$-equation: 
\begin{equation}
    \label{eqn: p-eqn}
    - \langle k, \omega_T' \rangle \hat{z}_p + \omega \odot \hat{z}_p + \varepsilon \hat{X}_p = 0.
\end{equation}

Since the $P$-equation~\eqref{eqn: p-eqn} constitutes an infinite-dimensional 
nonlinear algebraic equation, we employ a dimension-enlarged Newton scheme for its 
numerical solution. 
To this end, we define the nonlinear operator $F$, acting on the non-resonant 
Fourier coefficients $\hat{z}_p$, as:
\begin{equation}
    \label{eqn: nonlinear}
    F(\omega_T', \hat{z}_p; a, \omega) = - \langle k, \omega_T' \rangle \hat{z}_p + \omega \odot \hat{z}_p + \varepsilon \hat{X}_p(\hat{z}_p; a).
\end{equation}
The associated linearized (tangent) operator takes the form $T = D + \varepsilon S$, 
where $D$ is a diagonal operator and $S$ is the Jacobian arising from the nonlinear 
perturbation. 
More precisely, 
\begin{equation}
    \label{eqn: tangent-defn}
    D = \mathrm{diag}(- \langle k, \omega_T' \rangle E_{n} + \omega), \quad \mathrm{and} \quad S = \frac{ \partial \hat{X}_p(\omega, \hat{z}_p; a)}{\partial \hat{z}_p}, 
\end{equation}
where $E_{n}$ denotes the identity matrix on $\mathbb{R}^n$. 
In the case of lower-dimensional quasi-periodic solutions ($m<n$), the integer 
combinations $\langle k, \omega'_T \rangle$ involve fewer independent variables. 
As a consequence, the diagonal entries of $D$ exhibit higher multiplicity, meaning 
that multiple Fourier modes share the same tangential frequency components. 
This leads to a ``clustering'' of eigenvalues and a significant reduction in the 
prevalence of small divisors. 
In this sense, lower-dimensional quasi-periodic solutions are generally more 
accessible than their full-dimensional counterparts.

Let $M = M(\varepsilon)$ be a positive integer, whose explicit definition is given 
in~\eqref{eqn: M}, and define the truncation scale by $N_r := M^r$. 
Based on the $Q$-equation~\eqref{eqn: q-eqn} and the $P$-equation~\eqref{eqn: p-eqn}, 
we propose the following iterative scheme to update the frequency and the non-resonant 
coefficients from the $r$-th to the $(r+1)$-th step:
\begin{subequations}
    \label{eqn: numerical-scheme}
    \begin{empheq}[left=\empheqlbrace]{align} 
        & \omega^{(r+1)}_T = \omega_T + \varepsilon \hat{X}_q(\hat{z}_p^{(r)}; a) \odot a^{-1},           \label{eqn: frequency-iter} \\
        & \hat{z}_p^{(r+1)} = \hat{z}_p^{(r)} - \left[ (T + \varepsilon B)_{N_{r+1}}^{-1}\left( \hat{z}_p^{(r)}, \omega^{(r+1)}_T; a \right)  \right] F\left( \hat{z}_p^{(r)}, \omega^{(r+1)}_T; a, \omega \right). \label{eqn: coeffcient-iter}
    \end{empheq}        
\end{subequations}
where $a^{-1} = (a_1^{-1}, \ldots, a_m^{-1})$. 
The operator $B$ is introduced to ensure the convergence of the iteration which is 
given by
\begin{equation}
    \label{eqn: b-operator}
    B = - \frac{1}{e} \left( \frac{ \partial \langle k, \hat{X}_q \rangle }{\partial \hat{z}_p} \right) \hat{z}_p^{\top}.
\end{equation}
The restricted operator is defined as 
$(T + \varepsilon B)_N = P_N (T + \varepsilon B) P_N$, where $P_N$ is the projection 
operator given by
\[
(P_N \hat{z})(k) := \left\{ 
    \begin{aligned}
        & \hat{z}(k), && \mathrm{if}\; k \in \Lambda_{N}, \\
        & 0,          && \mathrm{if}\; k \notin \Lambda_{N},   
    \end{aligned} 
\right.
\]
and the truncation domain, or box lattice, is defined as:
\begin{equation}
    \label{eqn: box-lattice}
    \Lambda_N := \left\{ k \in \mathbb{Z}^m \big | |k|_{\infty} \leq N \right\}.
\end{equation}
Throughout this paper, unless otherwise specified, $|\cdot|$ and $|\cdot|_{\infty}$ 
denote the $\ell_1$-norm and the $\ell_{\infty}$-norm, respectively, on 
$\mathbb{Z}^m$ or $\mathbb{R}^m$.

%%%%%%%%%%%%%%%%%%%%%%%%%%%%%%%%%%%%%%%%%%%%%%%%%%%%%%%%%%%%%%%%%%%%%%%%%%%%%%%%%%%%%%%%%%%%%%%%%
\subsection{The main theorem: A priori bound}
\label{subsec: priori-estimate}

In this section, we establish an a priori bound for the alternating numerical 
procedure~\eqref{eqn: numerical-scheme}. 
For the Fourier coefficients $\hat{z} = \{ \hat{z}(k) \}_{k \in \mathbb{Z}^m}$, we 
denote the standard $\ell_2$-norm as:
\[
    \| \hat{z}\|_2 = \left( \sum_{k \in \mathbb{Z}^m} \|\hat{z}(k) \|_2^2 \right)^{\frac12}.
\]
Given that the Fourier vector $\hat{z}$ is continuously differentiable with respect 
to the tangential frequency $\omega_T$, we introduce the normed vector space 
\begin{equation}
    \label{eqn: l2-space}
    \mathscr{H}(\mathbb{Z}^m) = \left\{ \hat{z} = \left\{\hat{z}(k)\right\}_{k \in \mathbb{Z}^m} ~\bigg|~ \hat{z}(k) \in \mathbb{R}^n, \; \|\hat{z}\| = \|\hat{z}\|_2 + \|\partial_{\omega_T}\hat{z}\|_2 < \infty \right\}. 
\end{equation}
To characterize the regularity of solutions, we define the Gevrey decay set: 
\begin{equation}
    \label{eqn: gevrey-l2}
    \mathcal{K}(s) = \left\{ \hat{z} \in \mathscr{H}(\mathbb{Z}^m) ~\bigg|~ \sup_{k \in \mathbb{Z}^m} \big( \|\hat{z}(k)\| \exp\left\{|k|^s\right\} \big) \leq 1 \right\},
\end{equation}
where the sub-exponential decay of these Fourier coefficients plays a key role in 
preserving the regularity of the iterative scheme.

\begin{theorem}[Convergence and A Priori Bound]
    \label{thm: main}
    Let $\Omega \subseteq \mathbb{R}^m$ be a bounded domain, and let $\tau > m-1$ be fixed. 
    There exists a critical threshold $\varepsilon_0 = \varepsilon_0(H_1, \Omega, a) > 0$ 
    such that for any $0 < \varepsilon \leq \varepsilon_0$, the following properties 
    hold:
    \begin{itemize}
        \item[1.] \textbf{Measure of resonance}: There exist constants 
                  $\kappa = \kappa(n, m, \tau, \Omega) >0$ and $\delta = \delta(\varepsilon) > 0$, 
                  with $\delta \rightarrow 0$ as $\varepsilon \rightarrow 0$, such that 
                  the excluded set of ``bad'' frequencies $\Omega^{\star}$ satisfies 
                  $\mathrm{mes}(\Omega^{\star}) \leq \kappa + \delta$.

        \item[2.] \textbf{Iterative sequence}: For any initial pair
                  $\omega_T^{(0)} \in \Omega \setminus \Omega^\star$ and 
                  $\hat{z}_p^{(0)} = 0$, 
                  the alternating numerical procedure~\eqref{eqn: numerical-scheme}
                  generates a sequence of iterates 
                  $\{ (\omega_T^{(r)}, \hat{z}_p^{(r)}) \}_{r=0}^{\infty}$. 

        \item[3.] \textbf{Convergence}: There exists a Gevrey exponent $s = s(\varepsilon)$ 
                  such that this sequence remains within the product space 
                  $(\Omega \setminus \Omega^\star) \times \mathcal{K}(s)$ and converges 
                  to the exact solution pair $(\omega_T^\star, \hat{z}_p^\star)$.

        \item[4.] \textbf{A priori bound}: There exists a constant $M = M(\varepsilon)$ 
                  such that the convergence is characterized by the following 
                  super-exponential error bounds:
                  \begin{subequations}
                      \label{eqn: convergence-vector-frequency}
                      \begin{empheq}[left=\empheqlbrace]{align} 
                          & \| \hat{z}^{(r)} - \hat{z}^\star \| \leq \exp \left\{ -\frac{3}{2} (M^s)^r \right\},      \label{eqn: z-hat-converge}                 \\
                          & |\omega_T^{(r)} - \omega_T^\star| \leq \exp \left\{ -\frac{3}{2} (M^s)^r \right\}.       \label{eqn: frequency-hat-converge}
                      \end{empheq}  
                  \end{subequations}
    \end{itemize}
    Furthermore, for any $t \in [0, M^r]$, the approximate solution $z^{(r)}(t)$ 
    satisfies the error bound: 
    \begin{equation}
        \label{eqn: convergence-soln}
        \| z^{(r)}(t) - z^{\star}(t) \| \leq \exp \left\{ -(M^s)^r \right\}.
    \end{equation}
    This implies that as $r \rightarrow \infty$, the numerical solution $z^{(r)}(t)$ 
    converges to the exact quasi-periodic solution $z^{\star}(t) $ for any $t \geq 0$. 
\end{theorem}

\Cref{thm: main} remains valid under the stated analytic assumptions. 
The complete proof  is finalized in~\Cref{sec: iterative-scheme}, synthesizing the 
implementation framework established in~\Cref{sec: implement} and the multi-scale 
analysis presented in~\Cref{sec: multi-scale}.

%%%%%%%%%%%%%%%%%%%%%%%%%%%%%%%%%%%%%%%%%%%%%%%%%%%%%%%%%%%%%%%%%%%%%%%%%%%%%%%%%%%%%%%%%%%%%%%%%
\subsection{Organization of the paper}
\label{subsec: organization}

The remainder of this paper is organized as follows. 
\Cref{sec: basic-setting} establishes uniform properties of the associated vectors 
and operators within the Gevrey decay class, which serve as the foundation for the 
subsequent analysis.
\Cref{sec: implement} formulates the implementation conditions and provides the 
corresponding proofs using inductive arguments.
\Cref{sec: multi-scale} derives the inversion and localization conditions for the 
tangent operator restricted to outer boxes via a multi-scale analysis, built upon 
inductive bases for two types of small-scale cases.
\Cref{sec: iterative-scheme} establishes the iterative lemma and provides the proof 
of the main convergence results.
\Cref{sec: numer} presents numerical experiments on both the macroscopic 
H\'{e}non-Heiles model and the microscopic Fermi–Pasta–Ulam (FPU) model. 
Finally,~\Cref{sec: conclu} concludes the paper and outlines directions for future 
research.

\section{Uniform properties within the Gevrey decay set}
\label{sec: basic-setting}

In this section, we establish uniform decay and boundedness properties for the 
associated vectors and operators within the Gevrey decay set $\mathcal{K}(s)$. 
Let $\mathbb{T}^m = [0, 2\pi]^m$ denote the $m$-dimensional torus. 
For any $k \in \mathbb{Z}^{m}$, the component of the vector field $\hat{X}$ is given 
by
\begin{equation}
    \label{eqn: vec-fourier}
    \hat{X}(k) = \frac{1}{(2\pi)^{m}} \int_{\mathbb{T}^{m}} \frac{\partial H_1}{\partial \overline{z}} e^{-i \langle k, \theta \rangle } d\theta.
\end{equation}
Since $H_1$ is a polynomial with real coefficients, $\hat{X}(k)$ is real-valued. 
Furthermore, because the perturbation $H_1$ is real, the following symmetric 
relation holds:
\begin{equation}
    \label{eqn: real-identity}
    \frac{1}{(2\pi)^{m}} \int_{\mathbb{T}^{m}} \frac{\partial H_1}{\partial \overline{z}} e^{-i \langle k, \theta \rangle } d\theta =  \frac{1}{(2\pi)^{m}} \int_{\mathbb{T}^{m}} \frac{\partial H_1}{\partial z} e^{i \langle k, \theta \rangle } d\theta.
\end{equation}
We now establish the following property for the vector field 
$\hat{X} = (\hat{X}_q, \hat{X}_p)^{\top}$ (see~\Cref{subsec: vector-field} for a 
formal proof).

\begin{prop}
    \label{prop: vector-field}
    Suppose $\hat{z} \in \mathcal{K}(s)$. 
    Then there exist constants $\gamma_1, \gamma_2 > 0$, depending on $H_1$ and $s$, 
    such that the vector field $\hat{X}$ satisfies
    \begin{subequations}
        \label{eqn: vector-field}
        \begin{empheq}[left=\empheqlbrace]{align} 
            &  \sup_{k \in \mathbb{Z}^m} \left( \|\hat{X}(k)\| \exp\left\{ |k|^s \right\} \right) \leq \gamma_1,               \label{eqn: vector-field-comp}                 \\
            &  \max \big \{ \|\hat{X}_q\|, \|\hat{X}_p\| \big \} \leq \|\hat{X}\| \leq \gamma_2.                                   \label{eqn: vector-field-bound}
        \end{empheq}  
    \end{subequations}
\end{prop}

Next, we analyze the tangent operator 
$\partial \hat{X} / \partial \hat{z}_p = ( \partial \hat{X}_q / \partial \hat{z}_p, \partial \hat{X}_p / \partial \hat{z}_p )^{\top}$. 
Recalling the tangent operator defined in~\eqref{eqn: tangent-defn}, it follows that 
$S = \partial \hat{X}_p / \partial \hat{z}_p$. 
For any two resonant indices $k, k' \notin \mathcal{S}$, we 
utilize~\eqref{eqn: real-identity} to derive its component as
\[    
    S(k, k') = S(k-k') = \frac{2}{(2\pi)^{m}} \int_{\mathbb{T}^{m}} \frac{\partial^2 H_1}{\partial z \partial \overline{z}} \cos\left( \langle k - k', \theta \rangle\right) d\theta. 
\]
The tangent operator satisfies the following property 
(see~\Cref{subsec: convolution} for a formal proof).

\begin{prop}
    \label{prop: convolution}
    Suppose $\hat{z} \in \mathcal{K}(s)$. 
    Then there exist constants $\gamma_3, \gamma_4 > 0$, depending on $H_1$ and $s$, 
    such that the tangent operator $S $ is symmetric and satisfies:
    \begin{subequations}
        \label{eqn: tangent}
        \begin{empheq}[left=\empheqlbrace]{align} 
            & \sup_{k , k' \notin \mathcal{S}} \left(\|S(k, k')\| \exp\left\{ |k - k'|^s \right\} \right) \leq \gamma_3, \label{eqn: tangent-comp} \\
            & \max \left\{ \left\| \frac{\partial \hat{X}_q}{\partial \hat{z}_p} \right\|, \|S\| \right\} \leq \left\| \frac{\partial \hat{X}}{\partial \hat{z}_p} \right\| \leq \gamma_4. \label{eqn: tangent-bound}
        \end{empheq}        
    \end{subequations}
\end{prop}

To facilitate the iteration lemma~(\Cref{lem: iteration}), we also require estimates 
for the second derivative of the vector field $\hat{X}$, as stated below 
(see~\Cref{subsec: tensor} for the proof).

\begin{prop}
    \label{prop: tensor}
    Suppose $\hat{z} \in \mathcal{K}(s)$, and assume further that 
    $\mathrm{supp}~\hat{z} \subseteq \Lambda_N $. 
    Then there exists a constant $\gamma_5 = \gamma_5(H_1, s, m)$ such that 
    \begin{equation}
        \label{eqn: tensor}
        \max \left\{ \left\| \frac{\partial^2 \hat{X}_q}{\partial \hat{z}_p^2} \right\|, \left\| \frac{\partial S}{\partial \hat{z}_p } \right\| \right\} \leq \left\| \frac{\partial^2 \hat{X}}{\partial \hat{z}_p^2}  \right\| \leq \gamma_5 (2N+1)^\frac{m}{2}. 
    \end{equation}
\end{prop}

In the numerical scheme, we introduce the frequency 
iteration~\eqref{eqn: numerical-scheme}. 
To control the associated error, we define an additional operator $B$, whose 
properties are summarized below (see~\Cref{subsec: b-operator} for the proof).

\begin{prop}
    \label{prop: B-operator}
    Suppose $\hat{z} \in \mathcal{K}(s)$. 
    Then there exist constants $\gamma_6, \gamma_7 > 0$, depending on $H_1$ and $s$, 
    such that the operator $B$ satisfies: 
    \begin{subequations}
        \label{eqn: b-operator-bound}
        \begin{empheq}[left=\empheqlbrace]{align} 
            & \sup_{k, k' \notin \mathcal{S}} \left( \|B(k, k')\| \exp\left\{ |k|^s + |k'|^s \right\} \right) \leq \gamma_6, \label{eqn: b-comp} \\
            & \| B \| \leq \gamma_7. \label{eqn: b-bound}
        \end{empheq}        
    \end{subequations}
\end{prop}

Let $\gamma = \max_{1 \leq j \leq 7} \gamma_j$, where each $\gamma_j$ depends on the 
perturbation $H_1$. 
By appropriately scaling $H_1$, or equivalently, by choosing $\varepsilon$ 
sufficiently small, we may assume without loss of generality that 
$\gamma \leq 1/(2e) < 1/2$. 
Utilizing the boundedness of the vector field $\hat{X}_q$ and its tangent operator 
from~\Cref{prop: vector-field} and~\Cref{prop: convolution}, we can establish the 
following result concerning the tangential frequency drift.

\begin{prop}
    \label{prop: frequency-drift}
    Suppose $\hat{z} \in \mathcal{K}(s)$. 
    Then the tangential frequency drift determined by the 
    $Q$-equation~\eqref{eqn: q-eqn} satisfies:
    \begin{equation}
        \label{eqn: frequency-drift-bound}
        \left| \omega'_T - \omega_T \right| \leq \varepsilon,
    \end{equation}
    and the Jacobian of the tangential frequency map satisfies: 
    \begin{equation}
        \label{eqn: frequency-drift-derivative}
        1 - \varepsilon \leq \left\| \frac{\partial \omega'_T }{\partial \omega_T} \right\|_2 \leq 1 + \varepsilon.
    \end{equation}
\end{prop}

\Cref{prop: frequency-drift} ensures that the mapping between the drifted tangential 
frequency $\omega'_T$ and the original tangential frequency $\omega_T$ is a 
diffeomorphism.

\section{Implementation conditions}
\label{sec: implement}

The heart of the alternating numerical scheme~\eqref{eqn: numerical-scheme} hinges on 
updating the non-resonant vector via the dimension-enlarged Newton 
scheme~\eqref{eqn: coeffcient-iter}. 
To overcome the small-divisor problem, we must rigorously control the inverse of the 
truncated matrix $(T + \varepsilon B)_N$ as the dimension grows. 
To this end, we introduce a scale-dependent threshold parameter:
\begin{equation}
    \label{eqn: scale-dependent}
    \varepsilon_N := \exp\left\{ - (\log N)^{15} \right\}. 
\end{equation}
We formulate the following implementation conditions to ensure the convergence of 
the alternating numerical scheme.

\begin{tcolorbox}[breakable]
    \begin{condition}[Implementation Conditions]
        \label{cond: implementation}
        The truncated linearized operator $(T + \varepsilon B)_{N}$ is required to 
        satisfy the following two conditions:
        \begin{itemize}
            \item[(1)] \textbf{Inversion condition}: The operator norm of the inverse 
                       is bounded by the reciprocal of the threshold parameter:
                       \begin{equation}
                           \label{eqn: inverse-grow}
                           \| (T + \varepsilon B)_{N}^{-1} \|_2 \leq \frac{1}{\varepsilon_{N}}.
                       \end{equation}

            \item[(2)] \textbf{Localization condition}: For sufficiently large spatial 
                       separations within the lattice $\Lambda_N$, specifically when 
                       $|k - k'| \geq N^{\frac12}$, the entries 
                       of the inverse matrix must exhibit Gevrey-type decay:
                       \begin{equation}
                           \label{eqn: off-diagonal}
                           | (T + \varepsilon B)_{N}^{-1}\left( k, k' \right) | \leq \exp\left\{ -  \frac{|k - k'|^{s}}{2} \right\}.
                       \end{equation}
        \end{itemize}
    \end{condition}
\end{tcolorbox}

These implementation conditions are satisfied by systematically excluding ``bad'' 
frequencies. 
To handle the restricted operators, we distinguish between small and large scales 
using the threshold $M$ dictated by the perturbation intensity $\varepsilon$:
\begin{equation}
    \label{eqn: M}
    M = \exp\left\{ \left( \log \frac{1}{\varepsilon} \right)^{\frac{1}{20}} \right\}. 
\end{equation}
The scale $N$ is subsequently divided into two distinct regimes: the small-scale 
regime $M_0 \leq N \leq M$ and the large-scale regime $N > M$, where the initial 
scale is defined as 
$M_0 = \exp \left\{ \left( \log M \right)^{\frac{1}{20}} \right\}$. 
For notational simplicity, we denote the linearized operator as 
$L := T + \varepsilon B$ throughout the remainder of the paper.

%======================================================%
\subsection{Nearly-resonant sets}
\label{subsec: nearly-resonant}

We first consider the small-scale regime $M_0 \leq N \leq M$. 
To guarantee the invertibility of the restricted operator, we must exclude the 
frequencies that trigger resonance. 
In the case of full-dimensional quasi-periodic solutions, nearly-resonant sets 
involve the entire frequency $\omega$. 
However, in the lower-dimensional setting, only the tangent frequency $\omega_T$ is 
relevant. 
Accordingly, we define the following nearly-resonant sets in the tangent frequency 
space, corresponding to the resonance conditions outlined 
in~\Cref{subsec: lower-dimensional-tori}: 
\begin{itemize}
    \item \textbf{Tangent nearly-resonant sets}: Corresponding to the resonance 
    condition~\eqref{eqn: diophantine}, we define:
    \begin{equation}
        \label{eqn: diophantine-nearly}
        \Omega_{0, M}^{\tau} = \bigcup_{k \in \Lambda_{2M} \setminus \{0\}} \left\{ \omega_T \in \Omega \bigg |  |\langle k, \omega_T \rangle | < \frac{1}{|k|^{\tau}} \right\}.
    \end{equation}

    \item \textbf{Single-mode nearly-resonant sets}: Corresponding to the first 
    Melnikov condition~\eqref{eqn: melnikov-1}, we define
    \begin{equation}
        \label{eqn: 1-melnikov}
        \Omega_{1, M}^{\tau} = \bigcup_{k \in \Lambda_{2M}} \bigcup_{j=m+1}^{n} \left\{ \omega_T \in \Omega \bigg |  |\langle k, \omega_T \rangle - \omega_j| < \frac{1}{(|k|+1)^{\tau}} \right\}. 
    \end{equation}

    \item \textbf{Difference nearly-resonant sets}: Corresponding to the second 
    Melnikov condition~\eqref{eqn: melnikov-2-1}, we define:
    \begin{equation}
        \label{eqn: 2-melnikov}
        \Omega_{2, M}^{\tau,+} = \bigcup_{k \in \Lambda_{2M}} \bigcup_{\substack{j_1,j_2=m+1  \\ j_1 \neq j_2}}^{n}  \left\{ \omega_T \in \Omega \bigg| |\langle k, \omega_T \rangle - \omega_{j_1} + \omega_{j_2}| < \frac{1}{(|k|+2)^{\tau}} \right\}.
    \end{equation}
    In this scheme, we do not need to consider the second case of the second 
    Melnikov condition~\eqref{eqn: melnikov-2-2} typically found in the classical 
    KAM theorem.
\end{itemize}

For the total nearly-resonant set, 
$\Omega_{M}^{\tau} = \Omega_{0, M}^{\tau} \cup \Omega_{1, M}^{\tau} \cup \Omega_{2,M}^{\tau,+}$, 
we establish the following measure estimate, showing that the set of ``bad'' 
frequencies only occupies a negligible portion of the domain.

\begin{lemma}
    \label{lem: near-mes-initial}
    Let $\tau > m - 1$ be fixed. 
    There exists a constant $\kappa = \kappa(n, m, \tau, \Omega) > 0$, depending on 
    the dimension $n$, the tangent frequency dimension $m$, the parameter $\tau$, 
    and the domain $\Omega$, such that the nearly-resonant set satisfies
    \begin{equation}
        \label{eqn: near-mes-initial}
        \mathrm{mes}(\Omega_{M}^{\tau}) \leq \kappa. 
    \end{equation}
\end{lemma}

The proof is based on elementary measure estimates and is deferred 
to~\Cref{sec: near-mes-initial}.

%======================================================%
\subsection{Restricted operators on small boxes}
\label{subsec: restricted-small}

In the small-scale regime, for tangent frequencies within the non-resonant set 
$\Omega \setminus \Omega_{M}^{\tau}$, we establish the invertibility and decay 
properties of the operator $L$ restricted to central boxes. 
The following theorem verifies the implementation conditions for the restricted 
operator.

\begin{theorem}
    \label{thm: small-size}
    Suppose $\hat{z} \in \mathcal{K}(s)$ and $\omega_T \in \Omega \setminus \Omega_{M}^{\tau}$. 
    Then the restricted operator $L_N$ is invertible and satisfies the operator norm 
    bound
    \begin{equation}
        \label{eqn: small-size-inverse}
        \| L_{N}^{-1} \|_2 \leq \frac{1}{\varepsilon_N}.
    \end{equation}
    Moreover, for any $k \neq k'$, its matrix entries satisfy the Gevrey decay 
    estimate:
    \begin{equation}
        \label{eqn: small-size-localization}
        \left| L_{N}^{-1} \left( k, k' \right) \right| \leq \exp \left\{ -| k - k'|^{s} \right\} \leq \exp\left\{ - \frac{ |k - k'|^{s}}{2} \right\}.
    \end{equation} 
\end{theorem}

By~\Cref{prop: frequency-drift}, the diagonal entries of $D_{N}$ admit the uniform 
lower bound as
\begin{align}
    | D_{ N; j, k} | 
    & \geq \left| - \langle k, \omega_T \rangle + \omega_j \right| - \left| \langle k, \omega'_T - \omega_T \rangle \right| \nonumber  \\
    & \geq \frac{1}{m^{\tau}N^{\tau}} - N \exp\left\{ - (\log M)^{20} \right\} \geq \frac{1}{2m^{\tau}N^{\tau}},   \label{eqn: uniform-lower-original-small}
\end{align}
for any $k \in \Lambda_{N} \setminus \{e_j\}$ when $j=1,\ldots,m$, and any 
$k \in \Lambda_{N} $ when $j=m+1,\ldots,n$, which ensures that $D_N$ remains 
uniformly bounded away from zero. 
Consequently, $(T + \varepsilon B)_N^{-1}$ can be estimated via a Neumann series. 
Since the off-diagonal terms involve only convolution operations preserving Gevrey 
decay, the proof follows from standard arguments and is deferred 
to~\Cref{sec: small-size-original}.

%================================================================================%
\subsection{Verification of the implementation conditions}
\label{subsec: verification}

In this section, we demonstrate that the central restricted operator satisfies the 
implementation conditions (\Cref{cond: implementation}) within the large-scale 
regime $N > M$. 
We employ an inductive approach, utilizing~\Cref{thm: small-size} as the inductive 
base for the small-scale regime $M_0 \leq N \leq M$. Assume that for every scale $N$ 
within the range $M_0 \leq N \leq M^r$, the central restricted operator satisfies 
the implementation conditions. 
Our objective is to prove that these conditions remain valid for the subsequent 
scales $M^{r} < N \leq M^{r+1}$. 
This inductive step relies on the analytical properties of the operator when 
restricted to the outer lattice boxes $k_0 + \Lambda_N$, where 
$k_0 \notin \Lambda_{2N}$. 
For brevity, we denote the restriction of the operator $L$ to $k_0+\Lambda_N$ as 
$L_{k_0,N}$. 
The following theorem establishes its invertibility and off-diagonal decay 
properties to complete the induction. 
The proof, which leverages a multi-scale analysis framework, is detailed 
in~\Cref{sec: multi-scale}.

\begin{theorem}
\label{thm: off-side-box}
    Suppose $\hat{z} \in \mathcal{K}(s)$ and 
    $\omega_T \in \Omega \setminus (\Omega_{M}^{\tau} \cup G)$, where $G$ will be 
    clarified in~\Cref{sec: multi-scale}. 
    Then for any $k_0 \notin \Lambda_{2N}$, the restricted operator $L_{k_0,N}$ is 
    invertible and satisfies: 
    \begin{equation}
        \label{eqn: off-side-inverse}
        \| L_{k_0, N}^{-1} \|_2 \leq \frac{1}{\varepsilon_N}.
    \end{equation}
    Moreover, for any $k, k' \in  k_0 + \Lambda_N$ such that $|k - k'| \geq N^{1/2}$, 
    the matrix entries satisfy:
    \begin{equation}
        \label{eqn: off-side-localization}
        \left| L_{k_0, N}^{-1} \left( k, k' \right) \right| \leq \exp\left\{ - \frac{ |k - k'|^{s}}{2} \right\}.
    \end{equation} 
\end{theorem}

To establish the implementation conditions at scale $N$, we set $K = N^{1/10}$ and 
consider a central box $\Lambda_{10K}$. 
By the inductive hypothesis, the operator $L$ restricted to the box $\Lambda_{10K}$ 
satisfies the implementation conditions. 
We then introduce a subbox $\Lambda_{9K} \subseteq \Lambda_{10K}$ and a family of 
translated boxes $\{ k + \Lambda_{K} \}$ where $k \notin \Lambda_{9K}$. 
Consequently, the lattice box $\Lambda_{N}$ admits the following decomposition:
\begin{equation}
    \label{eqn: center-gluing-9k-implementation}
    \Lambda_{N} = \Lambda_{10K} \cup \left( \bigcup_{k \notin \Lambda_{9K}} (k + \Lambda_{K})\right),
\end{equation}
where the properties established in~\Cref{thm: off-side-box} are applicable to the 
restricted operators $L_{k, K}$. 
We now demonstrate how to derive the global inverse $L_{N}^{-1}$ using local 
information, specifically, the inverses of $L$ restricted to $\Lambda_{10K}$ and the 
family $\{ k + \Lambda_{K}\}$.

\paragraph{Central subbox ($k \in \Lambda_{9K}$)} For lattice points within the 
central subbox, we partition the lattice as 
$\Lambda_{N} = \Lambda_{10K} \cup \Lambda_{10K} ^{c} $, where 
$\Lambda_{10K} ^{c} = \Lambda_{N} \setminus  \Lambda_{10K} $. 
Applying the resolvent identity, we obtain the block representation:
\begin{equation}
    \label{eqn: resolvent-center}
    \begin{pmatrix} 
        L^{-1}_{10K } & 0  \\ 
        0 & L^{-1}_{10K,-} 
    \end{pmatrix} 
    = \left[ E + \varepsilon 
    \begin{pmatrix} 
        0 & L_{10K}^{-1}P^*  \\  
        L^{-1}_{10K,-}P & 0 
    \end{pmatrix}  
    \right] 
    \begin{pmatrix} 
        (L^{-1}_{N })_{\Lambda_{10K} \times \Lambda_{10K}} & (L^{-1}_{N })_{\Lambda_{10K} \times \Lambda_{10K}^{c}}  \\ 
        (L^{-1}_{N })_{\Lambda_{10K}^{c} \times \Lambda_{10K}} & (L^{-1}_{N })_{\Lambda_{10K}^c \times \Lambda_{10K}^c} 
    \end{pmatrix}, 
\end{equation}
where $P$ and $P^*$ are coupling operators generated from $B+S$, and 
$L^{-1}_{10K,-}$ denotes the inverse of the operator $L$ restricted to 
$\Lambda_{10K} ^{c}$. 
From the relation, the components are expressed as:   
\begin{subequations}
    \label{eqn: resolvent-center-relation}
    \begin{empheq}[left=\empheqlbrace]{align} 
        & (L^{-1}_{N })_{\Lambda_{10K} \times \Lambda_{10K}} = L^{-1}_{10K } - \varepsilon L_{10K}^{-1}P^* (L^{-1}_{N })_{\Lambda_{10K}^{c} \times \Lambda_{10K}},           \label{eqn: resolvent-center-relation-1} \\
        & (L^{-1}_{N })_{\Lambda_{10K} \times \Lambda_{10K}^{c}} = - \varepsilon  L_{10K}^{-1}P^*(L^{-1}_{N })_{\Lambda_{10K}^c \times \Lambda_{10K}^c}.                     \label{eqn: resolvent-center-relation-2}
    \end{empheq}        
\end{subequations}
By gluing~\eqref{eqn: resolvent-center-relation-1} 
and~\eqref{eqn: resolvent-center-relation-2} and extending the terms to the full 
index set $\Lambda_N$, we derive
\begin{equation}
    \label{eqn: resolvent-center-final}
    (L^{-1}_{N })_{\Lambda_{10K} \times \Lambda_{N}} = \begin{pmatrix} L_{10K}^{-1}& 0 \end{pmatrix}  - \varepsilon  \begin{pmatrix}  0  & (L_{10K}^{-1}P^*) \end{pmatrix} L_{N}^{-1}.
\end{equation}
Utilizing the inductive assumption alongside~\Cref{prop: convolution} 
and~\Cref{prop: B-operator}, we establish upper bounds for the operator 
$L^{-1}_{10}P^{*}$. 
For any $k \in \Lambda_{9K}$ and $k' \notin \Lambda_{10K}$, the distance satisfies 
$|k - k'| \geq K$. 
Consequently, the entries of $(L^{-1}_{10}P^{*})(k,k')$ exhibits Gevrey decay. 
Specifically, by decomposing the region into short-range and long-range domains, the 
entries satisfy:
\begin{equation}
    \label{eqn: inversion-localization-center}
    | (L^{-1}_{10K}P^*)(k, k')| \leq \left\{ 
    \begin{aligned}
        & \exp\left\{ - \frac{ |k - k'|^{s}}{2} \right\},    &&\mathrm{if}\; |k' - k| < 100K,  \\
        & \exp\left\{ - \frac{ 3|k - k'|^{s}}{4} \right\},   &&\mathrm{if}\; |k' - k| \geq 100K.
    \end{aligned} 
    \right.
\end{equation}
Finally, by restricting the identity~\eqref{eqn: resolvent-center-final} to the 
subset $\Lambda_{9K} \times \Lambda_{N}$, we arrive at:
\begin{equation}
    \label{eqn: resolvent-center-final-9k}
    (L^{-1}_{N })_{\Lambda_{9K} \times \Lambda_{N}} = \begin{pmatrix} L_{10K}^{-1}& 0 \end{pmatrix}_{\Lambda_{9K} \times \Lambda_{N}} - \varepsilon \begin{pmatrix}  0  & (L_{10K}^{-1}P^*) \end{pmatrix}_{\Lambda_{9K} \times \Lambda_{N}} L_{N}^{-1},
\end{equation}
where the entries of $(L_{10K}^{-1}P^*)_{\Lambda_{9K} \times \Lambda_{10K}^{c}}$ 
satisfy the Gevrey decay conditions~\eqref{eqn: inversion-localization-center}.

\paragraph{Outer boxes ($k \notin \Lambda_{9K}$)} For lattice points outside the 
central subbox, we partition the lattice as 
$\Lambda_{N} = (k + \Lambda_{K})\cup (k + \Lambda_{K})^{c} $, where 
$(k + \Lambda_{K})^{c} = \Lambda_{N} \setminus (k + \Lambda_{K}) $. 
Following a block representation and resolvent identity procedure analogous 
to~\eqref{eqn: resolvent-center-relation} ---~\eqref{eqn: resolvent-center-final-9k}, 
we establish
\begin{equation}
    \label{eqn: resolvent-outside-final-9k}
    (L^{-1}_{N })_{\{ k\} \times \Lambda_{N}} = \begin{pmatrix} (L_{k,K}^{-1})_{\{ k\} \times (k + \Lambda_{K})} & 0 \end{pmatrix} - \varepsilon \begin{pmatrix}  0  & (L_{k,K}^{-1}P^*)_{\{ k\} \times (k + \Lambda_{K})^c} \end{pmatrix} L_{N}^{-1},
\end{equation}
where the entries of $(L_{k, K}^{-1}P^*)_{ {\{ k\} \times (k + \Lambda_{K})^c}}$ 
satisfy the Gevrey decay estimate:
\begin{equation}
    \label{eqn: inversion-localization-outside}
    | (L^{-1}_{10K}P^*)(k, k')| \leq \left\{ 
    \begin{aligned}
        & \exp\left\{ - \frac{ |k - k'|^{s}}{2} \right\},     &&\mathrm{if}\; |k' - k| < 4K,  \\
        & \exp\left\{ - \frac{ 3|k - k'|^{s}}{4} \right\},   &&\mathrm{if}\; |k' -k| \geq 4K.
    \end{aligned} 
    \right.
\end{equation}

By combining~\eqref{eqn: resolvent-center-final-9k} 
and~\eqref{eqn: resolvent-outside-final-9k}, the global inverse can be expressed as
\begin{equation}
    \label{eqn: resolvent-relation}
    L_{N}^{-1} = \Psi_N  - \varepsilon \Phi _{N} L_{N}^{-1},
\end{equation}
where the operator $\Psi_N$ is constructed from the lower-scale inverses, 
$L_{10K}^{-1}$ and $L^{-1}_{k, K}$, as follows: 
\begin{equation}
    \label{eqn: entries-psi}
    \Psi_{N}(k, k') = \left\{ 
    \begin{aligned}
        & L_{10K}^{-1}(k, k'),  &&\mathrm{if}\; k \in \Lambda_{9K},\; k' \in \Lambda_{10K},        \\
        & L_{k, K}^{-1}(k, k'), &&\mathrm{if}\; k \notin \Lambda_{9K},\; k' -k \in \Lambda_K,    \\
        & 0,                    && \mathrm{otherwise}.   
    \end{aligned} 
    \right.
\end{equation}
Based on the estimates provided in~\eqref{eqn: inversion-localization-center} 
and~\eqref{eqn: inversion-localization-outside}, the entries of the operator 
$\Phi_N$ satisfy:
\begin{equation}
    \label{eqn: entries-phi}
    \Phi_N(k, k') \leq \left\{ 
    \begin{aligned} 
        & 0,                                                 &&\mathrm{if}\; k =  k'; \\
        & \exp\left\{ - \frac{ |k - k'|^{s}}{2} \right\},    &&\mathrm{if}\; 0 < |k' - k| < 100K;  \\ 
        & \exp\left\{ - \frac{ 3|k - k'|^{s}}{4} \right\},   &&\mathrm{if}\; |k' -k| \geq 100K.
    \end{aligned} 
    \right.
\end{equation}                                    
For a sufficiently small $\varepsilon$, the operator $E + \varepsilon \Phi_N$ is 
invertible via the Neumann series expansion. 
For long-range interactions where $|k-k'| \geq N^{1/4}$, the inverse satisfies:
\begin{equation}
    \label{eqn: inverse-decay-long}
    | (E + \varepsilon \Phi_N)^{-1}(k, k') | \leq \exp\left\{ - \frac{|k - k'|^s}{2} - \frac{1}{4} (100K)^s \right\}, 
\end{equation}
where this bound accounts for interaction paths containing at least one long-range 
jump. If a path consisted solely of short-range interactions to cover the distance 
$|k-k'|$, the ``cost'' (the cumulative decay exponent) would be even higher. 
Specifically, according to the concavity property for $s < 1$, the total index must 
satisfy:
\[
    \frac{|k - k'|}{100K} (100K)^{s} \geq \left( \frac{N^{\frac14}}{(100K)} \right)^{1-s} \cdot |k-k'|^{s}  \geq |k-k'|^s + \frac{1}{2}(100K)^s
\]
A more refined analysis for the Gevrey case reveals that the inverse 
$(E + \varepsilon \Phi_N)^{-1}(k, k')$ preserves short-range decay, while the 
long-range decay follows~\eqref{eqn: entries-phi} with the transition threshold 
shifted from $100K$ to $K^2$. 
In contrast, the concavity property is unavailable in the analytic case. 
Consequently, the parameter $\varepsilon$ must be more strictly tuned to compensate 
for the lack of sub-linear exponent growth and to ensure sufficient decay for purely 
short-range interaction paths. 
We can thus rewrite~\eqref{eqn: resolvent-relation} as
\begin{equation}
    \label{eqn: l-n-inverse}
    L_{N}^{-1} = (E + \varepsilon \Phi_{N})^{-1} \Psi_N. 
\end{equation}

Finally, by~\eqref{eqn: entries-psi}, the $\ell_2$-norm satisfies 
$\|\Psi_N\|_2 \leq (20K)^m \exp\left\{ (\log(10K))^{15} \right\}$.
By substituting this bound together with~\eqref{eqn: inverse-decay-long} 
into~\eqref{eqn: l-n-inverse}, we conclude that $L_N^{-1}$ satisfies the 
implementation conditions at scale $N$.

\section{Multi-scale analysis}
\label{sec: multi-scale}

In this section, we employ a multi-scale analysis to 
establish~\Cref{thm: off-side-box} for the ``outer'' boxes $k_0 + \Lambda_N$, where 
the center satisfies $k_0 \notin \Lambda_{2N}$. 
For any indices $k, k' \in k_0 + \Lambda_N$, the condition $k_0 \notin \Lambda_{2N}$ 
implies a lower bound of $|k| + |k'| \geq 2N$. 
Consequently, the restricted operator satisfies the following norm estimate:
\[
    \|B_{k_0, N}\|_2 \leq \sqrt{ \|B_{k_0, N}\|_1 \cdot \|B_{k_0, N}\|_\infty} \leq (2N+1)^{m} \exp\left\{ -N^{s} \right\}.
\]
Given that $T_{k_0,N} = D_{k_0,N} + \varepsilon S_{k_0,N}$ is symmetric with respect 
to the indices $k$ and $k'$, we analyze the numerical operator 
$L_{k_0,N} = T_{k_0,N} + \varepsilon B_{k_0,N}$. 
When the smallest eigenvalue of $T_{k_0, N}$ scales as $\exp(-(\log N)^s)$ for 
$s > 0$, we examine the product
\[
    L_{k_0, N}^{\top}L_{k_0, N} = T^{2}_{k_0, N} + \varepsilon (B_{k_0, N}^{\top}T_{k_0, N} + T_{k_0, N}B_{k_0, N}) + \varepsilon^2 B_{k_0, N}^{\top}B_{k_0, N}.
\]
Due to the rapid decay of $B_{k_0,N}$ established above, the perturbation terms are 
negligible compared to the leading term $T_{k_0,N}^2$. 
Therefore, obtaining a robust bound for $\| T_{k_0,N}^{-1} \|_2 $ is sufficient to 
control the inverse norm $\| L_{k_0,N}^{-1} \|_2$.

For the remainder of this section, we focus on estimating $T_{k_0,N}^{-1}$. 
Since the specific location of the center $k_0 \notin \Lambda_{2N}$ does not 
explicitly affect the subsequent analysis, we omit the subscript $k_0$ unless 
otherwise specified and simplify the notation by setting 
$\Gamma_N = k_0 + \Lambda_N$, $T_{N} := T_{k_0, N}$, and $D_{N} := D_{k_0, N}$.

%============================================================================%
\subsection{Inductive base for small boxes}
\label{subsec: small-away}

We begin the multi-scale analysis by considering the small-scale regime 
$M_0 \leq N \leq M$. 
According to~\Cref{prop: frequency-drift}, the gap between any two eigenvalues of 
$D_{N}$ satisfies the following lower bound: 
\begin{equation}
    \label{eqn: frequency-difference-small-multi}
    \left| D_{N; j_1, k_1} - D_{N; j_2, k_2} \right| = \left| - \langle k_1 - k_2, \omega_{T}' \rangle + \omega_{j_1} -\omega_{j_2} \right| \geq \left| - \langle k_1 - k_2, \omega_{T} \rangle + \omega_{j_1} -\omega_{j_2} \right| - 2N\varepsilon.
\end{equation}
For any tangent frequency $\omega_{T} \in \Omega \setminus \Omega_{M}^{\tau}$, we 
exploit the fact that $\varepsilon$ is ``super-polynomially'' small, specifically, 
it decays faster than any polynomial in $N$. 
This leads to several key observations regarding the resonance conditions. 
First, when the normal frequency difference is not zero, the definitions of the 
single-mode resonant sets~\eqref{eqn: 1-melnikov} and the difference resonant 
sets~\eqref{eqn: 2-melnikov} imply that the gap condition can only fail if 
$j_1 = j_2$ and $k_1 = k_2$. 
Hence, for distinct eigenvalues, the normal frequency components must vanish. 
Furthermore, the tangent resonant set defined in~\eqref{eqn: diophantine-nearly} 
ensures that the gap condition is violated by 
$(k_{1} - e_{j_1}) - (k_{2} - e_{j_2}) = 0$. 
Given that $j=1, \ldots, m$, there exists at most $m$ entries of $D_n$ falling below 
the threshold $(4mN)^{-\tau}$, which is formalized in the following proposition.

\begin{prop} 
    \label{prop: small-eig-one}
    Suppose that $\hat{z} \in \mathcal{K}(s)$ and the tangent frequency satisfies 
    $\omega_T \in \Omega \setminus \Omega_{M}^{\tau}$. 
    Then the diagonal operator $D_{N}$ admits at most $m$ entries
    $D_{N; j, k} = - \langle k, \omega_T' \rangle + \omega_j $ for $j=1,\ldots, n$, 
    whose absolute value is smaller than $(4mN)^{-\tau}$.
\end{prop}

For entries falling below the threshold $(4mN)^{-\tau}$, the indices 
$k_{j} - e_{j}$ for $j=1, \ldots, m$ remain fixed. 
For convenience, we denote this fixed index as $k_{\star} = k_{j} - e_j$. 
We then define the set of these entries as
\begin{equation}
    \label{eqn: singular-small}
    \Xi= \left\{ - \langle k_{\star}, \omega_T' \rangle + \omega_j - \omega_j' \big |  j = 1, \ldots, m \right\}. 
\end{equation}
This index set identifies the singular site, a singleton set denoted as 
$\Pi = \{ k_{\star} \}$. 
To establish the inductive base in the small-scale regime, we formulate the 
following two lemmas. 
First, we consider the operator $T$ restricted to the index set excluding the 
singular site, $\Gamma_{N} \setminus \Pi$. 
By applying~\Cref{prop: small-eig-one} and following a procedure similar to the 
proof of~\Cref{thm: small-size}, we derive the following bounds for its inverse.

\begin{lemma}
    \label{lem: small-singular}
    Suppose that $\hat{z} \in \mathcal{K}(s)$ and 
    $\omega_T \in \Omega \setminus \Omega_{M}^{\tau}$. 
    The restricted operator $T_{\Gamma_N \setminus \Pi}$ is invertible, and its 
    inverse satisfies the operator norm bound
    \begin{equation}
        \label{eqn: small-size-inverse-prepare}
        \| T_{\Gamma_N \setminus \Pi}^{-1} \|_2 \leq \frac{1}{\varepsilon_N}.
    \end{equation}
    In addition, the off-diagonal entries of $T _{\Gamma_N \setminus \Pi}^{-1}$ 
    satisfy the decay estimate
    \begin{equation}
        \label{eqn: small-size-off-diagonal-prepare}
        \left| T_{\Gamma_N \setminus \Pi}^{-1} \left( k, k' \right) \right| \leq \exp \left\{ - \frac{| k - k'|^{s}}{2} \right\},
    \end{equation}
    for any $k \neq  k' \in \Gamma_N \setminus \Pi$.
\end{lemma}

Next, we analyze $T_{N}^{-1}$ by leveraging the properties of 
$T^{-1}_{\Gamma_N \setminus \Pi }$. 
According to the resolvent identity, we express the block structure of the inverse as 
follows:
\begin{equation}
    \label{eqn: resolvent-small-inversion-local}
    \begin{pmatrix} 
        T^{-1}_{\Gamma_N \setminus \Pi } & 0  \\ 
        0 & T^{-1}_{\Pi} 
    \end{pmatrix} 
    = \left[ E + \varepsilon 
    \begin{pmatrix} 
        0 & T^{-1}_{\Gamma_N \setminus \Pi }P^*  \\  
        T^{-1}_{\Pi} P & 0 
    \end{pmatrix}  
    \right] 
    \begin{pmatrix} 
        (T^{-1}_{N })_{(\Gamma_N \setminus \Pi) \times (\Gamma_N \setminus \Pi)} & (T^{-1}_{N})_{(\Gamma_N \setminus \Pi) \times \Pi}  \\ 
        (T^{-1}_{N })_{\Pi \times (\Gamma_N \setminus \Pi) } \mathrel{\phantom{=l}} & (T^{-1}_{N})_{\Pi \times \Pi} \mathrel{\phantom{=l}}
    \end{pmatrix}. 
\end{equation}
Since $T_{N}$ is symmetric with respect to the indices $k$ and $k'$, the 
identity~\eqref{eqn: resolvent-small-inversion-local} yields the relation:
\begin{equation}
    \label{eqn: inversion-localization}
    (T^{-1}_{N })_{\Pi \times (\Gamma_N \setminus \Pi)}^{\top} = (T^{-1}_{N})_{(\Gamma_N \setminus \Pi) \times \Pi} = - \varepsilon T^{-1}_{\Gamma_N \setminus \Pi }P^*  (T^{-1}_{N })_{\Pi \times \Pi}.
\end{equation}
If the inverse is bounded, specifically, 
$\| T_{N}^{-1} \|_2 \leq \varepsilon_{N}^{-1}$, the 
relation~\eqref{eqn: inversion-localization} implies the localization condition. 
We thus establish that, within the small-scale regime, the existence of an inverse 
with exponential-logarithmic growth implies localization.

\begin{lemma}
    \label{lem: small-inverse-localization}
    Suppose that $\hat{z} \in \mathcal{K}(s)$ and 
    $\omega_T \in \Omega \setminus \Omega_{M}^{\tau}$. 
    If the restricted operator $T_N$ is invertible and satisfies the operator norm 
    bound
    \begin{equation}
        \label{eqn: trans-small-size-original-inverse}
        \| T_{N}^{-1} \|_2 \leq \frac{1}{\varepsilon_N},
    \end{equation}
    then the off-diagonal entries of $T_{N}^{-1}$ satisfy the decay estimate
    \begin{equation}
        \label{eqn: trans-small-size-off-diagonal}
        \left| T_{N}^{-1} \left( k, k' \right) \right| \leq \exp \left\{ - \frac{| k - k'|^{s}}{2} \right\},
    \end{equation}
    for any $k \neq k' \in \Gamma_{N}$.
\end{lemma}

In the small-scale regime, the growth of $\varepsilon_N^{-1}$ is suppressed by the 
small parameter $\varepsilon$. 
In the large-scale regime, however, this mechanism is no longer sufficient; the 
growth of $\varepsilon_N^{-1}$ is instead counteracted by the Gevrey decay of the 
interaction terms originating from $P^*$.

%===================================================%
\subsection{Size reduction: clustering of singular sites}
\label{sec: size-reduce-cluster}

In this section, we analyze the eigenvalues of the restricted operator $T_{N}$, which 
are given by 
\begin{equation}
    \label{eqn: eigenvalue-t-n}
    T_{N; j, k} = - \langle k, \omega'_T \rangle + \omega_j + \mu_{j,k}(\omega'_T). 
\end{equation}
Based on the results established in~\Cref{prop: convolution}, the perturbation term 
$\mu_{j, k}$ is bounded. 
Specifically, it satisfies the following gradient estimate:
\begin{equation}
    \label{eqn: perturb-omega-t}
    \max_{j=1,\ldots, n}\max_{k \in \mathbb{Z}^m} \left\| \frac{\partial \mu_{j,k}}{\partial \omega'_T} \right\|_2 \leq \varepsilon \left\| \frac{\partial S}{\partial \omega'_T} \right\|_2  \leq \frac{\varepsilon}{2}. 
\end{equation}
Recall from~\Cref{prop: frequency-drift} that the mapping between the tangent 
frequencies $\omega_T$ and $\omega'_T$ is a diffeomorphism. 
Given that the normal frequencies $\omega_j$ (for $j = m+1, \dots, n$) are 
differentiable with respect to $\omega'_T$, we conclude there exists a constant 
$K > 0$ such that:
\begin{equation}
    \label{eqn: frequency-omega-t}
    \max_{j=1,\ldots, n} \left\| \frac{\partial \omega_{j}}{\partial \omega'_T} \right\|_2 \leq K.
\end{equation}
We now examine the difference between any two eigenvalues of $T_N$:  
\begin{equation}
    \label{eqn: eigenvalue-difference}
    T_{N; j_1, k_1} - T_{N; j_2, k_2} = - \langle k, \omega_T' \rangle + \omega_{j_1} - \omega_{j_2} + \mu_{j_1,k_1}(\omega'_T) - \mu_{j_2,k_2}(\omega'_T)
\end{equation}
where $k = k_1 - k_2$. 
To resolve the small divisor problem, we exclude the resonant regions along the 
direction $k/\|k\|_2$. 
Using the bounds established in~\eqref{eqn: perturb-omega-t} 
and~\eqref{eqn: frequency-omega-t}, we observe that when $\|k\|_2$ is sufficiently 
large, the same excision procedure as in the unperturbed case applies. 
When the perturbation vanishes and the drifted tangent frequency $\omega_{T}'$ 
degenerates to $\omega_T$, the curved resonance regions ``straighten out'' into the 
nearly-resonant sets defined 
in~\eqref{eqn: diophantine-nearly},~\eqref{eqn: 1-melnikov}, 
and~\eqref{eqn: 2-melnikov}, which is shown in~\Cref{fig: melnikov_comparison}.

\begin{figure}[htb!]
    \centering
    \begin{subfigure}[t]{0.46\textwidth}
        \centering
        \begin{tikzpicture}[scale=1.6]
            \useasboundingbox (0, -0.8) rectangle (4, 2.3);
            
            \fill[pattern=north east lines, pattern color=gray!200] (0,0) rectangle (4,1);
            
            \draw[thick] (0,0) -- (4,0);
            \draw[thick] (0,1) -- (4,1);
            
            \draw[->, thick] (4.5, -0.3) -- (4.5, 1.5) node[below right] {$\frac{k}{\|k\|_2}$};
        \end{tikzpicture}
        \caption{Unperturbed resonance strip}
        \label{fig: melnikov_standard}
    \end{subfigure}
    \hspace{0.05\textwidth}%  <-- fixed space instead of \hfill
    \begin{subfigure}[t]{0.46\textwidth}
        \centering
        \begin{tikzpicture}[scale=1.6]
            \useasboundingbox (0, -0.8) rectangle (4, 2.3);
            
            % 内层波浪边界（虚线，加粗）
            \def\bottomwave{(0,0) .. controls (1.2,0.4) and (2.8,-0.1) .. (4,0.1)}
            \def\topwave{(0,1) .. controls (1.5,1.4) and (2.5,1.0) .. (4,1.2)}
            
            % 先填充外层实线之间的整个区域（斜线）
            \fill[pattern=north east lines, pattern color=black!60] (0, -0.3) rectangle (4, 1.5);
            
            % 再在虚线内部填充深色交叉线（先白色覆盖，再深色交叉线）
            \begin{scope}
                \clip (0, -0.3) rectangle (4, 1.5);
                \fill[pattern=north east lines, pattern color=gray!200] 
                    \bottomwave -- (4,1.2) .. controls (2.5,1.0) and (1.5,1.4) .. (0,1) -- cycle;
            \end{scope}
            
            % 绘制外层平直边界（实线）
            \draw[thick] (0, -0.3) -- (4, -0.3);
            \draw[thick] (0, 1.5) -- (4, 1.5);
            
            % 绘制内层波浪边界（虚线，加粗）
            \draw[very thick, dashed] \bottomwave;
            \draw[very thick, dashed] \topwave;
            
            % 垂直向上箭头及标注
            \draw[->, thick] (4.5, -0.3) -- (4.5, 1.5) node[below right] {$\frac{k}{\|k\|_2}$};
        \end{tikzpicture}
        \caption{Perturbed resonance region}
        \label{fig: melnikov_perturbed}
    \end{subfigure}
    
    \caption{Geometric illustration of resonance regions. In (b), the solid lines denote the boundaries of the excluded strip, while the dashed curves delineate the perturbed resonant region arising from the variation of the term $\mu_{j,k}$.}
    \label{fig: melnikov_comparison}
\end{figure}
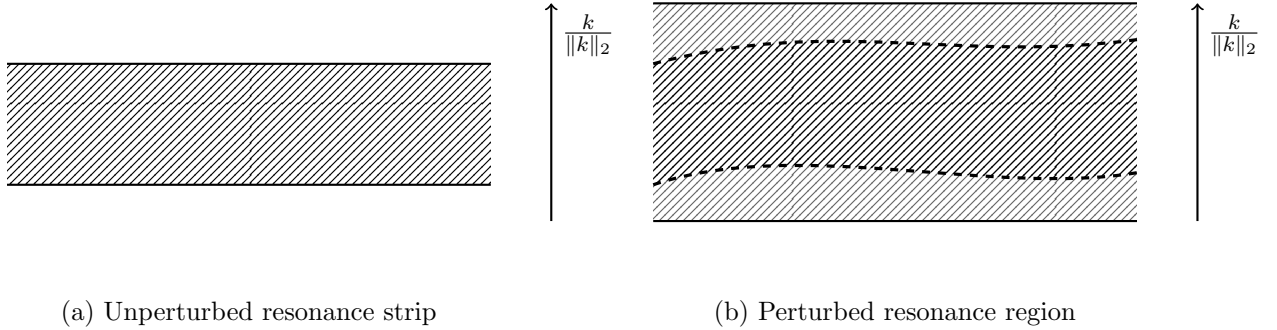

Following the intuitive description above, we now introduce the rigorous scaling 
parameters and nearly-resonance sets required for the large-scale regime where $N>M$. 
We begin by defining a reduced scale $N'$, derived from the primary scale $N$ via 
the following polylogarithmic relation:
\begin{equation}
    \label{eqn: 10-log-size}
    N' = \exp\left\{ (\log N)^{\frac{1}{10}} \right\}. 
\end{equation}
We next define the difference nearly-resonant set for the large-scale regime as the 
union of its components across all scales $N > M$.
\begin{defn}
    \label{defn: second-melnikov}
    The difference nearly-resonant set in the large-scale regime is given by
    \[
        G_{2} = \bigcup_{N>M} G_{2, N} 
    \]
    where each component set $G_{2,N}$ is defined as:
    \begin{equation}
        \label{eqn: nearly-resonant-multiscale}
        G_{2, N} = \bigcup_{j_1,j_2=1}^{n} \bigcup_{k \in \Lambda_{2N} \setminus \Lambda_{2N'}} \left\{ \omega_T \in \Omega \bigg| \left| - \langle k, \omega_T' \rangle + \omega_{j_1} - \omega_{j_2} + \mu_{j_1,k_1} - \mu_{j_2,k_2} \right| < 2\varepsilon_{N'} \right\}.
    \end{equation}
\end{defn}

Based on the scaling relation in~\eqref{eqn: 10-log-size}, the total measure 
satisfies
\[
    \mathrm{mes}(G_2) = \sum_{N > M} (2N+1)^{m} \exp\left\{ - (\log N)^{\frac32}\right\}  < \infty.  
\]
In particular, as $\varepsilon \rightarrow 0$, it follows that $M \rightarrow \infty$, 
which implies $\mathrm{mes}(G_2) \rightarrow 0$.

Consider the operator $T$ restricted to two boxes, $k_{0,1} + \Lambda_{N'}$ and 
$k_{0,2} + \Lambda_{N'}$, centered at $k_{0,1}, k_{0,2} \in \Gamma_N$. 
If the two centers satisfy the separation condition $|k_{0,1} - k_{0,2}| \geq 4N'$, then at 
least one of the restricted operators $T_{k_{0,1}, N'}$ and $T_{k_{0,2}, N'}$ has its 
smallest eigenvalue (in absolute value) bounded below by $\varepsilon_{N'}$. 
Equivalently, at most one of the resolvent norms, $\lVert T_{k_{0,1}, N'}^{-1} \rVert_2$ or 
$\lVert T_{k_{0,2}, N'}^{-1} \rVert_2$, can exceed $\varepsilon_{N'}^{-1}$. 
To formalize this spatial clustering, we define the set of ``singular'' centers 
within the domain $\Gamma_N$ as
\begin{equation}
    \label{eqn: eqn-violated-center}
    \Sigma(N, N') = \left\{ k \in \Gamma_N \bigg| \| T_{k, N'}^{-1} \|_2 > \frac{1}{\varepsilon_{N'}} \right\}.
\end{equation}
The spatial distribution of these singular centers is characterized by the following 
lemma.

\begin{lemma}
    \label{lem: second-melnikov}
    Suppose that $\hat{z} \in \mathcal{K}(s)$ and 
    $\omega_T \in \Omega \setminus (\Omega_{M}^{\tau} \cup G_{2})$. 
    If two box centers satisfy $k_{0,1}, k_{0,2} \in \Sigma(N, N') $, then they must 
    be clustered such that:   
    \[
        k_{0,1} - k_{0,2} \in  \Lambda_{4N'}.
    \]    
\end{lemma}

%===================================================%
\subsection{Multi-scale induction: inversion implies localization}
\label{sec: inversion-localization}

Based on~\Cref{lem: second-melnikov}, we establish the inclusion 
$\Sigma(N, N') \subseteq \Gamma_{4N'}$, which is subsequently contained within 
$\Gamma_{10N'}$. 
For the purpose of our decomposition, we designate $\Gamma_{10N'}$ as the singular 
box. Accordingly, any box whose center lies outside $\Sigma(N, N')$ is referred to 
as a regular box. 
With this classification, we decompose the box $\Gamma_N$ into a singular box 
together with a collection of regular boxes: 
\begin{equation}
    \label{eqn: decomposition-lattice}
    \Gamma_{N} = \Gamma_{10N'} \cup \left( \bigcup_{k \notin \Gamma_{9N'}} (k + \Lambda_{N'}) \right).
\end{equation}
where the choice of size $10N'$ rather than $4N'$ is deliberate. 
Indeed, any box $ k+ \Lambda_{N'}$ centered at 
$k \in \Gamma_{10N'} \setminus \Gamma_{9N'}$ is still treated as regular. 
This enlargement facilitates the subsequent gluing procedure, consistent with the 
strategy detailed in~\Cref{subsec: verification}.

We repeat the scale reduction process for each box in the 
decomposition~\eqref{eqn: decomposition-lattice}. 
This iteration continues through a sequence of scales satisfying the hierarchy 
in~\eqref{eqn: 10-log-size} until they reach the small-scale regime $[M_0, M]$. 
For the singular box,~\Cref{prop: small-eig-one} ensures that the set of singular 
sites is a singleton, $\Pi = \{k_{\star}\}$. 
Within this small-scale regime, we have established the following results.  
\begin{itemize}
    \item[(i)] For the singular box excluding the singular site 
    $\Pi$,~\Cref{lem: small-singular} guarantees that the inversion 
    condition~\eqref{eqn: small-size-inverse-prepare} and the localization 
    conditions~\eqref{eqn: small-size-off-diagonal-prepare} hold simultaneously.

    \item[(ii)] For the regular boxes,~\Cref{lem: small-inverse-localization}, we 
    establishes that the inversion 
    condition~\eqref{eqn: trans-small-size-original-inverse} directly implies the 
    localization condition~\eqref{eqn: trans-small-size-off-diagonal}.
\end{itemize}
With these small-scale results serving as the inductive bases, we establish that 
inversion implies localization across the large-scale regime $N>M$, rigorously stated 
as the following lemma.

\begin{lemma}
    \label{lem: induction-multi-scale}
    Suppose that $\hat{z} \in \mathcal{K}(s)$ and 
    $\omega_T \in \Omega \setminus (\Omega_{M}^{\tau} \cup G_{2})$. 
    If the restricted operator $T_{N}$ satisfies the inversion 
    condition~\eqref{eqn: trans-small-size-original-inverse} for any 
    $M^{r} < N \leq M^{r+1}$, then the entries of its inverse satisfies the 
    localization condition
    \begin{equation}
        \label{eqn: ri-off-diagonal-final}
        |T_N^{-1}(k_1, k_2)| \leq \exp\left\{- \frac{|k_1- k_2|^s}{2} \right\}
    \end{equation}
    for any $|k_1 - k_2| \geq N^{\frac12}$. 
\end{lemma}

The proof consists of two successive steps: first, a ``gluing'' process following 
the procedure as detailed in~\Cref{subsec: verification}; and second, a 
demonstration that inversion implies localization, which remains entirely consistent 
with~\eqref{eqn: resolvent-small-inversion-local} 
and~\eqref{eqn: inversion-localization}.

\paragraph{The ``gluing'' procedure}
We begin by considering the set $\Gamma_{N} \setminus \Pi $ at a scale satisfying 
$M^r\leq N \leq M^{r+1}$. 
According to~\eqref{eqn: decomposition-lattice}, the set $\Gamma_{N} \setminus \Pi$ 
can be decomposed as:
\begin{equation}
    \label{eqn: decomposition-lattice-singular}
    \Gamma_{N} \setminus \Pi = ( \Gamma_{10N'} \setminus \Pi ) \cup \left( \bigcup_{k \notin \Gamma_{9N'}} (k + \Lambda_{N'}) \right).
\end{equation}
By the inductive hypothesis for $M_0 \leq N \leq M^r$, the following holds:
\begin{itemize}
    \item The operator restricted to the central region 
    $\Gamma_{10N'} \setminus \Pi $, with a scale satisfying 
    $M_0 \leq 4N' \leq M^r$, fulfills the inversion 
    condition~\eqref{eqn: small-size-inverse-prepare} and the localization 
    condition~\eqref{eqn: small-size-off-diagonal-prepare}.

    \item The operator restricted to each outer box $k + \Lambda_{N'}$, where 
    $M_0 \leq N' \leq M^r$, satisfies the property that the inversion 
    condition~\eqref{eqn: trans-small-size-original-inverse} directly implies the 
    localization condition~\eqref{eqn: trans-small-size-off-diagonal}.
\end{itemize}

By applying the same resolvent identity framework developed 
in~\eqref{eqn: resolvent-center} ---~\eqref{eqn: resolvent-center-final-9k}, we can 
derive the following representation for the inverse operator restricted to 
$\Gamma_{9N'} \setminus \Pi \times \Gamma_{N}$ as
\begin{equation}
    \label{eqn: resolvent-10N'}
    \left( T^{-1}_{\Gamma_{10N'} \setminus \Pi } \right)_{(\Gamma_{9N'} \setminus \Pi) \times \Gamma_{N}} = \begin{pmatrix} T_{\Gamma_{10N'} \setminus \Pi }^{-1}& 0 \end{pmatrix}_{(\Gamma_{9N'} \setminus \Pi) \times \Gamma_{N}} - \varepsilon \begin{pmatrix}  0  & \left(T_{\Gamma_{10N'} \setminus \Pi }^{-1}P^*\right) \end{pmatrix}_{(\Gamma_{9N'} \setminus \Pi) \times \Gamma_{N}} T_{N}^{-1},
\end{equation}
where the entries of 
$(T_{\Gamma_{10N'} \setminus \Pi}^{-1}P^*)_{(\Gamma_{9N'} \setminus \Pi) \times (\Gamma_{10N'} \setminus \Pi)^c}$ 
exhibit the Gevrey decay. 
Specifically, these entries satisfy the following decay estimates across the 
short-range and long-range domains:
\begin{equation}
    \label{eqn: msa-gvery-center}
    | (T^{-1}_{\Gamma_{10N'} \setminus \Pi}P^*)(k, k')| \leq \left\{ 
    \begin{aligned}
        & \exp\left\{ - \frac{ |k - k'|^{s}}{2} \right\},    &&\mathrm{if}\; |k' - k| < 100N',  \\
        & \exp\left\{ - \frac{ 3|k - k'|^{s}}{4} \right\},   &&\mathrm{if}\; |k' -k| \geq 100N'.
    \end{aligned} 
    \right.
\end{equation} 
Similarly, for a singleton $\{k\}$ within the outer boxes, the inverse is represented 
as:
\begin{equation}
    \label{eqn: resolvent-outside-N'}
    (T^{-1}_{k, N' })_{\{ k\} \times \Gamma_{N}} = \begin{pmatrix} (T_{k,N'}^{-1})_{\{ k\} \times (k + \Lambda_{N'})} & 0 \end{pmatrix}  - \varepsilon  \begin{pmatrix}  0  & (T_{k,N'}^{-1}P^*)_{\{ k\} \times (k + \Lambda_{N'})^c} \end{pmatrix} T_{N}^{-1},
\end{equation}
where the corresponding Gevrey decay for these entries is: 
\begin{equation}
    \label{eqn: msa-gvery-offside}
    | (T^{-1}_{\Gamma_{10N'} \setminus \Pi }P^*)(k, k')| \leq \left\{ 
    \begin{aligned}
        & \exp\left\{ - \frac{ |k - k'|^{s}}{2} \right\},    &&\mathrm{if}\; |k' - k| < 4N',  \\
        & \exp\left\{ - \frac{ 3|k - k'|^{s}}{4} \right\},   &&\mathrm{if}\; |k' - k| \geq 4N'.
    \end{aligned} 
    \right.
\end{equation}
Following the same procedure established 
in~\eqref{eqn: resolvent-relation} ---~\eqref{eqn: l-n-inverse}, we derive the 
unified representation:
\begin{equation}
    \label{eqn: final-expression-msa}
    T_{\Gamma_{N} \setminus \Pi }^{-1} = (E - \varepsilon \Phi_{N})^{-1} \Psi_N,
\end{equation}
which confirms that 
$\| T_{\Gamma_N \setminus \Pi}^{-1} \|_2 \leq \exp \left\{ (\log 10N')^{15} \right\} < \exp \left\{ (\log N)^{15}\right\}$ 
and that the entries satisfy the Gevrey decay 
condition~\eqref{eqn: small-size-off-diagonal-prepare}. 
Furthermore, $T_{\Gamma_{N} \setminus \Pi}$ satisfies~\Cref{lem: small-singular} at 
the scale $M^{r} < N \leq M^{r+1}$.

\paragraph{Inversion implies localization}
We now establish~\Cref{lem: small-inverse-localization} for the scale 
$M^{r} < N \leq M^{r+1}$. 
At this scale, the relation~\eqref{eqn: inversion-localization} is maintained as 
follows:
\[
    (T^{-1}_{N})_{\Pi \times (\Gamma_N \setminus \Pi)}^{\top} = (T^{-1}_{N})_{(\Gamma_N \setminus \Pi) \times \Pi} = - \varepsilon T^{-1}_{\Gamma_N \setminus \Pi }P^*  (T^{-1}_{N })_{\Pi \times \Pi}
\]
Building upon~\Cref{lem: small-singular} and assuming that the inversion 
condition~\eqref{eqn: trans-small-size-original-inverse} holds, we derive the 
localization condition~\eqref{eqn: trans-small-size-off-diagonal}. 
It is important to note that the parameter $\varepsilon$ is insufficient to suppress 
the growth of $(T_N^{-1})_{\Pi \times \Pi}$ at large scales. 
Instead, the proof relies crucially on the Gevrey decay inherited from the 
off-diagonal operator $P^*$, which effectively counteracts the exponential 
polylogarithmic growth, ensuring that localization properties are preserved across 
the entire domain.

%===================================================================%
\subsection{Measure estimate for ``bad'' frequencies}
\label{subsec: measure-estimate}

The final stage of our analysis involves quantifying the set of ``bad'' frequencies, 
those that must be excluded to ensure the inversion 
condition~\eqref{eqn: trans-small-size-original-inverse} remains valid. 
In this section, we establish the rigorous measure estimates for these resonant sets. 
Recall that the restricted operator $T_{N}$ admits the following block decomposition: 
\[
    T_{N} = 
    \begin{pmatrix} 
        T_{\Gamma_N \setminus\Pi} & \varepsilon P^*  \\ 
        \varepsilon P & T_{\Pi} 
    \end{pmatrix}.
\]
By analyzing this block structure, we identify the Schur complement with respect to 
$T_{\Pi}$ as:
\begin{equation}
    \label{eqn: schur-complement}
    U = T_{\Pi} - \varepsilon^2 P T_{\Gamma_N \setminus\Pi}^{-1}P^*.
\end{equation}
Based on the Schur complement defined in~\eqref{eqn: schur-complement}, we establish 
an upper bound for the $\ell_2$-norm of the inverse operator:
\[
    \| T_N^{-1} \|_2 \leq 4 \| T_{\Gamma_N \setminus \Pi}^{-1} \|_2^2 \cdot \| U^{-1} \|_2 + \| T_{\Gamma_N \setminus \Pi}^{-1} \|_2.
\]
Given the established bound 
$\| T_{\Gamma_N \setminus \Pi}^{-1} \|_2 \leq \exp \left\{ (\log 10N')^{15} \right\}$, 
we derive the following inclusion for the resonant sets:
\begin{equation}
    \label{eqn: inverse-relation-schur-complement}
    \left\{ \omega_T \in \Omega \bigg | \| T_N^{-1} \|_2 > \frac{1}{\varepsilon_N} \right\} \subseteq \left\{ \omega_T \in \Omega \bigg | \| U^{-1} \|_2 > \frac{1}{\sqrt{\varepsilon_N}} \right\}, 
\end{equation}
which implies that analyzing the Schur complement $U$ is sufficient for our measure 
estimates. 
Thus, we further expand the Schur complement~\eqref{eqn: schur-complement} as
\[
    U = \mathrm{diag}\left( - \langle k_{\star}, \omega_{T}' \rangle+ \omega_j \right) + \varepsilon S(k_{\star}, k_{\star}) - \varepsilon^2 P T_{ \Gamma_{N} \setminus \Pi}^{-1} P^*.
\]
By introducing the shift parameter 
$\sigma_{\star} = - \langle k_{\star}, \omega_{T}' \rangle$, we define the residual 
term as:
\[
    \Upsilon = \mathrm{diag}\left( \omega_j \right) + \varepsilon S(k_{\star}, k_{\star}) - \varepsilon^2 P T_{\Gamma_{N} \setminus \Pi}^{-1} P^*
\]
Given that $\| T_{\Gamma_N \setminus \Pi}^{-1} \|_2 \leq \exp \left\{ (\log 10N')^{15} \right\}$, 
it follows that $\| \Upsilon  \|_2 \leq \exp \left\{ (\log 10N')^{15} \right\}$. 
Consequently, we can express the Schur complement from~\eqref{eqn: schur-complement} 
in the simplified form:
\begin{equation}
    \label{eqn: schur-complement-1}
    U = \sigma_{\star} E + \Upsilon,
\end{equation}
where, to verify the implementation conditions detailed 
in~\Cref{subsec: verification}, there is one and only one singular site within 
each outer box of the decomposition~\eqref{eqn: center-gluing-9k-implementation}. 
Combining~\eqref{eqn: inverse-relation-schur-complement} 
and~\eqref{eqn: schur-complement-1}, we define the single-mode nearly-resonant set at 
the scale $N$ as: 
\begin{equation}
    \label{eqn: single-mode-resonant-set}
    G_{1,N} = \bigcup_{k \in \Lambda_{N^{10}}} \left\{ \omega_T \in \Omega \bigg | \left| |\sigma_{\star}| - \| \Upsilon \|_2 \right| < \sqrt{ \varepsilon_{N} } \right\},
\end{equation}
which ensures that the inversion 
condition~\eqref{eqn: trans-small-size-original-inverse} holds for all tangent 
frequencies $\omega_{T} \notin G_{1,N}$. 
To estimate the measure of this set, we utilize a simplified version of the Cartan 
estimate~\citep{levin1996lectures}.

\begin{lemma}[Theorem 4 of Lecture 11.3 in~\citet{levin1996lectures}]
    \label{lem: cartan}
    Let $f(x)$ be a function analytic in the disk $\{ z: |z| \leq 2eR \}$, $|f(0)| = 1$, 
    and let $\eta$ be an arbitrary small positive number. 
    We use the notation $ M_f(r) := \max_{|z|=r} |f(z)| $.
    Then the estimate
    \[
        \log |f(z)| > - H(\eta) \log M_f(2eR), \qquad H(\eta) = \log \frac{15e^3}{\eta},
    \]
    is valid everywhere in the disk $\{z: |z| \leq R \}$ except a set of disks 
    $(C_j)$ with sum of radii
    \[
        \sum r_{j} \leq \eta R.
    \]
\end{lemma}

To complete the measure estimates, we analyze the magnitude of the residual term 
$\Upsilon$ across two distinct regimes. 
For small residuals, given by $ \| \Upsilon \|_2 < \sqrt{\varepsilon_N}$, the 
single-mode nearly-resonant set is contained within
\begin{equation}
    \label{eqn: single-msa-small-residual}
    G_{1,N} \subseteq \bigcup_{k \in \Lambda_{N^{10}}} \left\{ \omega_T \in \Omega \bigg | |\sigma_{\star}| < 2\sqrt{\varepsilon_N} \right\}.
\end{equation}
Note that this case encompasses the small-scale regimes, where the tighter condition 
$|\sigma_\star| < \varepsilon_N$ is actually sufficient. 
The measure for this regime is bounded by:
\begin{equation}
    \label{eqn: mes-g-1-N-1}
    \mathrm{mes}(G_{1,N}) \leq 2^{m+2} d^{m-1} (2N^{10}+1)^{m} \sqrt{\varepsilon_N}. 
\end{equation}
When the residual term is large, i.e., $\| \Upsilon \|_2 \geq \sqrt{\varepsilon_N}$, 
we define an analytic function:
\[
    f(\sigma_{\star}) = \frac{|\sigma_{\star}|}{\| \Upsilon \|_2 } - 1. 
\]
with the disk parameter set as $\| \Upsilon \|_2 = 2e(1+e)^{-1}R$. 
It follows that $|f(0)| = 1$ and $M_f(2eR) = e$. 
Accordingly, the single-mode nearly-resonant set satisfies:
\begin{equation}
    \label{eqn: single-msa-large-residual}
    G_{1,N} \subseteq \bigcup_{k \in \Lambda_{N^{10}}} \left\{ \omega_T \in \Omega \bigg | |f(\sigma_{\star})| < \varepsilon_N^{1/4} \right\}. 
\end{equation}
By setting $\eta = 15e^3 \varepsilon_N^{1/4}$,~\Cref{lem: cartan} provides the 
following measure estimate: 
\begin{equation}
    \label{eqn: mes-g-1-N-2}
    \mathrm{mes}(G_{1,N}) \leq 8(1+e)e^2 (2N^{10}+1)^m \varepsilon_{N}^{1/4}. 
\end{equation}
Based on~\eqref{eqn: single-mode-resonant-set}, the total single-mode resonant set is 
the union across all scales $N \geq M_0$ as 
\begin{equation}
    \label{eqn: final-msa-single-mode-nearly-resonant}
    G_1 = \bigcup_{N \geq M_{0}} G_{1,N}. 
\end{equation}
Synthesizing the estimates from~\eqref{eqn: mes-g-1-N-1} and~\eqref{eqn: mes-g-1-N-2}, 
we derive
\[
    \mathrm{mes}(G_1) \leq \sum_{N \geq M_0} 2^{m+2} d^{m-1} (2N^{10}+1)^m \varepsilon_{N}^{1/4} < \infty.
\]
Notably, as $\varepsilon \rightarrow 0$, if follows $M_0 \rightarrow \infty$, which 
implies $\mathrm{mes}(G_1) \rightarrow 0$. 
The combined set $G = G_1 \cup G_2$ corresponds to the set defined 
in~\Cref{thm: off-side-box}. 
Consequently, we obtain $\text{mes}(G) \leq \delta$, where $\delta \to 0$ as 
$\varepsilon \to 0$, consistent with the statement of~\Cref{thm: main}.

\section{Convergence of the numerical scheme}
\label{sec: iterative-scheme}

In this section, we establish the theoretical foundation for the convergence of the 
proposed numerical scheme~\eqref{eqn: numerical-scheme}.~\Cref{subsec: induction} 
derives the iteration lemma, which serves as the technical cornerstone of our 
analysis. 
Subsequently, in~\Cref{subsec: proof-main}, we utilize this result to complete the 
formal proof of the main statement.

%===================================================%
\subsection{The iteration lemma}
\label{subsec: induction}

Before proceeding to the inductive estimates, we must characterize the derivative of 
the inverse operator $\partial L_N^{-1}$, specifically regarding its norm bounds and 
off-diagonal decay properties. 
Here, the derivative is taken with respect to the tangent frequency $\omega_T$, 
which we omit from the notation for simplicity. 
Using the standard operator identity $L_N^{-1} L_N = I$, the derivative of the 
inverse is given by:
\begin{equation}
    \label{eqn: t-n-derivative-inverse}
    \partial L_{N}^{-1} = - L_{N}^{-1}(\partial L_{N} ) L_{N}^{-1}.
\end{equation}
This relationship allows us to establish the following lemma regarding the operator's 
composite norm and spatial localization.

\begin{lemma}
    \label{lem: inverse-derivative}
    If the restricted operator $L_{N}$ satisfies~\Cref{cond: implementation}, then 
    its inverse satisfies the norm bound:
    \begin{equation}
        \label{eqn: inverse-restricted-norm}
        \| \partial L_{N}^{-1} \|_2 \leq \frac{4\sqrt{m}N}{\varepsilon_N^2}. 
    \end{equation}
    Furthermore, for sufficiently large spatial separations in the lattice 
    $\Lambda_N$, specifically when $|k - k'| \geq N^{\frac34}$ the entries 
    of the inverse matrix must exhibit Gevrey-type decay:
    \begin{equation}
        \label{eqn: inverse-restricted-entries}
        | \partial L_{N}^{-1}\left( k, k' \right) | \leq \exp\left\{ - \frac{| k - k' |^{s}}{4}\right\} .
    \end{equation}
\end{lemma}

The inversion result~\eqref{eqn: inverse-restricted-norm} follows directly from the 
inversion condition~\eqref{eqn: inverse-grow} and the derivative 
identity~\eqref{eqn: t-n-derivative-inverse}. 
The localization result~\eqref{eqn: inverse-restricted-entries} involves basic 
inequalities associated with Gevrey regularity; the detailed proof is provided 
in~\Cref{subsec: derivative-app}. 
Given that~\Cref{cond: implementation} was verified 
previously,~\Cref{lem: inverse-derivative} implies the iteration lemma. 
For technical convenience, we adopt the convention that $N_{-1} = 0$ and 
$\Lambda_{N_{-1}} = \varnothing$.

\begin{theorem}[The Iteration Lemma]
    \label{lem: iteration}
    Suppose that the tangent frequency satisfies 
    $\omega_T \in \Omega \setminus (\Omega_{M}^{\tau} \cup G)$. 
    Let $\varepsilon_0 = \varepsilon_0(H_1, \Omega, a) > 0$ be the threshold provided 
    in~\Cref{thm: main}. 
    For any $0 < \varepsilon \leq \varepsilon_0$, there exists a Gevrey exponent 
    $s(\varepsilon) > 0$ such that, at the $r$-th iteration, the approximate solution 
    $\hat{z}^{(r)}$ is supported on: 
    \begin{equation}
        \label{eqn: vector-support-r}
        \mathrm{supp} \; \hat{z}^{(r)} \subseteq  \Lambda_{N_r}. 
    \end{equation}
    For any index $i \in \{ 0, 1, \ldots, r\}$ and any 
    $k \in \Lambda_{N_i} \setminus \Lambda_{N_{i-1}}$, assume the Gevrey decay 
    property holds: 
    \begin{equation}
        \label{eqn: vector-decay-r}   
        \| \hat{z}^{(r)}(k ) \| \leq \sum_{\ell=i}^{r}\exp\left\{ - \frac32 N_{\ell}^{s} \right\},
    \end{equation}
    and the vector field $F$ satisfies: 
    \begin{equation}
        \label{eqn: vector-field-decay-r}  
        \big\| F\big( \hat{z}^{(r)}_p, \;\omega_{T}^{(r+1)}; a, \omega \big) \big\| \leq  \exp\left\{-2N_{r}^{s} \right\}.                                                                                                                                                                                                                                                                   
    \end{equation}
    Under the numerical update scheme~\eqref{eqn: numerical-scheme}, these 
    properties are preserved at the $(r+1)$-th iteration. 
    Specifically, the update iterate $\hat{z}^{(r+1)} $ satisfies: 
    \begin{equation}
        \label{eqn: vector-support-r+1}
        \mathrm{supp} \; \hat{z}^{(r+1)} \subseteq  \Lambda_{N_{r+1}}. 
    \end{equation}
    For any $i \in \{ 0, 1, \ldots, r+1\}$ and any 
    $k \in \Lambda_{N_i} \setminus \Lambda_{N_{i-1}}$, the following decay property 
    holds: 
    \begin{equation}
        \label{eqn: vector-decay-r+1}   
        \| \hat{z}^{(r+1)}(k) \| \leq \sum_{\ell=i}^{r+1}\exp\left\{ - \frac32 N_{\ell}^s \right\},
    \end{equation}
    and the updated vector field $F$ satisfies: 
    \begin{equation}
        \label{eqn: vector-field-decay-r+1}  
        \big\| F \big( \hat{z}_p^{(r+1)},\omega_T^{(r+2)}; a, \omega \big) \big\| \leq  \exp\left\{-2N_{r+1}^{s} \right\}.                                                                                                                                                                                                                                                                 
    \end{equation}
\end{theorem}

\begin{proof}[Proof of~\Cref{lem: iteration}]
    For conciseness, we suppress the auxiliary parameters from the 
    dimension-enlarged Newton scheme~\eqref{eqn: coeffcient-iter} and define the 
    iterative update as:
    \begin{equation}
        \label{eqn: nash-moser-single-simplification}
        \Delta^{(r)} = \hat{z}^{(r+1)}_p - \hat{z}^{(r)}_p = - L_{N_{r+1}}^{-1} \big(  \hat{z}^{(r)}_p \big) F\big(  \hat{z}^{(r)}_p \big).
    \end{equation}
    The proof is organized in the following three steps.

    \paragraph{Support enlargement} This step verifies the inclusion 
    property~\eqref{eqn: vector-support-r+1}. 
    Given that the perturbation $H_1$ is a polynomial of total degree at most $d$, 
    the vector field is located in the support:
    \begin{equation}
        \label{eqn: support-d-enlarge}
        \mathrm{supp} \; F\big( \hat{z}^{(r)}_p \big) \subseteq [-dN_{r}, dN_{r}]^m \subseteq \left[ -\frac{N_{r+1}}{4}, \frac{N_{r+1}}{4} \right]^m,
    \end{equation}
    which confirms~\eqref{eqn: vector-support-r+1} in accordance with the update 
    rule~\eqref{eqn: numerical-scheme}.

    \paragraph{Lattice vector decay} This step establishes the 
    bound~\eqref{eqn: vector-decay-r+1}. 
    Given the inversion condition~\eqref{eqn: inverse-grow} and the previous vector 
    field estimate~\eqref{eqn: vector-field-decay-r}, we bound the iterative 
    difference as:
    \[
        \| \Delta^{(r)} \|_2 \leq \big \| L_{N_{r+1}}^{-1} \big( \hat{z}^{(r)}_p \big) \big\|_2 \big \| F\big( \hat{z}^{(r)}_p\big) \big\|_2 \leq \frac{1}{\varepsilon_{N_{r+1}}} \exp\left\{-2N_{r}^{s}\right\}.
    \]
    Furthermore, by applying~\Cref{lem: inverse-derivative}, we estimate the 
    derivative of the update as:
    \begin{align*}
        \| \partial \Delta^{(r)} \|_2 
        & \leq \big\| \partial L_{N_{r+1}}^{-1}  \big(  \hat{z}^{(r)}_p  \big) \big \|_2 \big\| F\big(  \hat{z}^{(r)}_p\big) \big\|_2 + \big \| L_{N_{r+1}}^{-1}  \big(  \hat{z}^{(r)}_p  \big) \big \| _2\big\| \partial F\big(  \hat{z}^{(r)}_p\big) \big \|_2                                                            \nonumber \\
        & \leq \left( \frac{4\sqrt{m}N_{r+1} + \varepsilon_{N_{r+1}}}{\varepsilon_{N_{r+1}}^2}\right) \exp\left\{-2N_{r}^{s}\right\}.       
    \end{align*} 
    By combining the two bounds above, there exists a threshold $\varepsilon_0 > 0$ 
    such that such that:
    \begin{equation}
        \label{eqn: difference-derivative-estimate}
        \| \Delta^{(r)} \| \leq \exp\left\{-\frac{7N_{r}^{s}}{4}\right\},
    \end{equation}
    for any $0 < \varepsilon \le \varepsilon_0$. 
    We now choose $s = s(\varepsilon)>0$ such that $6M^{s} \leq 7$. 
    Consequently, at the $(r+1)$-th iteration, for any index 
    $i \in \{0, 1, \ldots, r+1\}$ and any 
    $k \in \Lambda_{M^i} \setminus \Lambda_{M^{i-1}}$, the following holds:
    \begin{align*}
        \| \hat{z}^{(r+1)}(k) \| \leq \| \hat{z}^{(r)}(k) \| + \|\Delta^{(r)}(k)\| \leq \sum_{\ell=i}^{r+1}\exp\left\{ - \frac32 N_{\ell}^{s} \right\},
    \end{align*}
    which confirms that the decay property~\eqref{eqn: vector-decay-r+1} is 
    preserved across successive iterations.

    \paragraph{Super-exponential convergence} To establish the super-convergence 
    property~\eqref{eqn: vector-field-decay-r+1}, we analyze the behavior of the 
    vector field following the numerical update. 
    By applying a Taylor expansion to the vector field and the frequency, and 
    invoking~\Cref{prop: tensor}, the post-update residual is bounded by the 
    truncation error and the quadratic terms of the iteration: 
    \begin{equation}
        \label{eqn: vector-estimate-taylor-r+1}
        \| F(\hat{z}_p^{(r+1)}, \omega_T^{(r+2)}) \|_2 \leq \| [ L - L_{N_{r+1}}] \Delta^{(r)} \|_2 + (2N_{r+1}+1)^\frac{m}{2} \exp\left\{- 3N_{r+1}^{s} \right\}. 
    \end{equation}
    Similarly, the derivative of the residual satisfies:
    \begin{align}
        \| \partial F(\hat{z}_p^{(r+1)}, \omega_T^{(r+2)}) \|_2 
        \leq & \| \partial [ L - L_{N_{r+1}}] \Delta^{(r)} \|_2 + \| [ L - L_{N_{r+1}} ]\partial \Delta^{(r)} \|_2  \nonumber  \\ 
             & + 3(2N_{r+1}+1)^\frac{m}{2}\exp\left\{- 3N_{r+1}^{s} \right\}.     \label{eqn: vector-estimate-derivative-taylor-r+1}
    \end{align}
    To estimate the Gevrey norms in~\eqref{eqn: vector-estimate-taylor-r+1} 
    and~\eqref{eqn: vector-estimate-derivative-taylor-r+1}, we exploit the 
    off-diagonal decay of the operators $L $ and $L_{N_{r+1}}^{-1}$ established 
    in~\Cref{prop: convolution} and~\Cref{lem: inverse-derivative}. 
    Since the estimation procedures for these terms are analogous, we focus on 
    $[L- L_{N_{r+1}}] \Delta^{(r)}$ as the representative example. 
    We utilize the following decomposition:
    \begin{align}
        [L- L_{N_{r+1}}] \Delta^{(r)} 
        =  & \left[ (E - P_{N_{r+1}}) LP_{N_{r+1}} \right] \left[ (E - P_{N_{r+1}/3})  \Delta^{(r)} \right]  \nonumber \\
           & + \left[ (E - P_{N_{r+1}}) L P_{N_{r+1}/3} \right] \Delta^{(r)}.  \label{eqn: decompostion-t-tr}
    \end{align}
    Given the confirmed off-diagonal decay, there exists a threshold 
    $\varepsilon_0 > 0$ such that for $0 < \varepsilon \leq \varepsilon_0$, the 
    following estimates holds:
    \begin{equation}
        \label{eqn: final-super-convergence-bound-1}
        \|(E - P_{N_{r+1}}) L P_{N_{r+1}/3} \| \leq \frac13 \exp\left\{ - \frac{7}{12} N_{r+1}^{s} \right\},
    \end{equation}
    and
    \begin{equation}
        \label{eqn: final-super-convergence-bound-2}
        \|(E - P_{N_{r+1}/3}) L_{N_{r+1}}^{-1} P_{N_{r+1}/4} \| \leq \frac13 \exp\left\{ - \frac{1}{48} N_{r+1}^{s} \right\}.
    \end{equation}
    Noting that $\|(E - P_{N_{r+1}}) LP_{N_{r+1}} \| \leq \varepsilon \|S+B\| \leq 1$, 
    we combine the bounds~\eqref{eqn: final-super-convergence-bound-1} 
    and~\eqref{eqn: final-super-convergence-bound-2} with the decomposition 
    in~\eqref{eqn: decompostion-t-tr} to account for the truncation errors 
    in~\eqref{eqn: vector-estimate-taylor-r+1} 
    and~\eqref{eqn: vector-estimate-derivative-taylor-r+1}. 
    For a sufficiently small threshold $\varepsilon_0 > 0$, we obtain
    \begin{equation}
        \| F(\hat{z}_p^{(r+1)}, \omega_T^{(r+2)}) \| \leq  \frac{2}{3} \exp\left\{ -\left( \frac{2}{M^{s}}  + \frac{1}{48}\right) N_{r+1}^{s} \right\}  +  4(2N_{r+1}+1)^\frac{m}{2}\exp\left\{- 3N_{r+1}^{s} \right\}.   \label{eqn: final-super-convergence-1}
    \end{equation}
    By choosing $s = s(\varepsilon)>0$ such that $95 M^{s} \leq 96$, we arrive at 
    the desired super-convergence bound:
    \begin{equation}
        \label{eqn: final-super-convergence-2}
        \| F(\hat{z}_p^{(r+1)}, \omega_T^{(r+2)}) \|_s \leq \exp\left\{ - 2  N_{r+1}^{s} \right\}, 
    \end{equation}
    which concludes the proof of the Iteration Lemma~(\Cref{lem: iteration}). 
\end{proof}

The decay estimates in~\Cref{lem: iteration} demonstrate that despite potential 
fluctuations in individual terms, the cumulative decay remains strictly controlled. 
For any multi-index $k \in \Lambda_{N_r} \setminus \Lambda_{N_{r-1}}$, the series 
converges as follows:
\[
    \sum_{\ell = r}^{\infty} \exp\left\{-\frac{3N_{\ell}^{s}}{2}\right\} \leq \exp\left\{-N_{r}^{s} \right\} \leq \exp\left\{ - |k|^{s} \right\},
\]
which ensures that the sequence $\hat{z}^{(r)}$ remains within the space 
$\mathcal{K}(s)$. 
Consequently, the solution preserves its spatial localization throughout the 
iteration, preventing any degradation of regularity as the lattice size $N$ increases.

%=============================================%
\subsection{Proof of the main theorem}
\label{subsec: proof-main}

Since $\hat{z}_q^{(r)}$ stays fixed throughout the iterations, the estimate 
\eqref{eqn: difference-derivative-estimate} remains valid when $ \hat{z}_p $ is 
replaced by $ \hat{z} $. The displacement between successive iterations is bounded as:
\[
    \| \hat{z}^{(r+1)} - \hat{z}^{(r)} \| = \| \Delta^{(r)} \| \leq \exp \left\{ - \frac{7}{4} N_{r}^s \right\} = \exp \left\{ - \frac{7}{4} (M^s)^r \right\},
\]
which implies the convergence of the sequences $\{ \hat{z}^{(r)} \}$. 
Furthermore, we can bound the distance to the limit $\hat{z}^\star$ by estimating 
the tail of the sequence:
\[
    \| \hat{z}^{(r)} - \hat{z}^\star \| \leq \sum_{\ell = 0}^{\infty} \| \hat{z}^{(r+\ell +1)} - \hat{z}^{(r+\ell )} \| \leq \sum_{\ell  = 0}^{\infty} \exp \left\{ - \frac{7}{4} (M^s)^{r+\ell } \right\} \leq \exp \left\{ - \frac{3}{2} (M^s)^{r} \right\},
\]
which confirms \eqref{eqn: z-hat-converge}. 
Moreover, from the tangent frequency update~\eqref{eqn: frequency-iter} and the 
uniform bound \eqref{eqn: tangent-bound}, we have:
\begin{equation*}
    | \omega_T^{(r+1)} - \omega_T^{(r)} | \leq \varepsilon \gamma \| \hat{z}^{(r+1)} - \hat{z}^{(r)} \| \leq \varepsilon \gamma \exp \left\{ - \frac{7}{4} (M^s)^r \right\}.
\end{equation*}
Thus, the convergence of the frequency sequences $\{ \omega_T^{(r)} \}$ is guaranteed, 
and we also have
\[
    | \omega_T^{(r)} - \omega_T^\star | \leq \sum_{\ell = 0}^{\infty} | \omega_T^{(r+\ell +1)} - \omega_T^{(r+\ell )} | \leq \sum_{\ell = 0}^{\infty} \varepsilon \gamma \exp \left\{ - \frac{7}{4} (M^s)^{r+\ell } \right\}  \leq \exp \left\{ - \frac{3}{2} (M^s)^{r} \right\},
\]
which confirms \eqref{eqn: frequency-hat-converge}. 
Finally, we express the exact solution via its Fourier expansion:
\[
    z^\star(t) = \sum_{k \in \mathbb{Z}^m} \hat{z}^\star(k) e^{i \langle k, \omega_T^\star \rangle t}.
\]
Based on the support property \eqref{eqn: vector-support-r}, the error can be split 
into two terms:
\begin{align*}
    \| z^{(r)}(t) - z^\star(t) \| &\leq \left\| \sum_{k \in \mathbb{Z}^m} \left( \hat{z}^\star(k) - \hat{z}^{(r)}(k) \right) e^{i \langle k,  \omega_T^\star\rangle t} \right\|
    + \left\| \sum_{k \in \Lambda_{N_{r}}} \hat{z}^{(r)}(k) \left( e^{i \langle k, \omega_T^\star \rangle t} - e^{i \langle k, \omega_T^{(r)} \rangle t} \right) \right\|  \\
    & \defi I_1 + I_2.
 \end{align*}
By applying the compact support property~\eqref{eqn: vector-support-r} and the Gevrey 
decay property~\eqref{eqn: vector-decay-r} alongside the Cauchy-Schwarz inequality, 
we estimate the term $I_1$ as: 
\begin{align*}
    I_1 &\leq \sum_{k \in \Lambda_{M^{r}} } \| \hat{z}^\star(k) - \hat{z}^{(r)}(k) \| + \sum_{k \notin \Lambda_{M^{r}} } \| \hat{z}^\star(k) \|  \\
        &\leq \| \hat{z}^{(r)} - \hat{z}^\star \| \cdot (2M^{r} + 1)^{m/2} + \sum_{k \notin \Lambda_{M^{r}} } e^{-\frac{5}{4} |k|^s}  \leq \frac{1}{2} \exp \left\{ - (M^s)^{r} \right\}.
\end{align*}
where the last inequality follows from the bound~\eqref{eqn: z-hat-converge} 
established above. 
Given the uniform boundedness of $\| \hat{z}^{(r)} \|$, we again apply the 
Cauchy-Schwarz inequality to estimate the term $I_2$ as
\begin{align*}
    I_2 &\leq \| \hat{z}^{(r)} \|  \left( \sum_{k \in \Lambda_{M^{r}}} \left| e^{i \langle k, \omega_T^\star \rangle t} - e^{i \langle k, \omega_T^{(r)} \rangle t} \right|^2 \right)^\frac{1}{2}   \\
        &\leq \left( \| \hat{z}^\star \| + \exp \left\{ - \frac{3}{2} (M^s)^{r} \right\} \right) \cdot | \omega_T^{(r)} - \omega_T^\star | \cdot \left( \sum_{k \in \Lambda_{M^{r}}} \|k t \|^2 \right)^\frac{1}{2}  \leq \frac{1}{2} \exp \left\{ - (M^s)^{r} \right\},
\end{align*}
which holds uniformly for any $t \in [0, N_r] = [0, M^{r}]$. 
By combining the estimate for terms $I_1$ and $I_2$, we arrive at the final 
time-domain error bound \eqref{eqn: convergence-soln}, which completes the proof 
of~\Cref{thm: main}.

\section{Numerical Experiments}
\label{sec: numer}

Building upon the numerical construction of full-dimensional quasi-periodic 
solutions developed in~\citet{fu2026numerical}, we extend this framework to the 
computation of elliptic lower-dimensional quasi-periodic solutions. 
In contrast to the full-dimensional setting, the lower-dimensional case requires the 
simultaneous treatment of both the tangent frequencies $\omega_T$ and the normal 
frequencies $\omega_N$. The latter are of fundamental importance, as they govern the 
stability and variations of the associated linear quasi-periodic solutions. 
More specifically, the tangent directions correspond to the excited amplitudes where 
$a \neq 0$, while the normal directions correspond to the non-excited modes where 
$a = 0$.

In this section, we demonstrate the computational efficiency of our alternating 
numerical procedure~\eqref{eqn: numerical-scheme} for constructing such 
quasi-periodic solutions. 
Our numerical experiments are conducted on two representative physical models 
spanning different scales: the macroscopic system from celestial mechanics given by 
the H\'{e}non-Heiles model, and the microscopic nonlinear lattice system described 
by the FPU model.

%======================================================%
\subsection{The H\'{e}non-Heiles model}
\label{subsec: henon-heiles}

The H\'{e}non-Heiles model serves as a seminal two-degree-of-freedom system, 
originally developed to describe stellar motion within an axisymmetric galactic 
potential~\citep{henon1964applicability}. 
The dynamics of this system evolve within the phase space $\mathbb{R}^4$, dictated 
by the Hamiltonian
\begin{equation}
    \label{eqn: henon-heiles-hamiltonian}
    H = H(x, y) = \frac{1}{2} \omega_1(x_1^2 + y_1^2) + \frac{1}{2} \omega_2(x_2^2 + y_2^2) + \varepsilon \left(y_1^2 y_2 - \frac{ y_2^3}{3} \right),
\end{equation}
which is endowed with the canonical symplectic structure $\varpi = dx \wedge dy$. 
To facilitate the implementation of our alternating numerical procedure~\eqref{eqn: numerical-scheme}, we introduce a complex change of variables for $j=1,2$, defined as: 
\[
    z_{j} = \frac{y_j - ix_{j}}{\sqrt{2}} \quad \mathrm{and} \quad \overline{z}_{j} = \frac{y_j + ix_{j}}{\sqrt{2}}.
\]
Under this transformation, the Hamiltonian is reformulated as:
\begin{equation}
    \label{eqn: henon-complex-hamiltonian}
    H = H(z, \overline{z}) = \omega_{1}|z_{1}|^{2} + \omega_{2}|z_{2}|^{2} + \varepsilon\left[ \frac{1}{2\sqrt{2}}(z_{1} + \overline{z}_{1})^{2}(z_{2} + \overline{z}_{2}) - \frac{1}{6\sqrt{2}}(z_{2} + \overline{z}_{2})^{3} \right],
\end{equation}
with the corresponding symplectic structure $\varpi = i dz \wedge d\overline{z}$.

\begin{figure}[htb!]
    \centering
    \begin{subfigure}[t]{0.46\linewidth}
    \centering
    \includegraphics[scale=0.28]{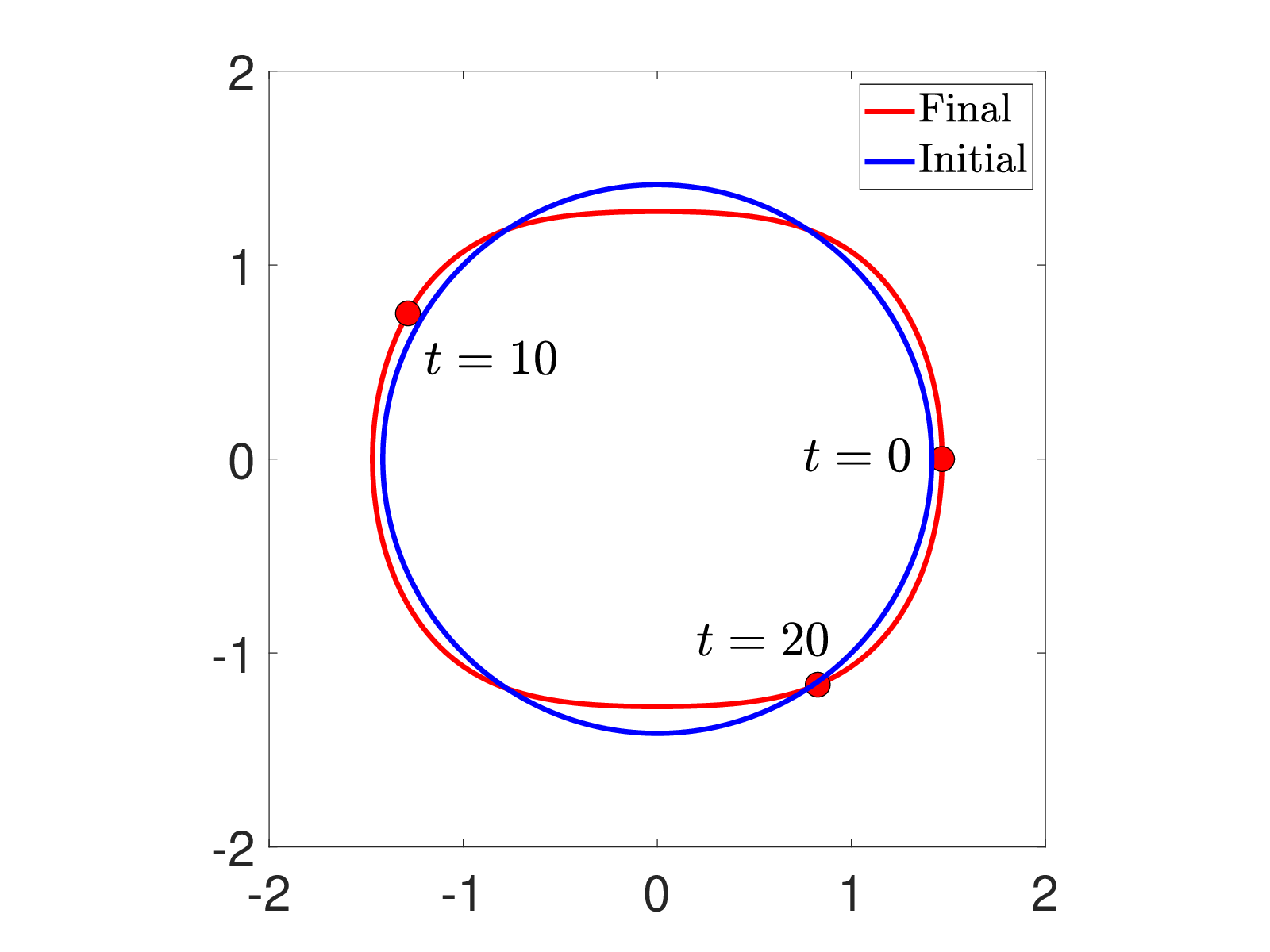}
    \end{subfigure}
    \begin{subfigure}[t]{0.46\linewidth}
    \centering
    \includegraphics[scale=0.28]{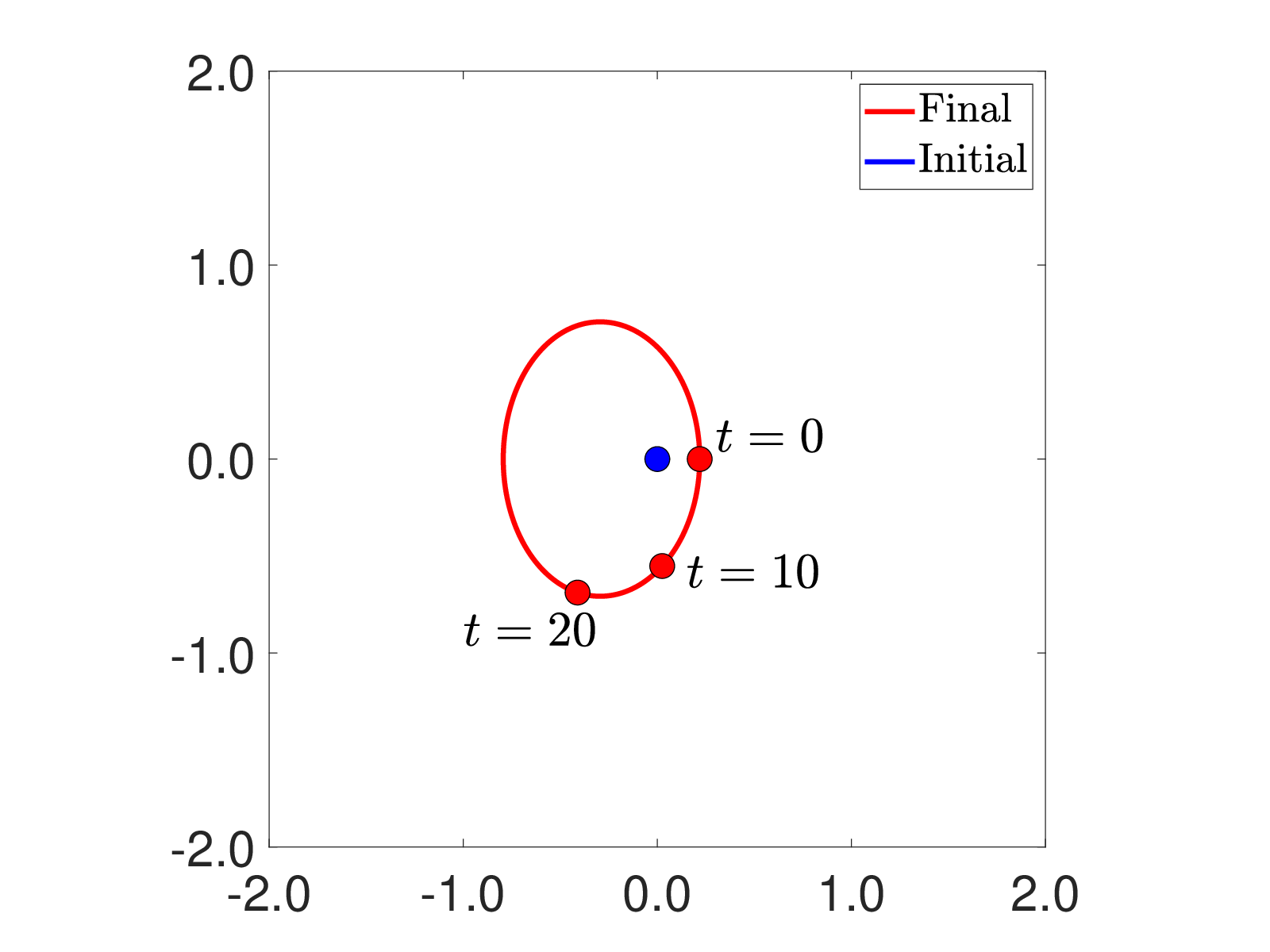}
    \end{subfigure} \\
        \begin{subfigure}[t]{0.46\linewidth}
    \centering
    \includegraphics[scale=0.28]{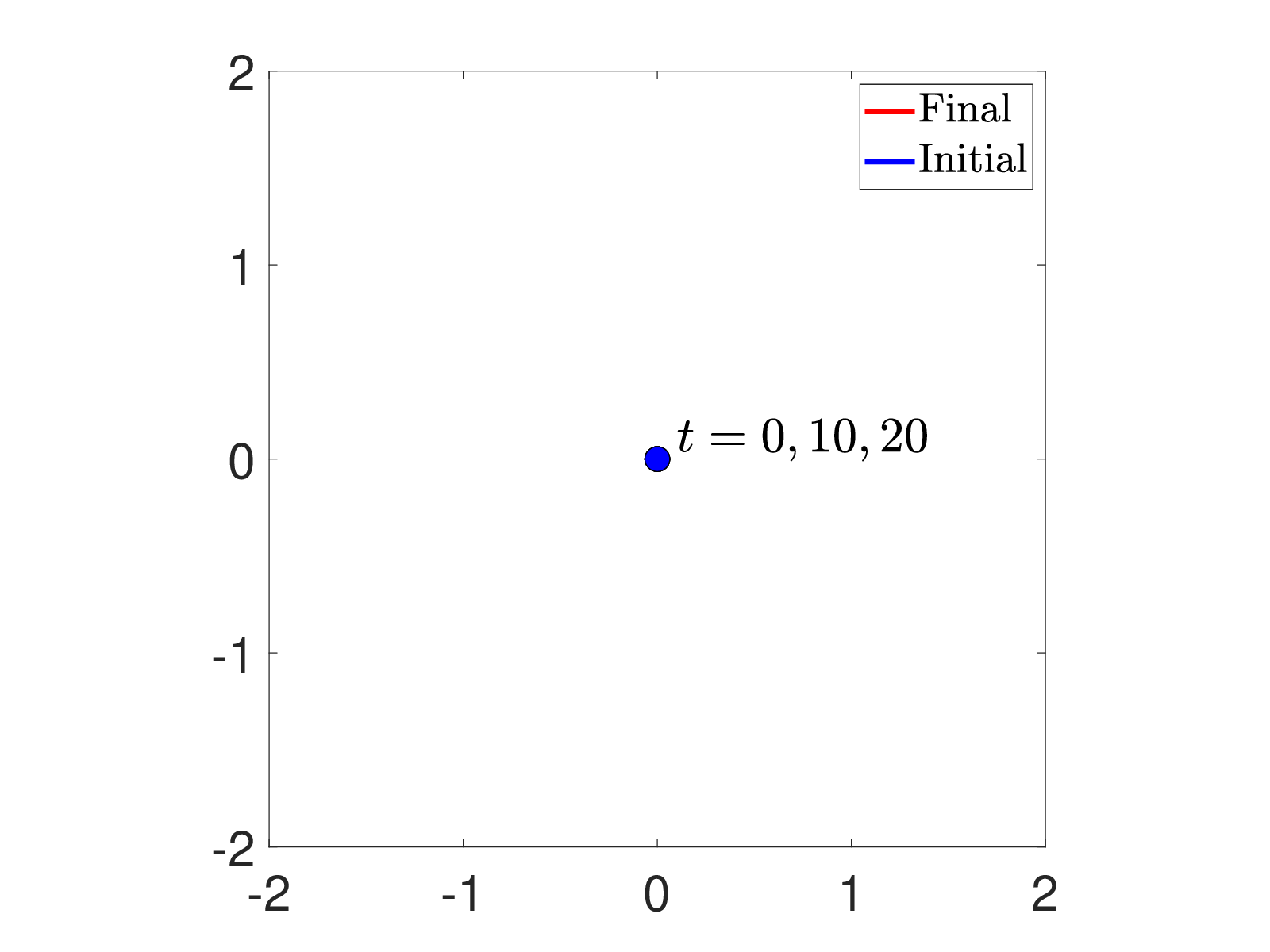}
    \caption{$(x_1, y_1)$}
    \end{subfigure}
    \begin{subfigure}[t]{0.46\linewidth}
    \centering
    \includegraphics[scale=0.28]{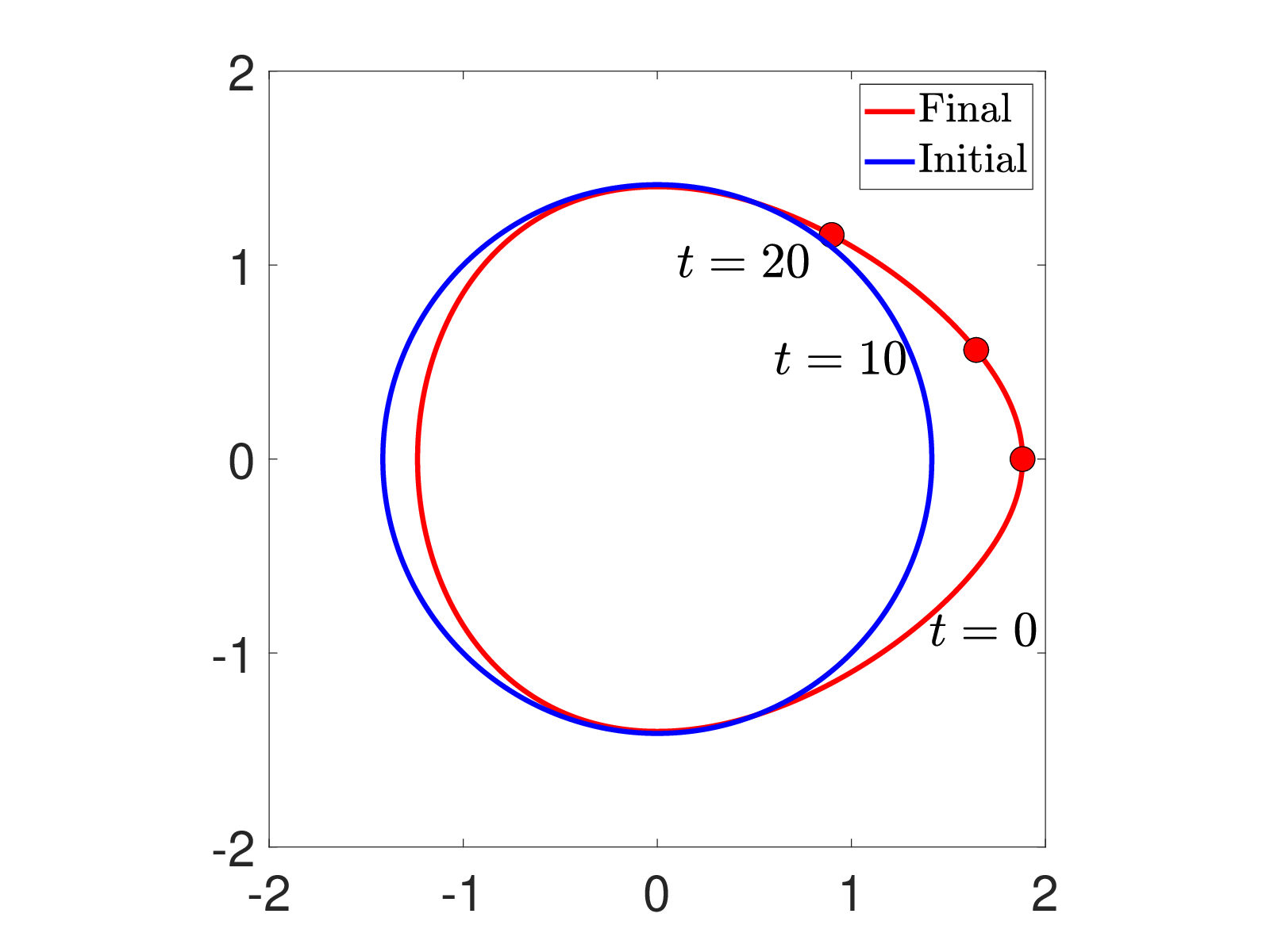}
    \caption{$(x_2, y_2)$}
    \end{subfigure}
    \caption{Phase space trajectories of the H\'{e}non-Heiles model with markers indicating specific time points at $t = 0$, $10$,  and $20$, where the top and bottom panels correspond to initial amplitudes $a=(1,0)$ and $a=(0,1)$, respectively.}
    \label{fig: henon-heiles-traj}
\end{figure}

In the numerical experiments, the perturbation parameter is set to 
$\varepsilon = 0.5$, and the two frequencies are chosen as $\omega_1 = 1$ and 
$\omega_2 = \sqrt{2}$ to satisfy the irrational Diophantine condition. 
In contrast to the full-dimensional setting where the initial amplitude is typically 
taken as $a = (1, 1)$, we here select the initial amplitudes as $a= (1, 0)$ and 
$a=(0,1)$ to obtain the lower-dimensional quasi-periodic solutions, with the 
resulting numerical performance illustrated in~\Cref{fig: henon-heiles-traj}. 
It can be observed that when the normal frequency $\omega_N$ occupies different modes, 
the system exhibits distinct dynamical behaviors: when the normal frequency satisfies 
$\omega_{N} = \omega_2$, the tangent mode is perturbed and shifts inward, while the 
normal mode deviates from the origin to form a perturbed periodic solution; 
conversely, when the normal frequency satisfies $\omega_{N} =\omega_1$, the tangent 
mode is also perturbed, manifesting as a rightward protrusion, whereas the normal 
mode remains localized at the origin. 
These variations are attributed to the intrinsic asymmetry of the nonlinear 
perturbation.

To further assess the convergence performance, which is consistent with the rigorous 
statement provided in~\Cref{thm: main}, we examine the iteration errors which 
exhibit super-exponential convergence as illustrated in~\Cref{fig: henon-heiles-con}. 
\begin{figure}[htb!]
    \centering
    \begin{subfigure}[t]{0.325\linewidth}
        \centering
        \includegraphics[scale=0.19]{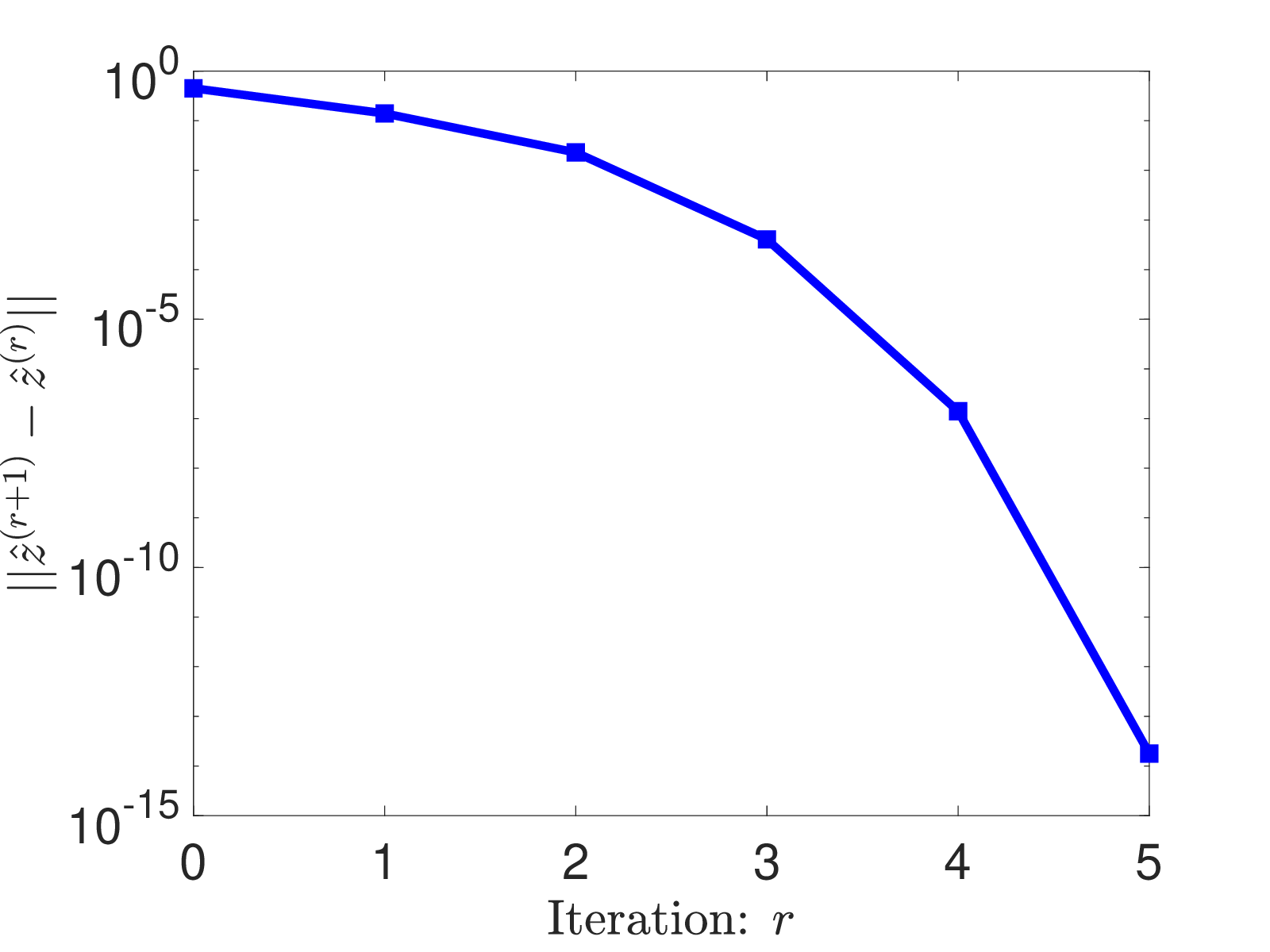}
    \end{subfigure}
    \begin{subfigure}[t]{0.325\linewidth}
      \centering
    \includegraphics[scale=0.19]{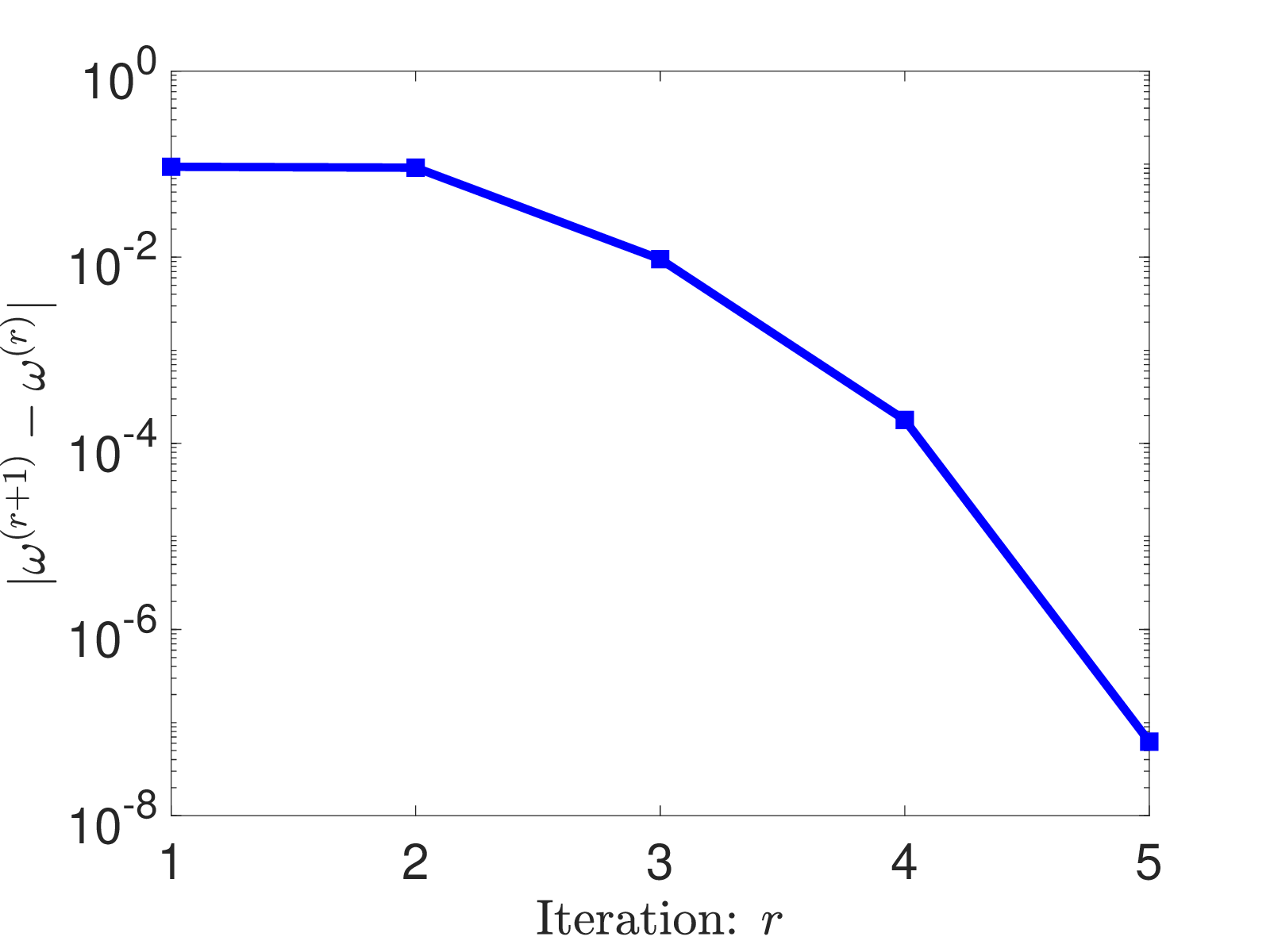} 
    \end{subfigure}
        \begin{subfigure}[t]{0.325\linewidth}
        \centering
        \includegraphics[scale=0.19]{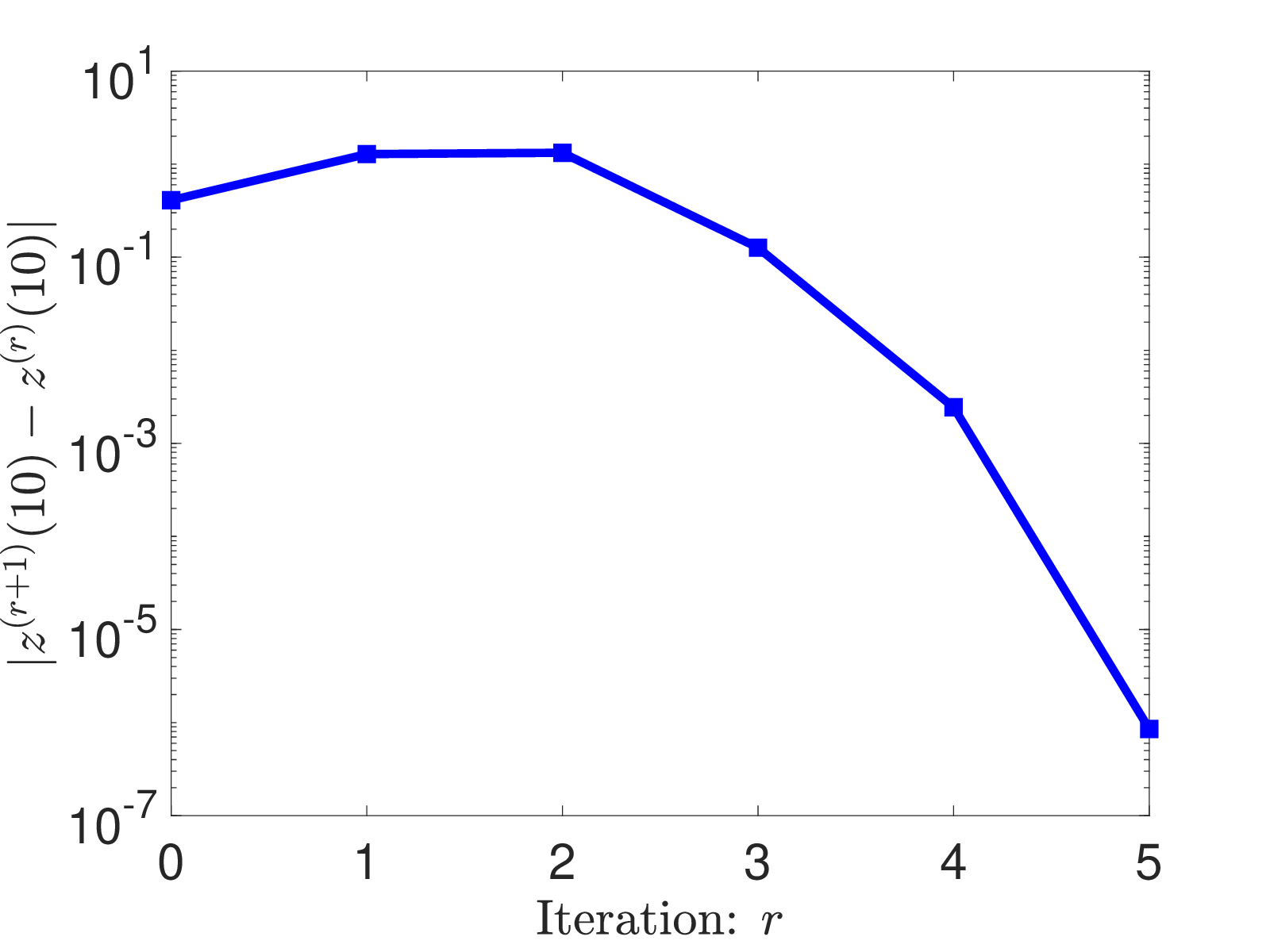}
    \end{subfigure} \\
        \begin{subfigure}[t]{0.325\linewidth}
        \centering
        \includegraphics[scale=0.19]{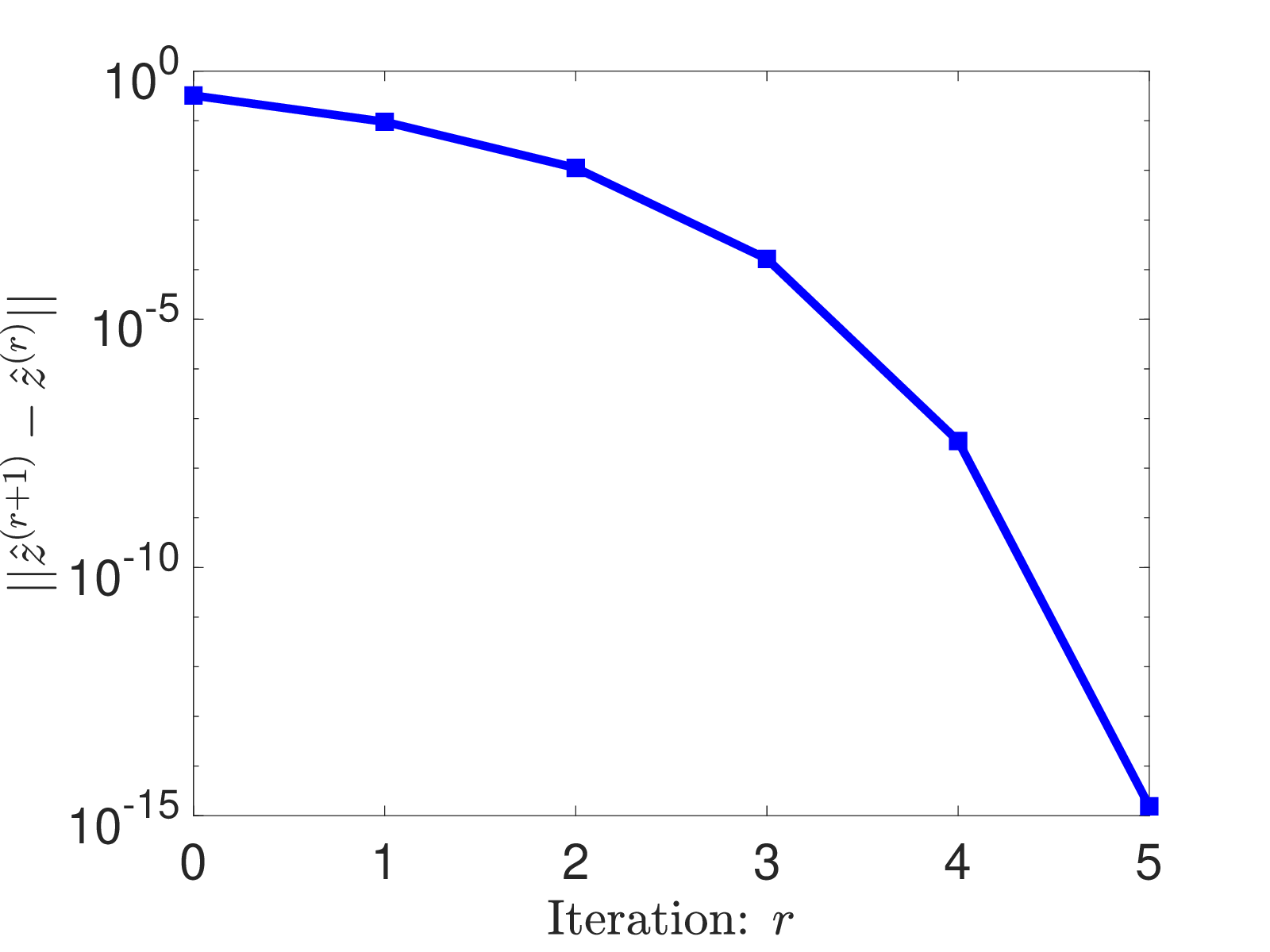}
        \caption{Fourier coefficient vectors}
    \end{subfigure}
    \begin{subfigure}[t]{0.325\linewidth}
        \centering
        \includegraphics[scale=0.19]{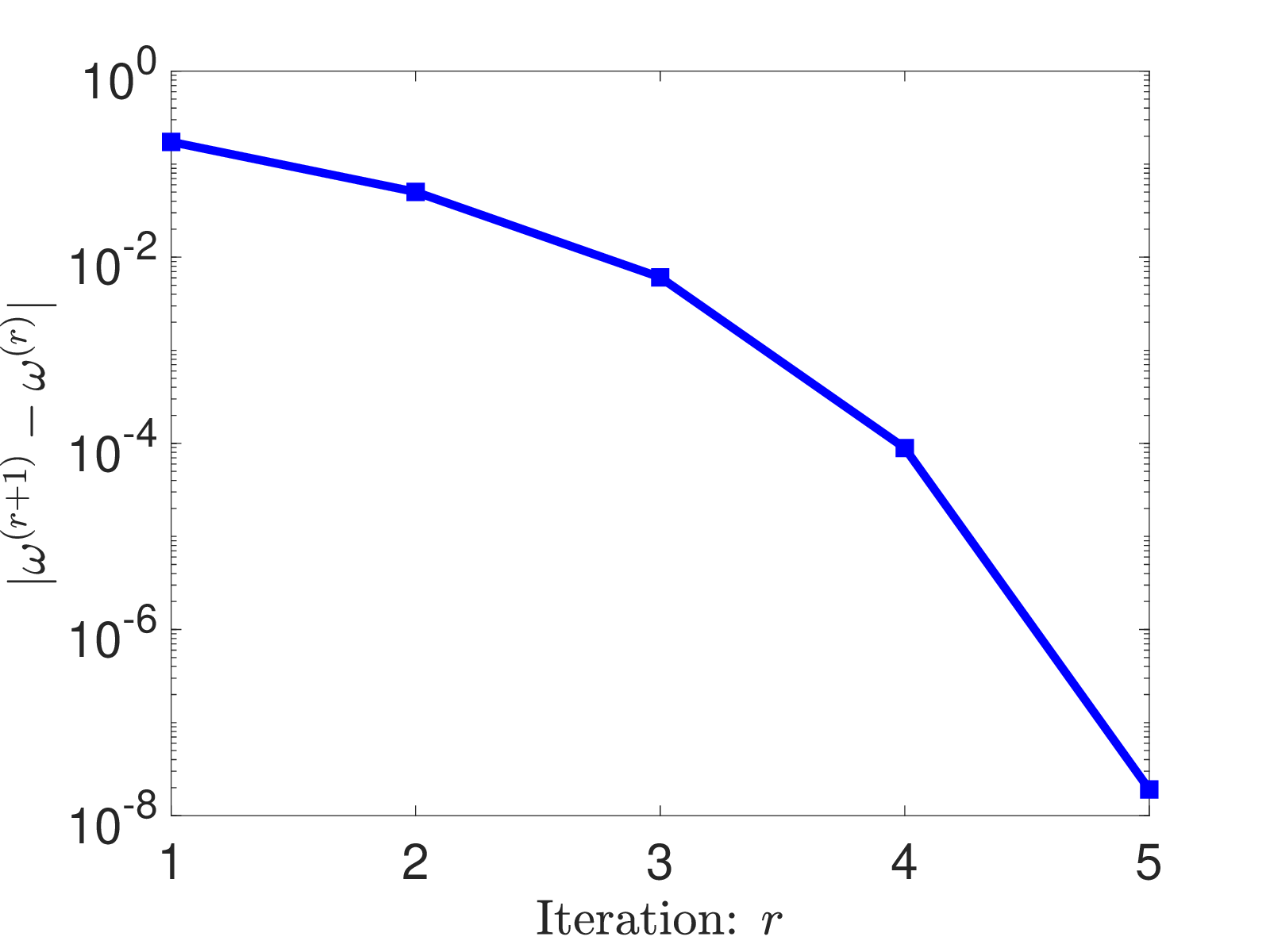}
        \caption{State at $t=10$}
    \end{subfigure}
        \begin{subfigure}[t]{0.325\linewidth}
        \centering
        \includegraphics[scale=0.19]{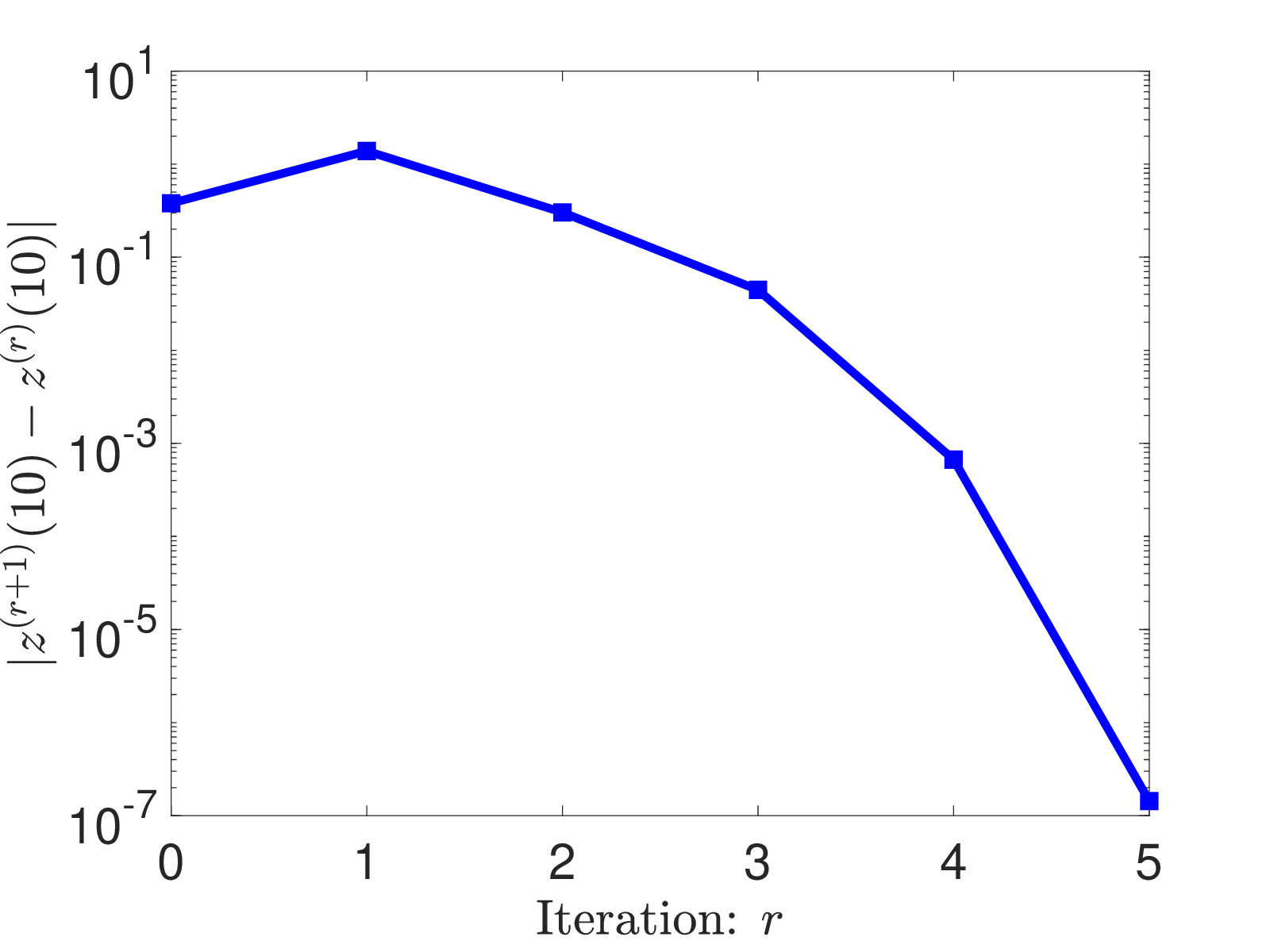}
        \caption{Frequency error}
    \end{subfigure}
    \caption{Convergence profiles across iterations for the H\'{e}non-Heiles model, where the top and bottom panels correspond to initial amplitudes $a=(1,0)$ and $a=(0,1)$, respectively.}
    \label{fig: henon-heiles-con}
\end{figure}
Specifically, we track the three key quantities: (i) the Fourier coefficient error, 
$\| \hat{z}^{(r+1)} - \hat{z}^{(r)} \|_2$, measuring the convergence of the 
coefficients, (ii) the frequency error, $|\omega^{r+1} - \omega^{(r)}|$, quantifying 
the refinement of the computed frequency, and (iii) the pointwise solution error at 
$t = 10$, $|z^{r+1}(10) - z^{(r)}(10)|$, evaluating the discrepancy of the solution 
at a fixed time. 
All three metrics consistently confirm the rapid convergence behavior of the 
proposed alternating scheme~\eqref{eqn: numerical-scheme}.

%======================================================%
\subsection{The FPU model}
\label{subsec: fpu}

Although the H\'{e}non-Heiles model provides a two-dimensional benchmark, it is 
important to further validate the effectiveness of the proposed numerical 
scheme~\eqref{eqn: numerical-scheme} in higher-dimensional and more complex systems. 
For this purpose, we apply this scheme to the Fermi--Pasta--Ulam (FPU) model 
($\beta$-class), a foundational system in nonlinear physics introduced 
by~\citet{fermi1965studies} to investigate the rate of energy transfer toward 
equipartition and thermal equilibrium in a vibrating string.

The FPU model describes a one-dimensional chain of $n$ identical particles of unit 
mass ($m=1$) coupled via nonlinear springs. 
Let $q = (q_1, \ldots, q_n)^{\top}$ denote the generalized coordinates and 
$p = (p_1, \ldots, p_n)^{\top}$ the conjugate momenta. 
Under the Dirichlet boundary condition ($q_0 = q_{n+1} = 0$), the Hamiltonian is 
defined on $\mathbb{R}^{2n}$ as 
\begin{equation}
    \label{eqn: fpu-hamiltonian}
    H = \sum_{j=1}^{n} \frac{p_j^2}{2} + \sum_{j=0}^{n} \frac{(q_{j+1} - q_j)^2}{2} + \varepsilon \sum_{j=0}^{n} \frac{(q_{j+1} - q_j)^4}{4},
\end{equation}
which is endowed with the canonical symplectic structure $\varpi = dp \wedge dq$. 
To facilitate the analysis, we diagonalize the quadratic part via a symmetric 
orthogonal matrix $V \in \mathbb{R}^{n \times n}$ with entries:
\[
    V_{k,j} = \sqrt{\frac{2}{n+1}} \sin\left(\frac{j k \pi}{n+1}\right), \quad j, k = 1, 2, \dots, n.
\]
We then perform the canonical transformation $(p, q) \mapsto (x, y)$ defined by 
$x = W^{-1}Vp$ and $y = WVq$, where $W = \mathrm{diag}(\sqrt{\omega_j})$ is a 
diagonal scaling matrix, and the harmonic frequencies are given by:
\[
    \omega_k = 2 \sin\left(\frac{k \pi}{2(n+1)}\right), \quad k = 1, 2, \dots, n.
\]
Under the defined canonical transformation, the 
Hamiltonian~\eqref{eqn: fpu-hamiltonian} is recast into a more computationally 
tractable form:
\begin{equation}
    \label{eqn: fpu-canonical-hamiltonian}
    H =  H(x, y) = \sum_{j=1}^{n} \frac{\omega_j}{2} (x_j^2 + y_j^2) + \varepsilon H_1(y),
\end{equation}
where the symplectic structure remains invariant 
($\varpi = dx \wedge dy = dp \wedge dq$). 
The quartic perturbation $H_1(y)$ incorporates the scaling factors and takes the 
form:
\[
    H_1(y) = \frac{1}{4} \sum_{j=0}^{n} \left( \sum_{k=1}^{n} \frac{V_{j+1, k} - V_{j, k}}{\sqrt{\omega_k}} y_k \right)^4.
\]
Finally, following the procedure 
in~\eqref{eqn: henon-heiles-hamiltonian} ---~\eqref{eqn: henon-complex-hamiltonian}, 
we transition the Hamiltonian~\eqref{eqn: fpu-canonical-hamiltonian} into complex 
coordinates $z$ and its conjugate $\bar{z}$ as
\begin{equation}
    \label{eqn: fpu-complex-hamiltonian}
    H = H(z, \overline{z}) = \sum_{j=1}^{n} \omega_j |z_j|^2 + \varepsilon H_1(z, \overline{z}),
\end{equation}
which is endowed with the symplectic $2$-form $\varpi = i dz \wedge d \overline{z}$. 
In the complex coordinates, the quartic perturbation is then expressed as
\[
    H_1(z, \overline{z}) = \frac{1}{4} \sum_{j=0}^{n} \left( \sum_{k=1}^{n} \frac{V_{j+1, k} - V_{j, k}}{\sqrt{2\omega_k}} (z_k + \overline{z}_k) \right)^4.
\]

With the Hamiltonian formulated in complex 
variables~\eqref{eqn: fpu-complex-hamiltonian}, we demonstrate the effectiveness of 
the proposed alternating numerical procedure~\eqref{eqn: numerical-scheme} in 
capturing the lower-dimensional quasi-periodic solutions within the FPU model. 
For these numerical trials, we set the system dimension at $n = 3$. 
Unlike the two-dimensional Hénon-Heiles model, which is constrained to a single 
tangent frequency and one normal frequency, the three-particle FPU model allows us 
to verify the algorithm's performance across more complex frequency distributions. 
Specifically, we investigate the following two configurations:
\begin{itemize}
    \item \textbf{Case-I (Periodic solutions)}: One tangent frequency and two normal 
    frequencies.

    \item \textbf{Case-II (Quasi-periodic solutions)}: Two tangent frequencies and 
    one normal frequency.
\end{itemize}
The subsequent numerical experiments are divided into two regimes to rigorously 
evaluate the capture of both one-dimensional and two-dimensional quasi-periodic 
solutions.

%======================================================%
\subsubsection{Single-frequency periodic solutions}
\label{subsubsec: fpu-1d}

To provide a clear demonstration of the numerical experiments, we set the 
perturbation parameter to $\varepsilon = 1$. The resulting dynamics, depicted 
in~\Cref{fig: FPU-1d-traj}, demonstrate that the interaction between two normal 
frequencies produces behavior significantly more complex than those observed in the 
Hénon-Heiles model (see~\Cref{fig: henon-heiles-traj}).
\begin{figure}[htb!]
    \centering
    \begin{subfigure}[t]{0.325\linewidth}
    \centering
    \includegraphics[scale=0.19]{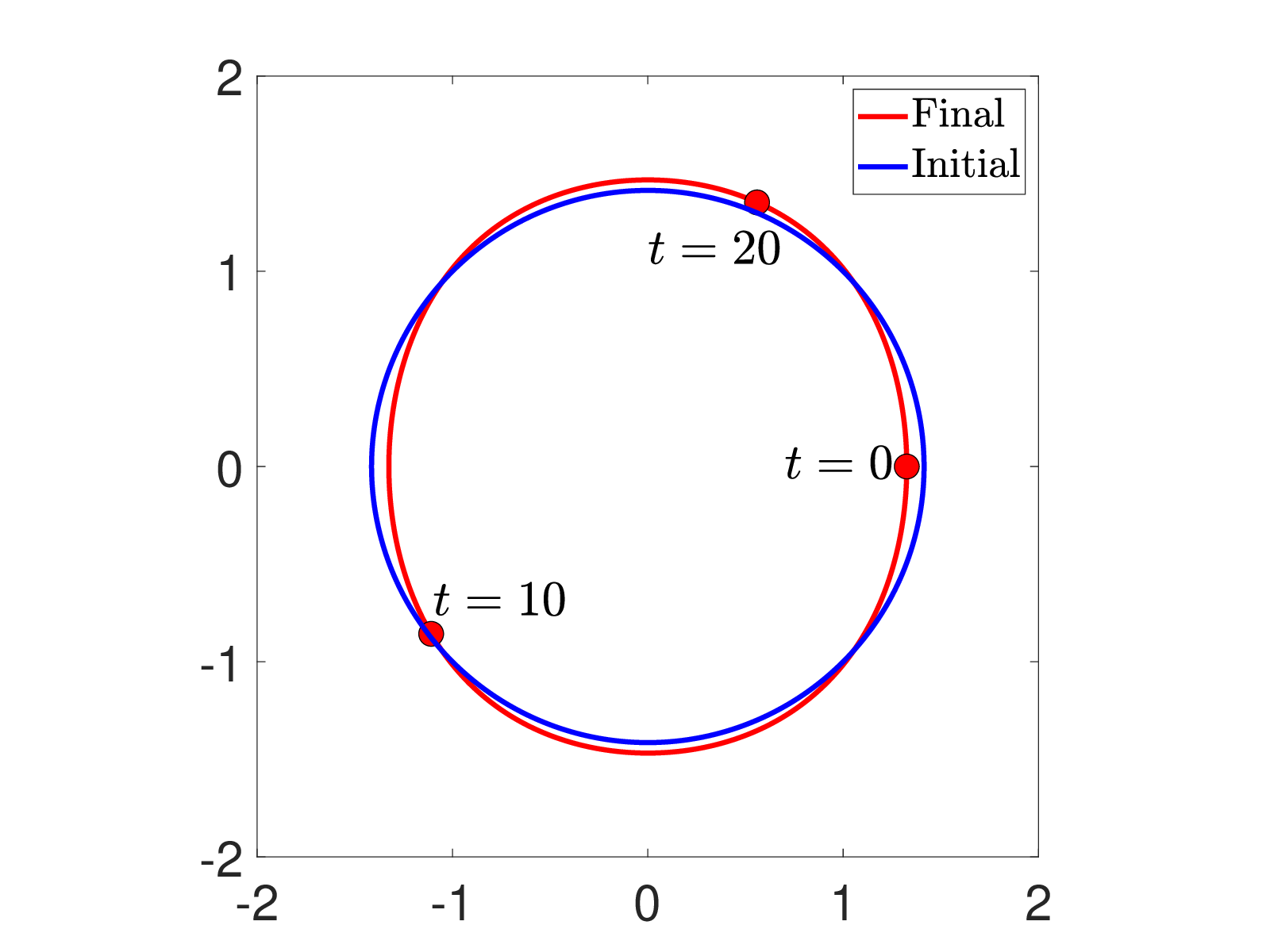}
    \end{subfigure}
    \begin{subfigure}[t]{0.325\linewidth}
    \centering
    \includegraphics[scale=0.19]{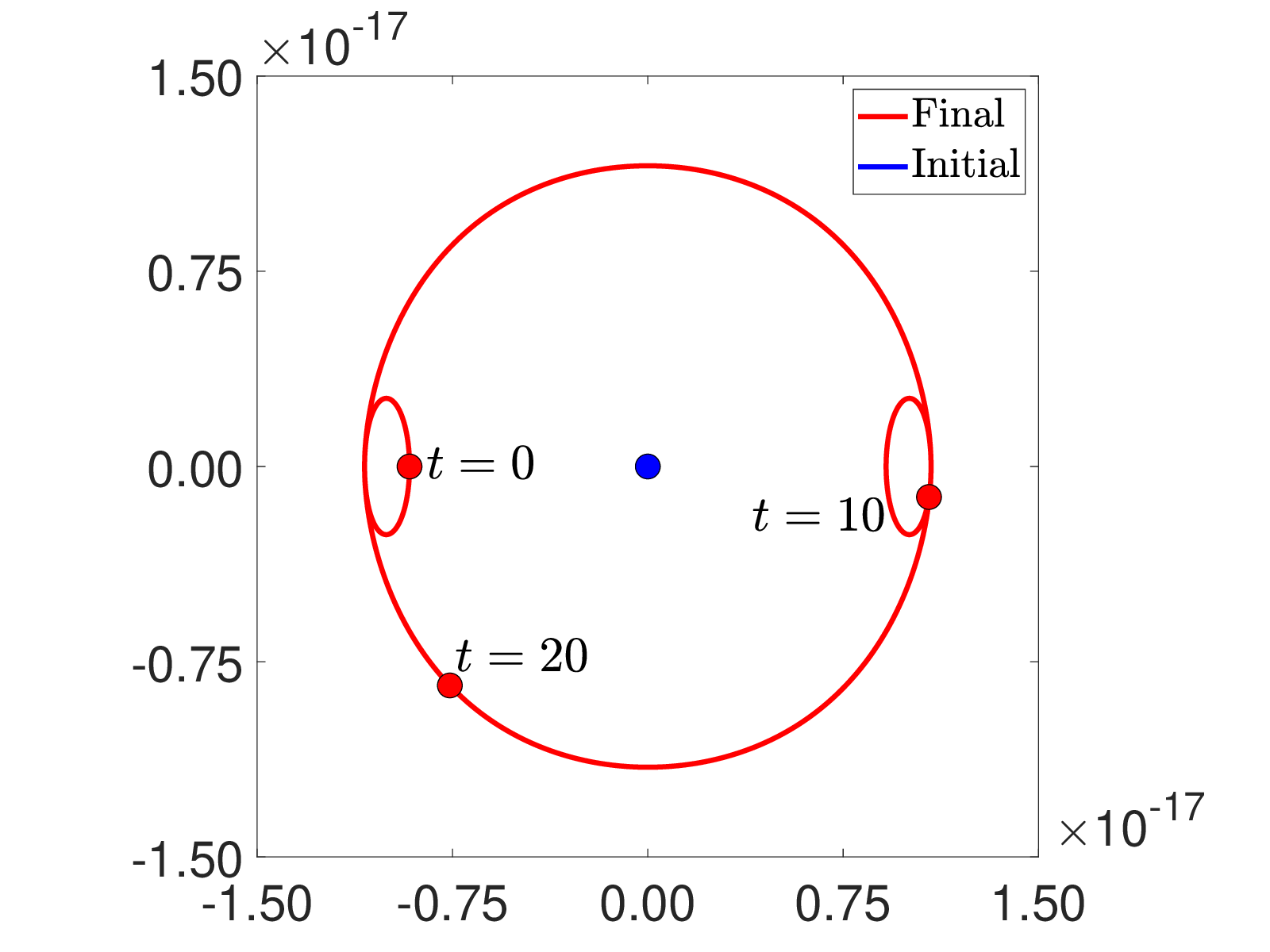}
    \end{subfigure}
    \begin{subfigure}[t]{0.325\linewidth}
    \centering
    \includegraphics[scale=0.19]{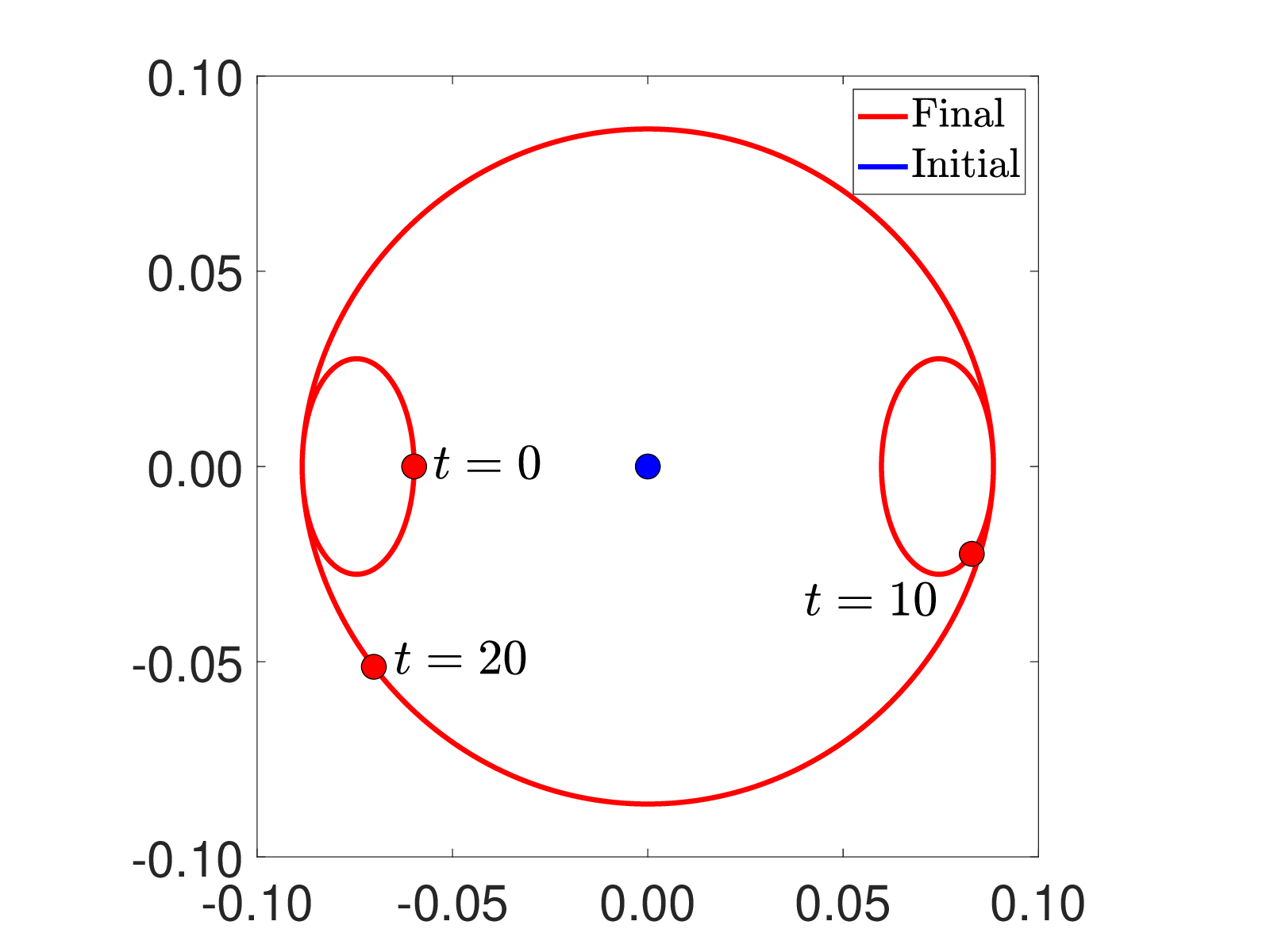}
    \end{subfigure} \\
    \begin{subfigure}[t]{0.325\linewidth}
    \centering
    \includegraphics[scale=0.19]{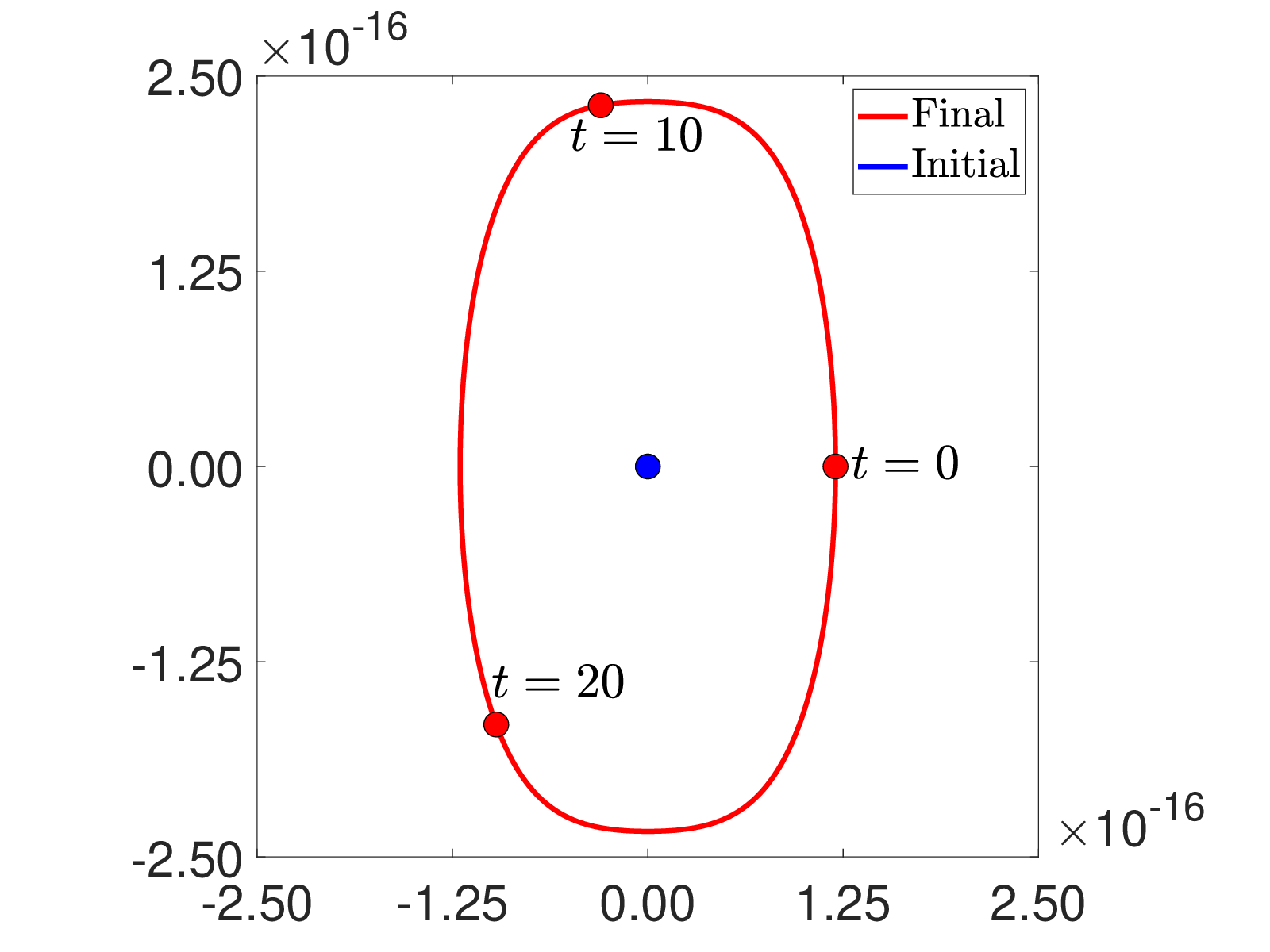}
    \end{subfigure}
    \begin{subfigure}[t]{0.325\linewidth}
    \centering
    \includegraphics[scale=0.19]{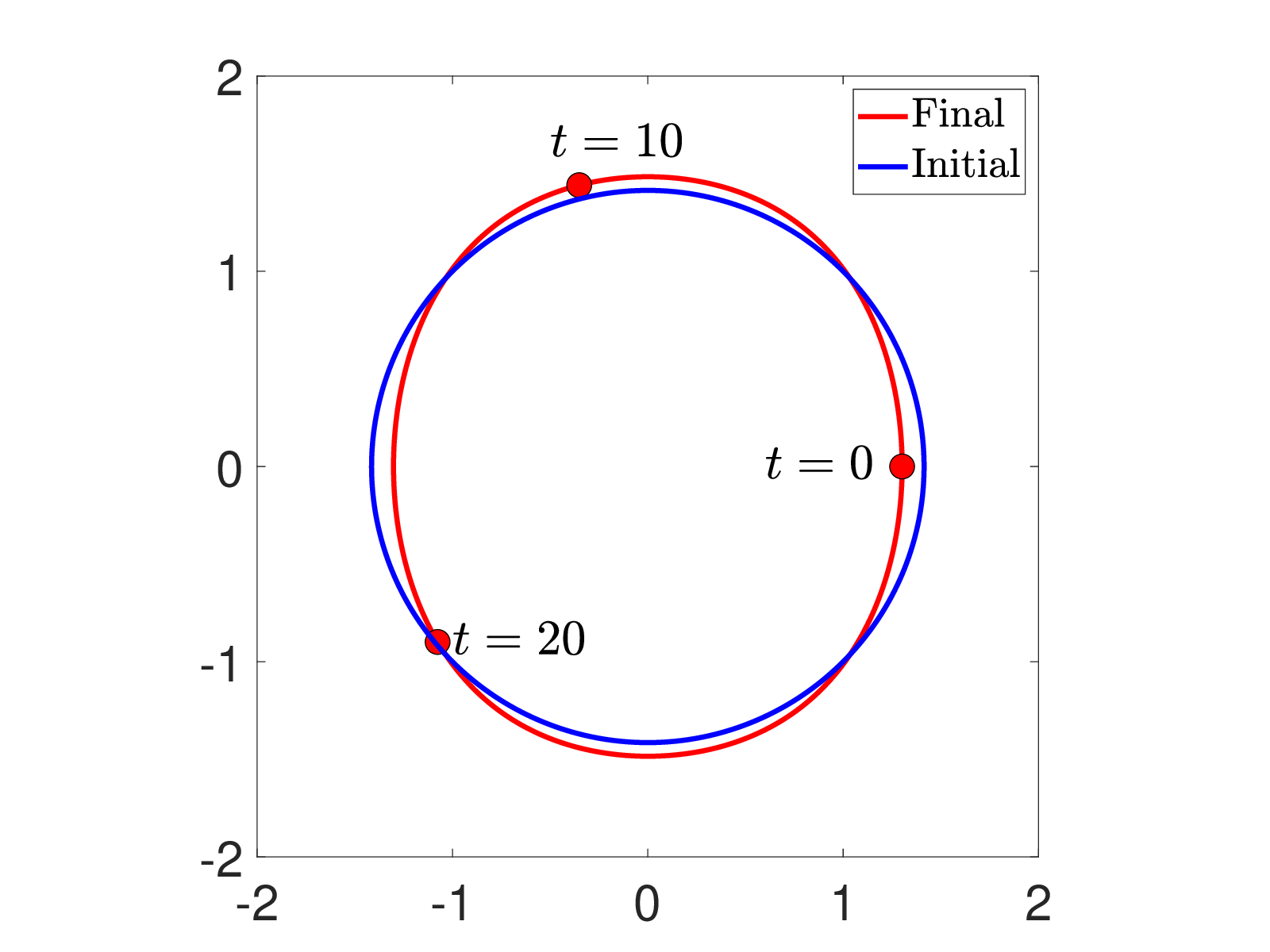}
    \end{subfigure}
    \begin{subfigure}[t]{0.325\linewidth}
    \centering
    \includegraphics[scale=0.19]{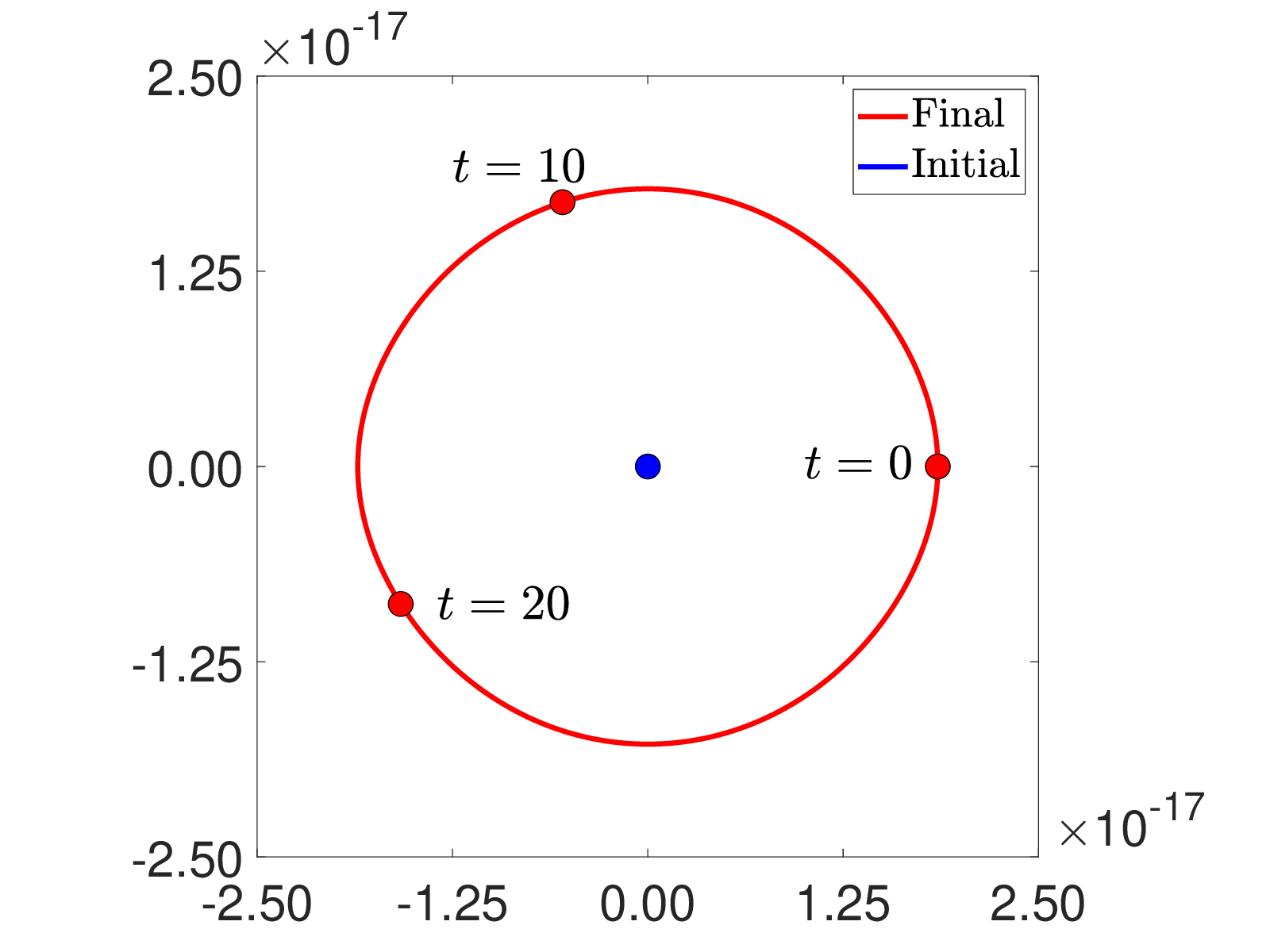}
    \end{subfigure} \\
    \begin{subfigure}[t]{0.325\linewidth}
    \centering
    \includegraphics[scale=0.19]{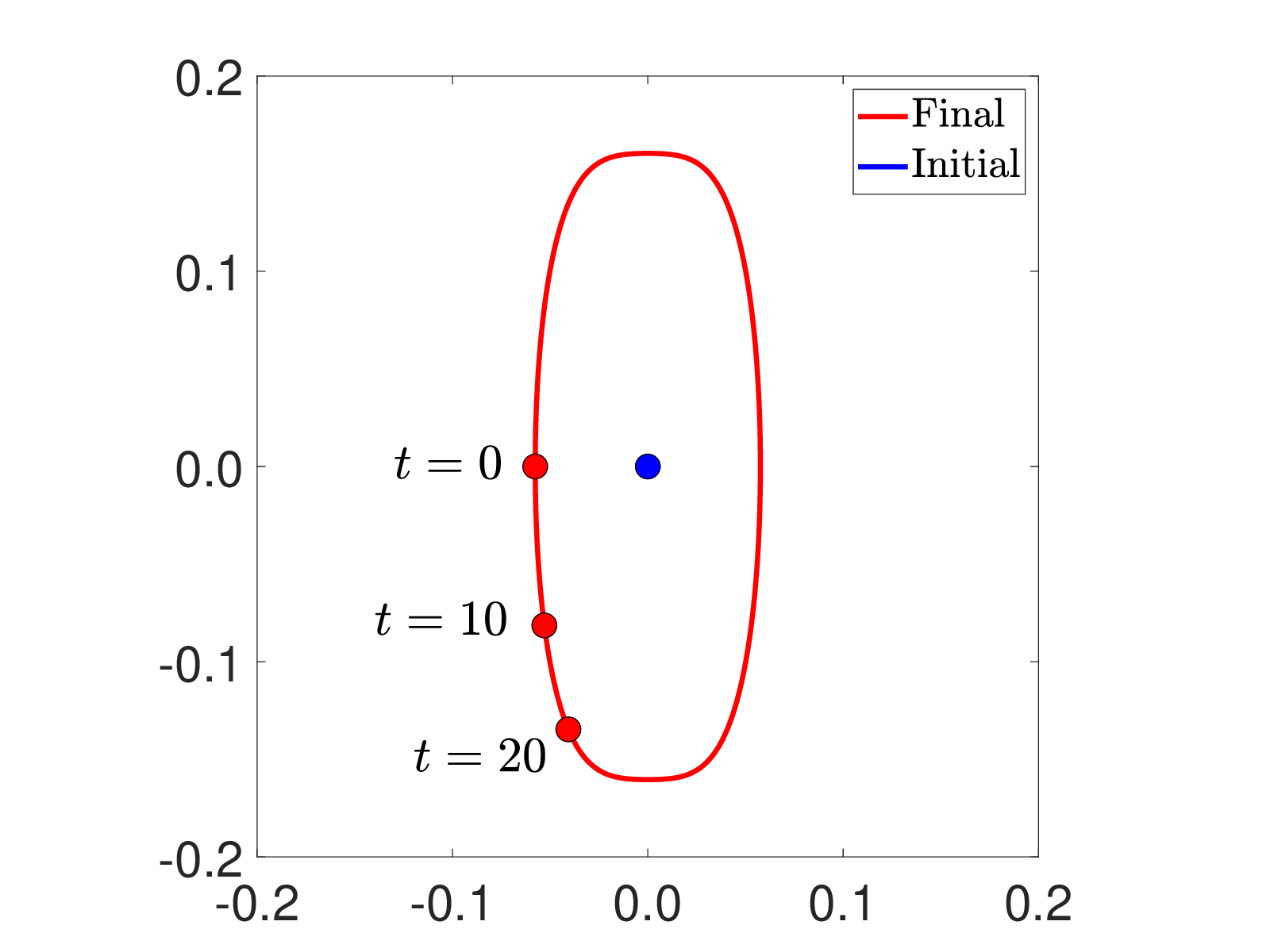}
    \caption{$(x_1, y_1)$}
    \end{subfigure}
    \begin{subfigure}[t]{0.325\linewidth}
    \centering
    \includegraphics[scale=0.19]{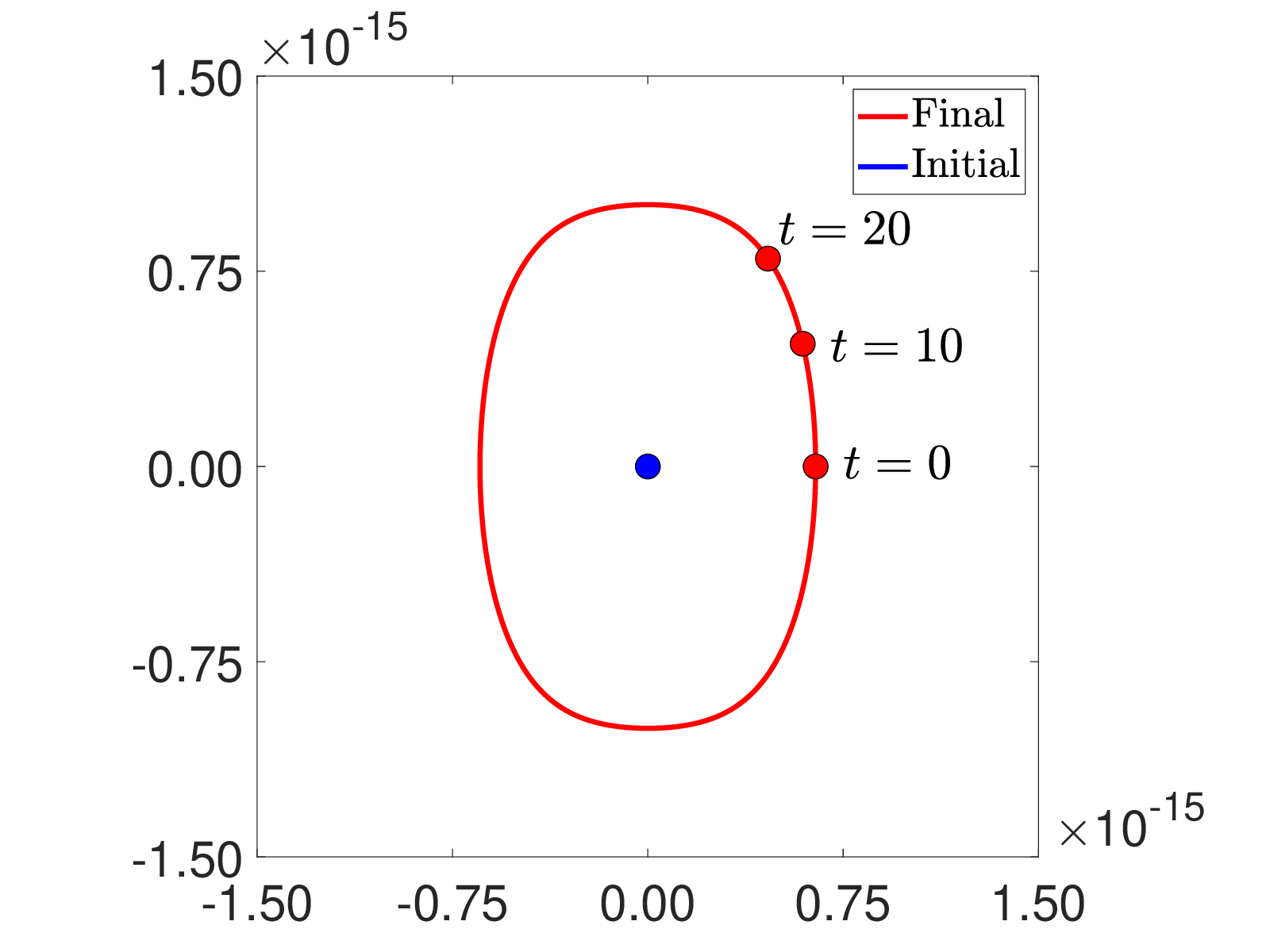}
    \caption{$(x_2, y_2)$}
    \end{subfigure}
    \begin{subfigure}[t]{0.325\linewidth}
    \centering
    \includegraphics[scale=0.19]{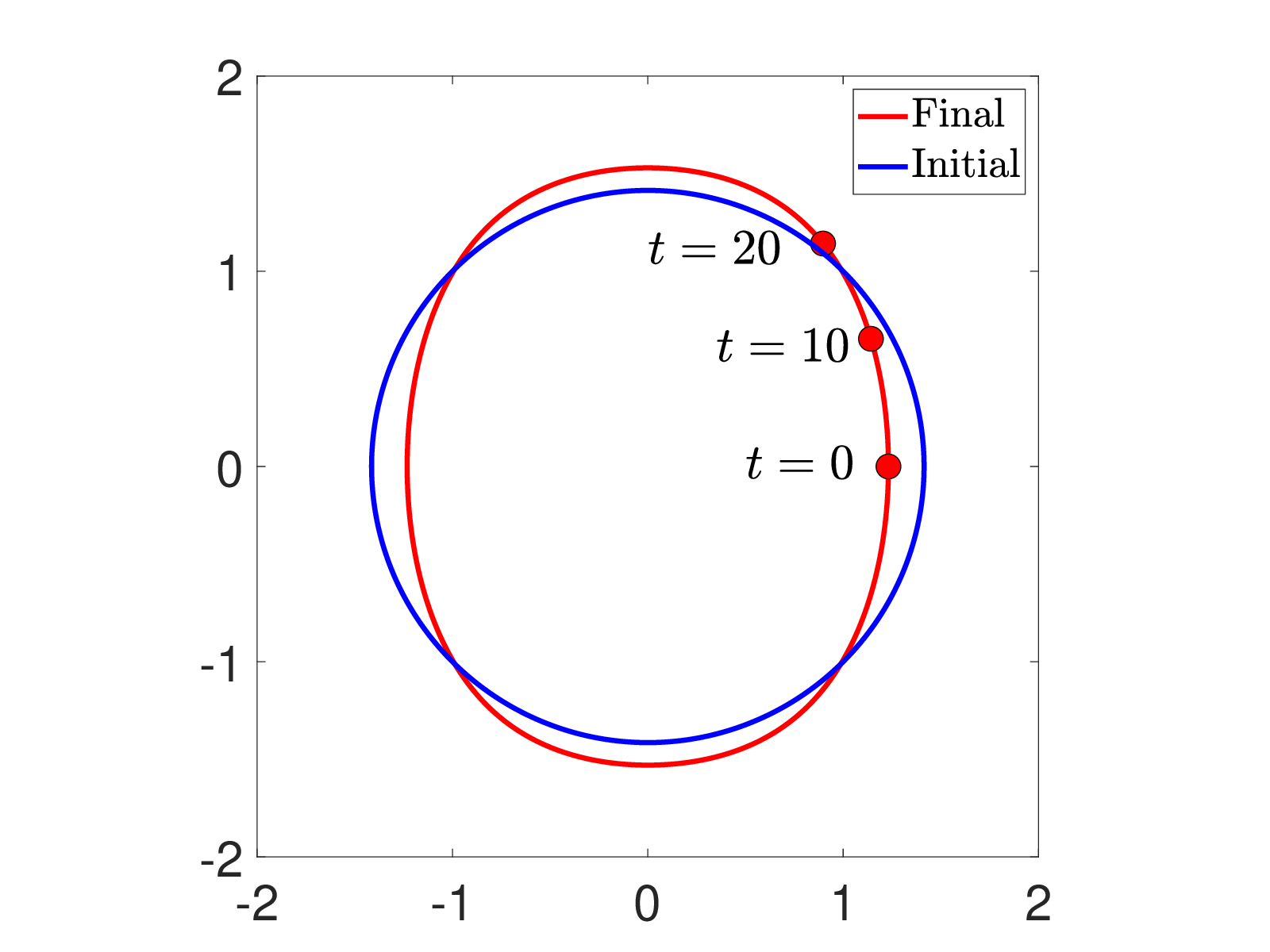}
    \caption{$(x_3, y_3)$}
    \end{subfigure}
    \caption{Phase space trajectories of the FPU model with markers indicating specific time points at $t = 0$, $10$,  and $20$. The panels represent different initial conditions: (top) $a = (1, 0, 0)$; (middle) $a=(0, 1, 0)$; (bottom) $a = (0, 0, 1)$.} 
    \label{fig: FPU-1d-traj}
\end{figure}
The dynamical evolution depends heavily on which mode is selected for the tangent 
frequency. 
Specifically, when the tangent frequency is fixed to either the first or third mode 
(located near the boundary), one normal mode undergoes significant perturbation 
while the other remains minimally affected. 
In the two instances, the two normal modes exhibit distinct motion behaviors. 
Conversely, when the tangent frequency is assigned to the second mode 
(situated in the middle), both normal modes experience only marginal perturbations. 
Overall, the coupling of these two normal frequencies fosters more chaotic long-term 
dynamics, increasing the complexity of the system's quasi-periodic behavior.

The convergence performance for the three cases is illustrated 
in~\Cref{fig: FPU-1d-con}. 
To evaluate  the efficiency of the proposed alternating 
scheme~\eqref{eqn: numerical-scheme}, we also monitor the three metrics: the Fourier 
coefficient error $\| \hat{z}^{(r+1)} - \hat{z}^{(r)} \|_2$, the frequency error 
$|\omega^{r+1} - \omega^{(r)}|$, and the pointwise solution error at $t = 10$, 
$|z^{r+1}(10) - z^{(r)}(10)|$. 
\begin{figure}[htb!]
    \centering
     \begin{subfigure}[t]{0.325\linewidth}
    \centering
    \includegraphics[scale=0.19]{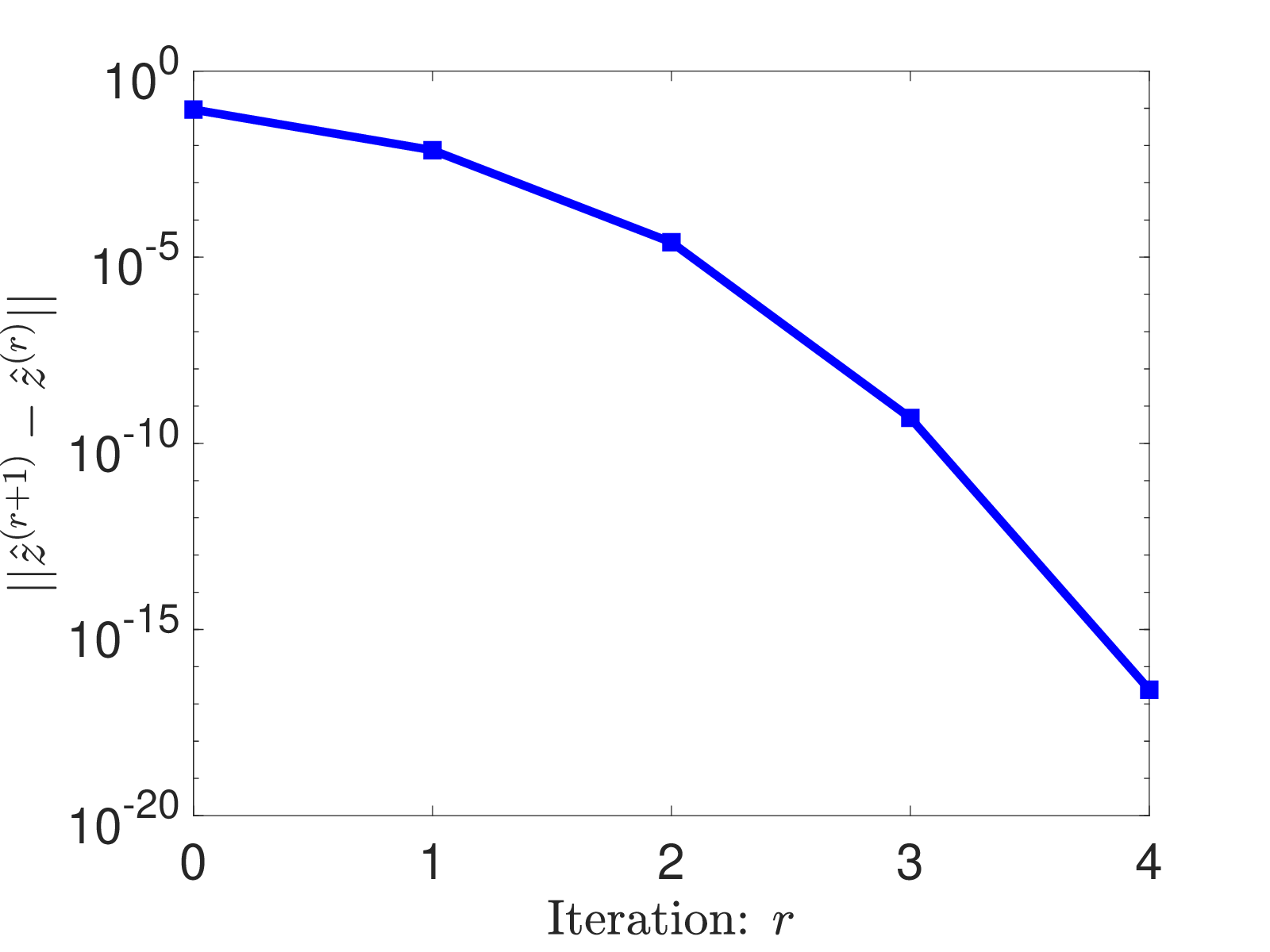}
    \end{subfigure}
    \begin{subfigure}[t]{0.325\linewidth}
    \centering
    \includegraphics[scale=0.19]{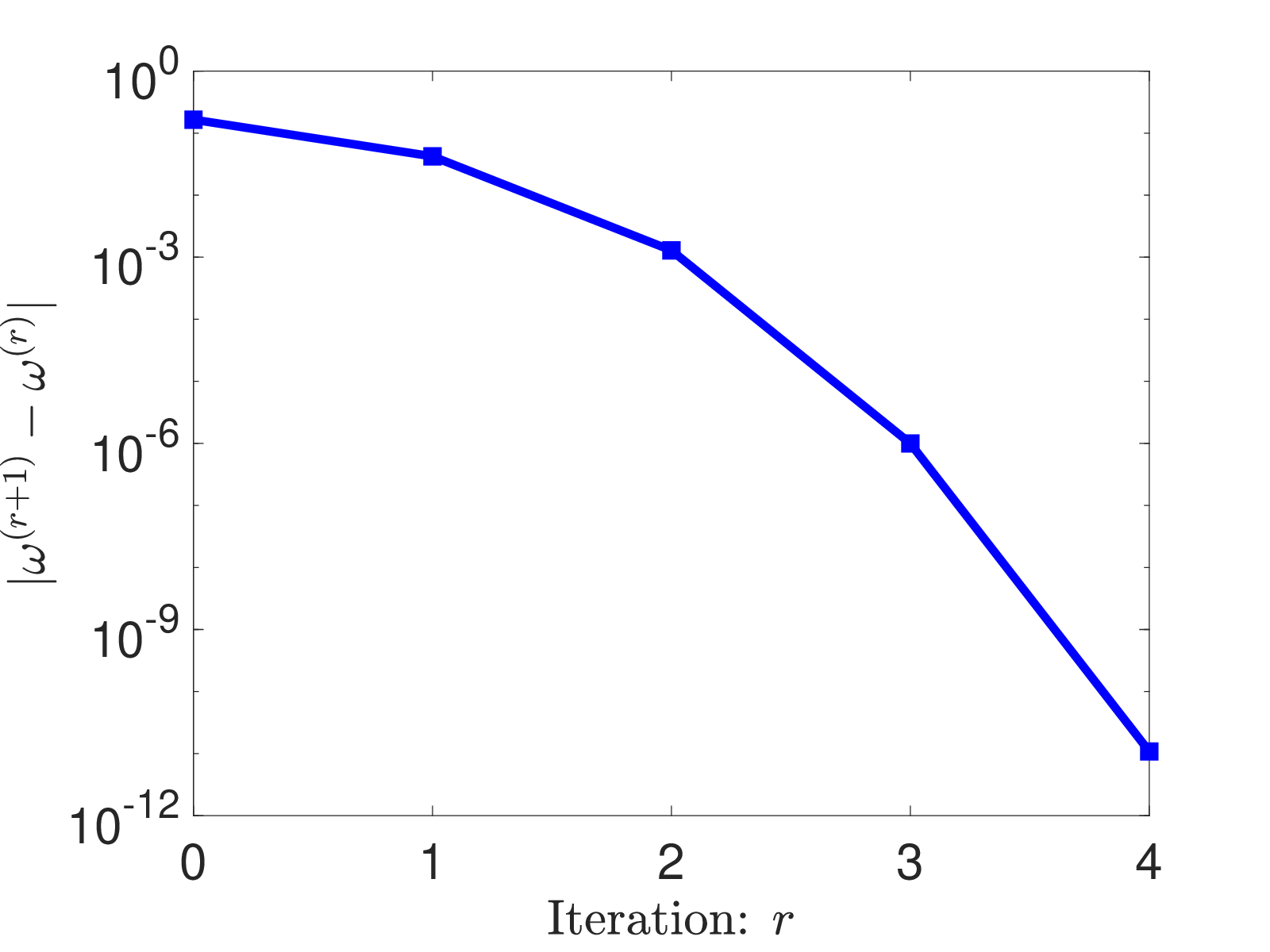}
    \end{subfigure}
    \begin{subfigure}[t]{0.325\linewidth}
    \centering
    \includegraphics[scale=0.19]{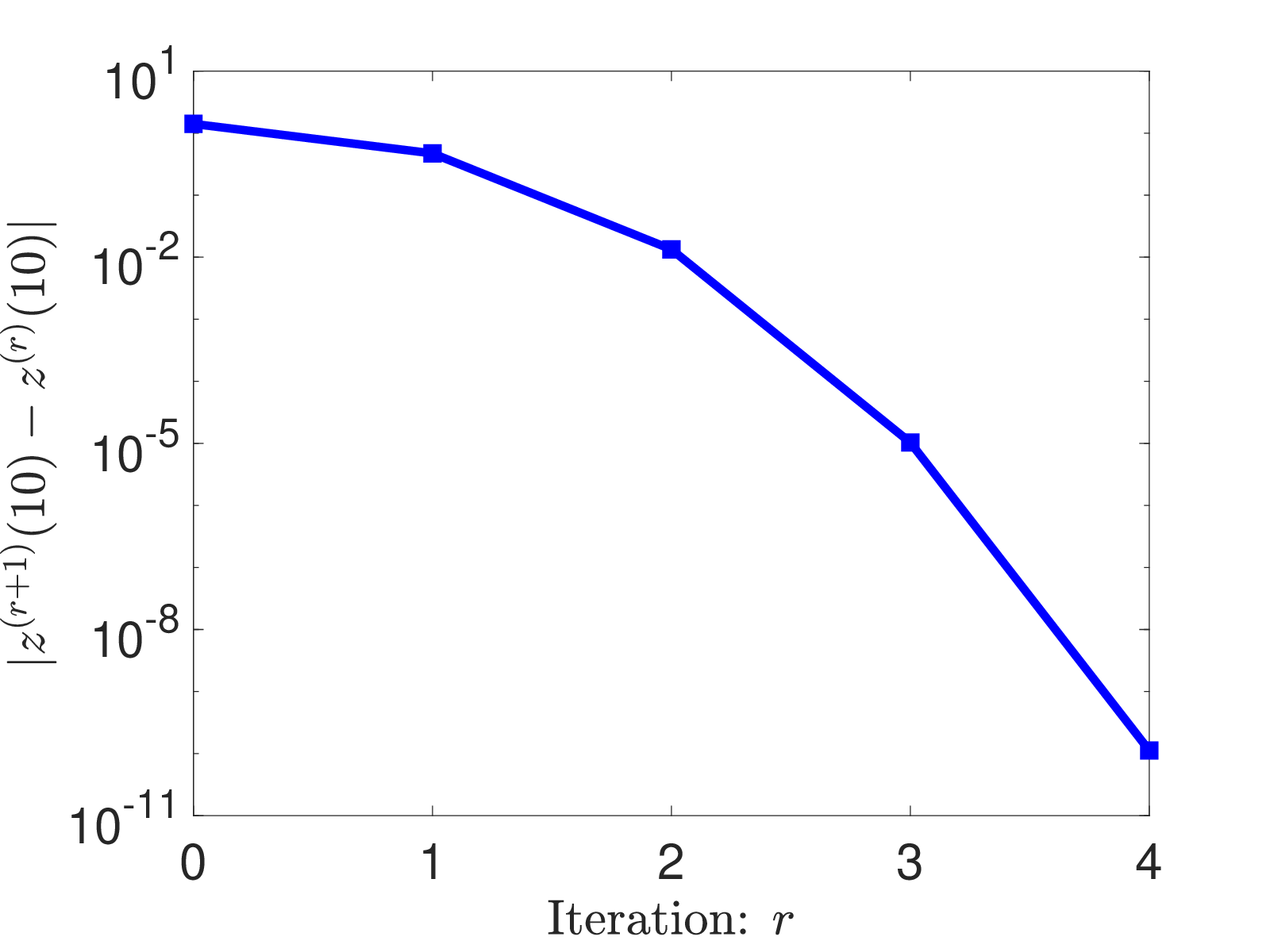}
    \end{subfigure} \\
    \begin{subfigure}[t]{0.325\linewidth}
    \centering
    \includegraphics[scale=0.19]{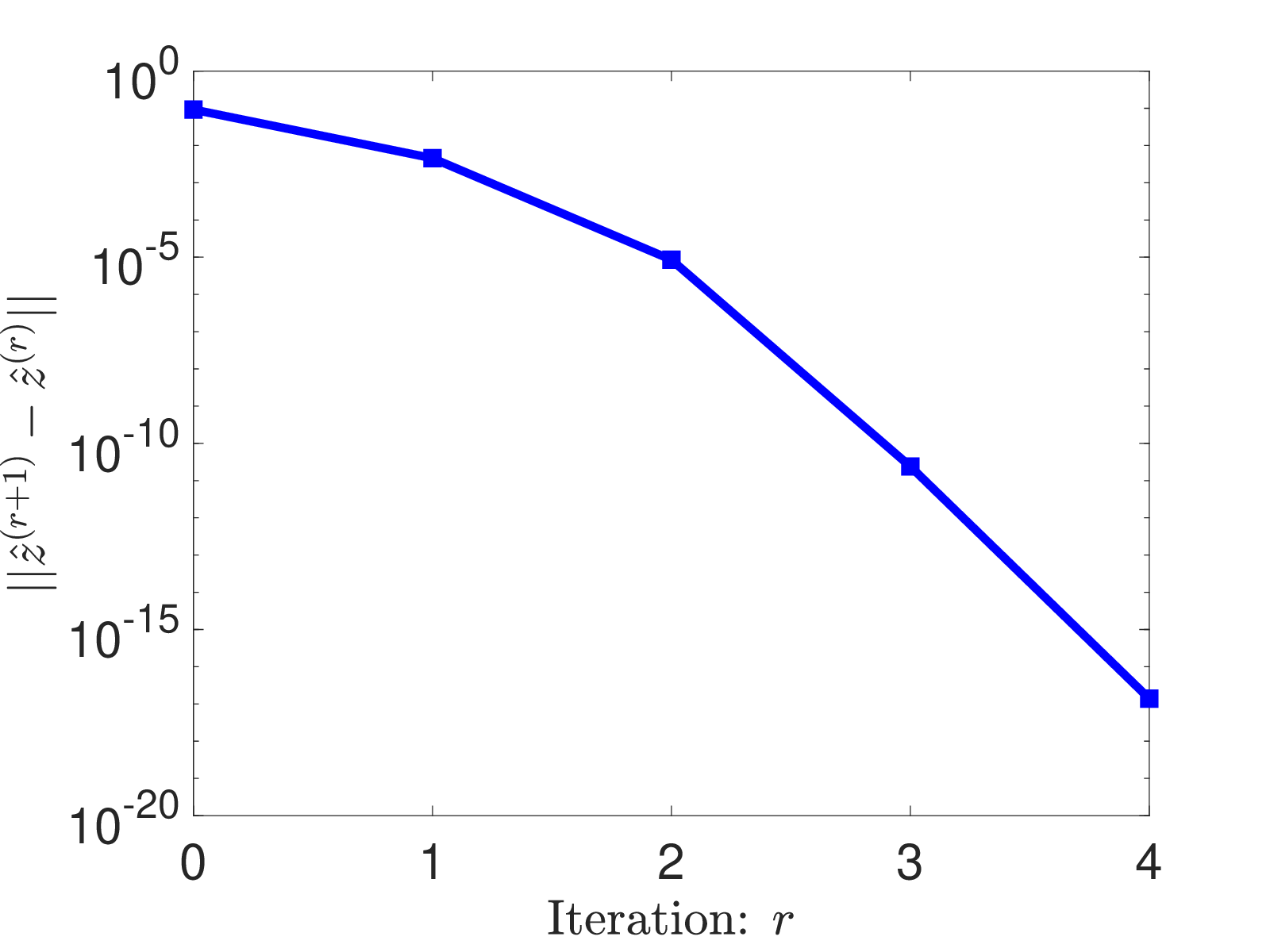}
    \end{subfigure}
    \begin{subfigure}[t]{0.325\linewidth}
    \centering
    \includegraphics[scale=0.19]{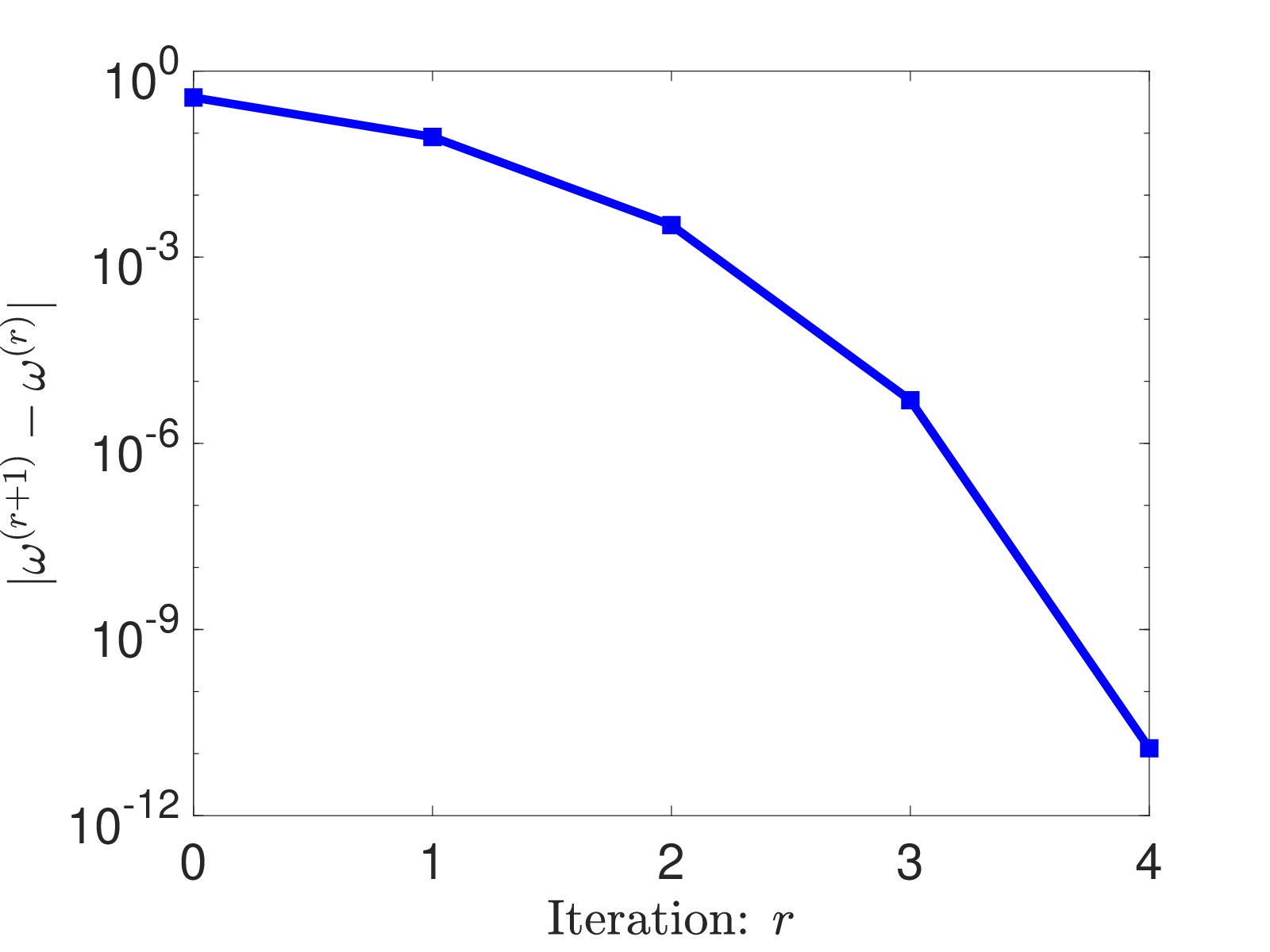}
    \end{subfigure}
    \begin{subfigure}[t]{0.325\linewidth}
    \centering
    \includegraphics[scale=0.19]{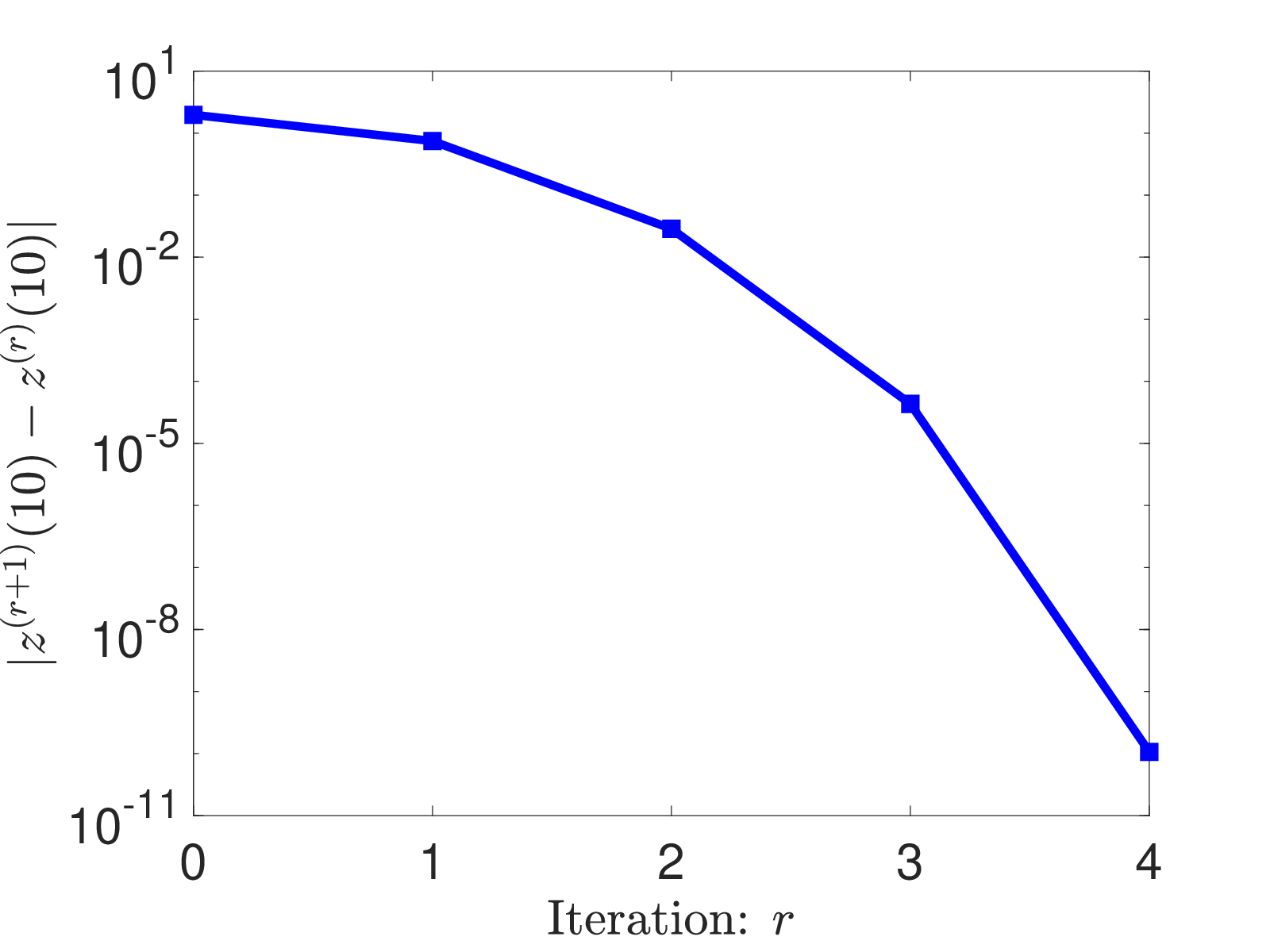}
    \end{subfigure} \\
    \begin{subfigure}[t]{0.325\linewidth}
    \centering
    \includegraphics[scale=0.19]{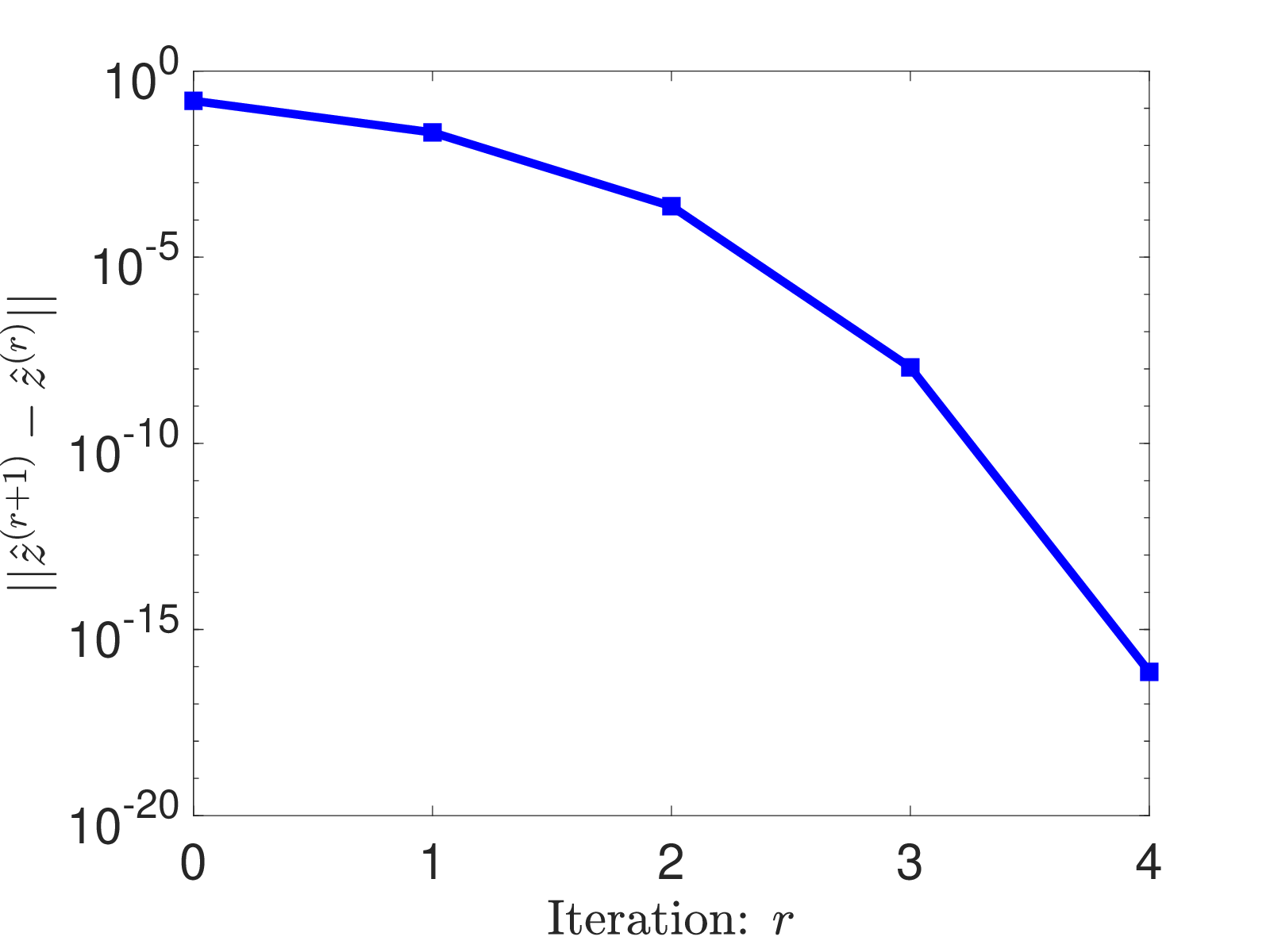}
    \caption{Fourier coefficient vectors}
    \end{subfigure}
    \begin{subfigure}[t]{0.325\linewidth}
    \centering
    \includegraphics[scale=0.19]{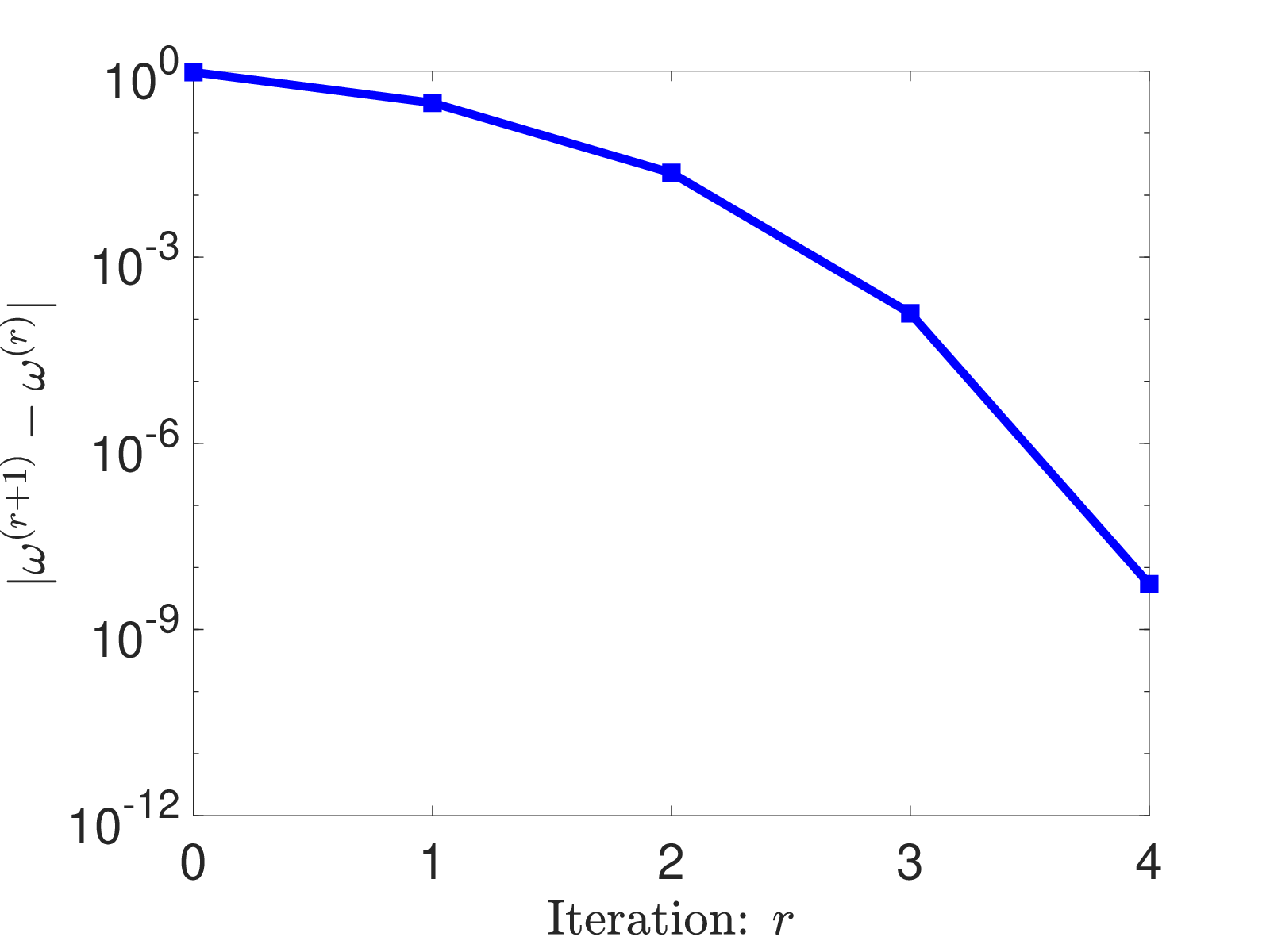}
    \caption{Frequency error}
    \end{subfigure}
    \begin{subfigure}[t]{0.325\linewidth}
    \centering
    \includegraphics[scale=0.19]{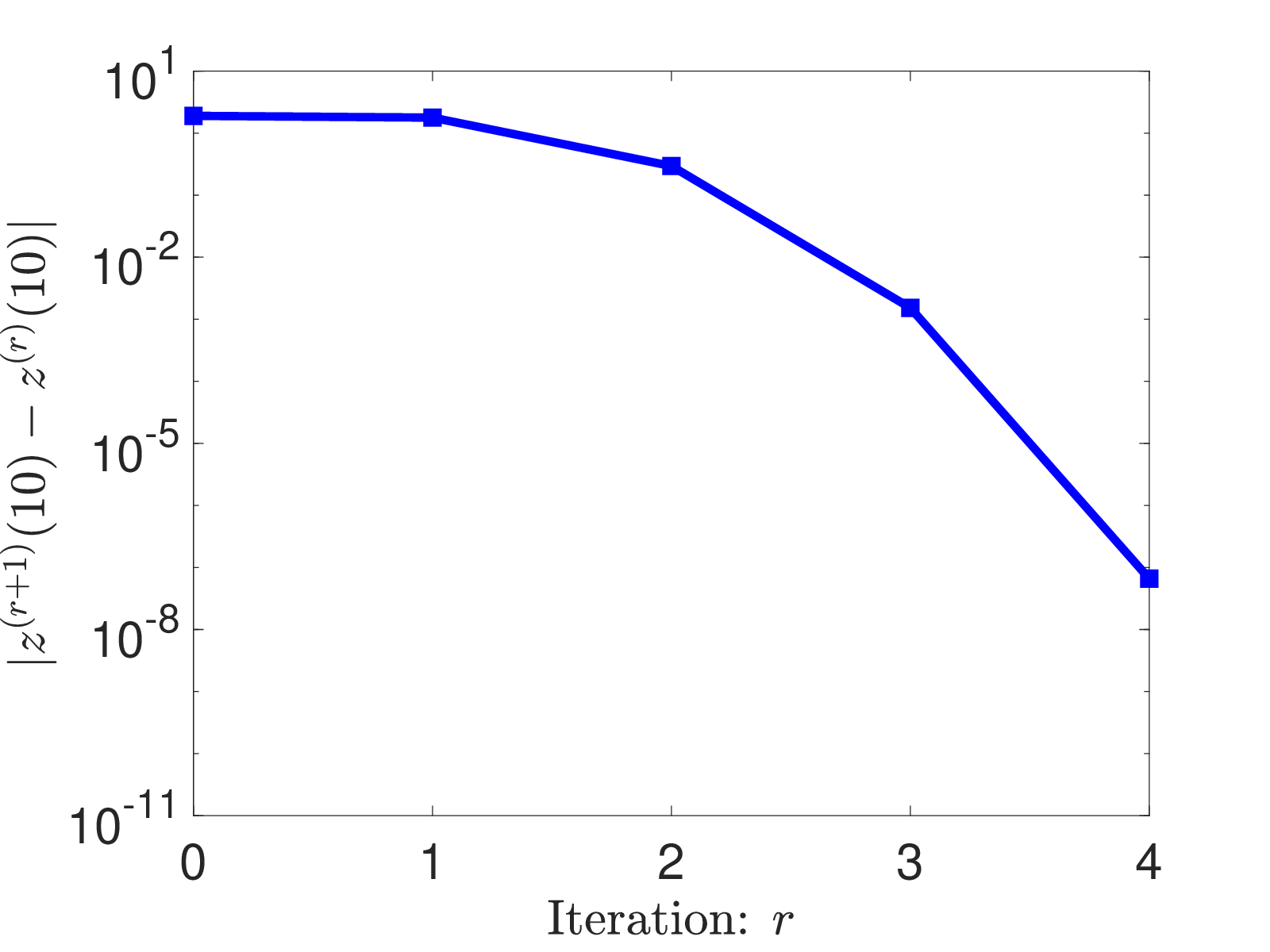}
    \caption{State at $t=10$}
    \end{subfigure}
    \caption{Convergence profiles across iterations for the FPU model. The panels represent different initial conditions: (top) $a = (1, 0, 0)$; (middle) $a=(0, 1, 0)$; (bottom) $a = (0, 0, 1)$.}
    \label{fig: FPU-1d-con}
\end{figure}
As illustrated, all three indicators consistently exhibit super-exponential decay. 
This rapid convergence not only validates the theoretical efficiency of the 
iterative scheme but also highlights its practical utility in achieving 
high-precision results while bypassing the cumulative phase errors associated with 
symplectic integrators.

%======================================================%
\subsubsection{Multi-frequency quasi-periodic solutions}
\label{subsubsec: fpu-2d}

In this section, we extend the alternating numerical 
procedure~\eqref{eqn: numerical-scheme} to compute multi-frequency quasi-periodic 
solutions. By setting $\varepsilon = 0.1$, we enhance the visibility of the 
trajectories within the phase space. Unlike the periodic solutions depicted 
in~\Cref{fig: henon-heiles-traj} and~\Cref{fig: FPU-1d-traj}, which appear as 
perfectly closed curves in phase space, the trajectories in~\Cref{fig: FPU-2d-traj} 
characterize quasi-periodic motion. 
Although these higher-dimensional trajectories never return to their exact starting 
states, they remain densely confined within the immediate vicinity of their linear 
unperturbed counterparts, providing visual evidence of the persistence and nonlinear 
stability of the underlying invariant structures.
\begin{figure}[htb!]
    \centering
    \begin{subfigure}[t]{0.325\linewidth}
    \centering
    \includegraphics[scale=0.19]{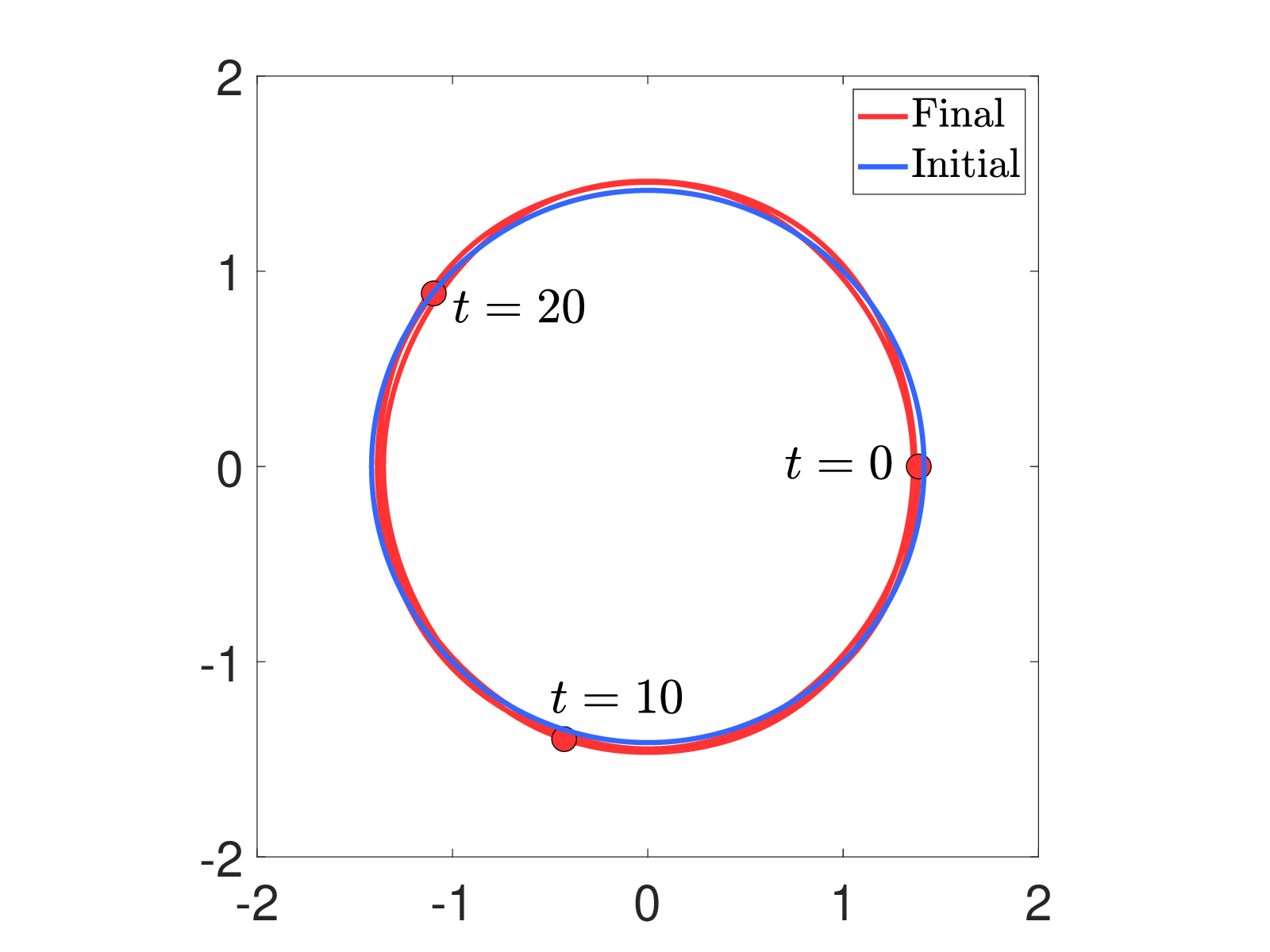}
    \end{subfigure}
    \begin{subfigure}[t]{0.325\linewidth}
    \centering
    \includegraphics[scale=0.19]{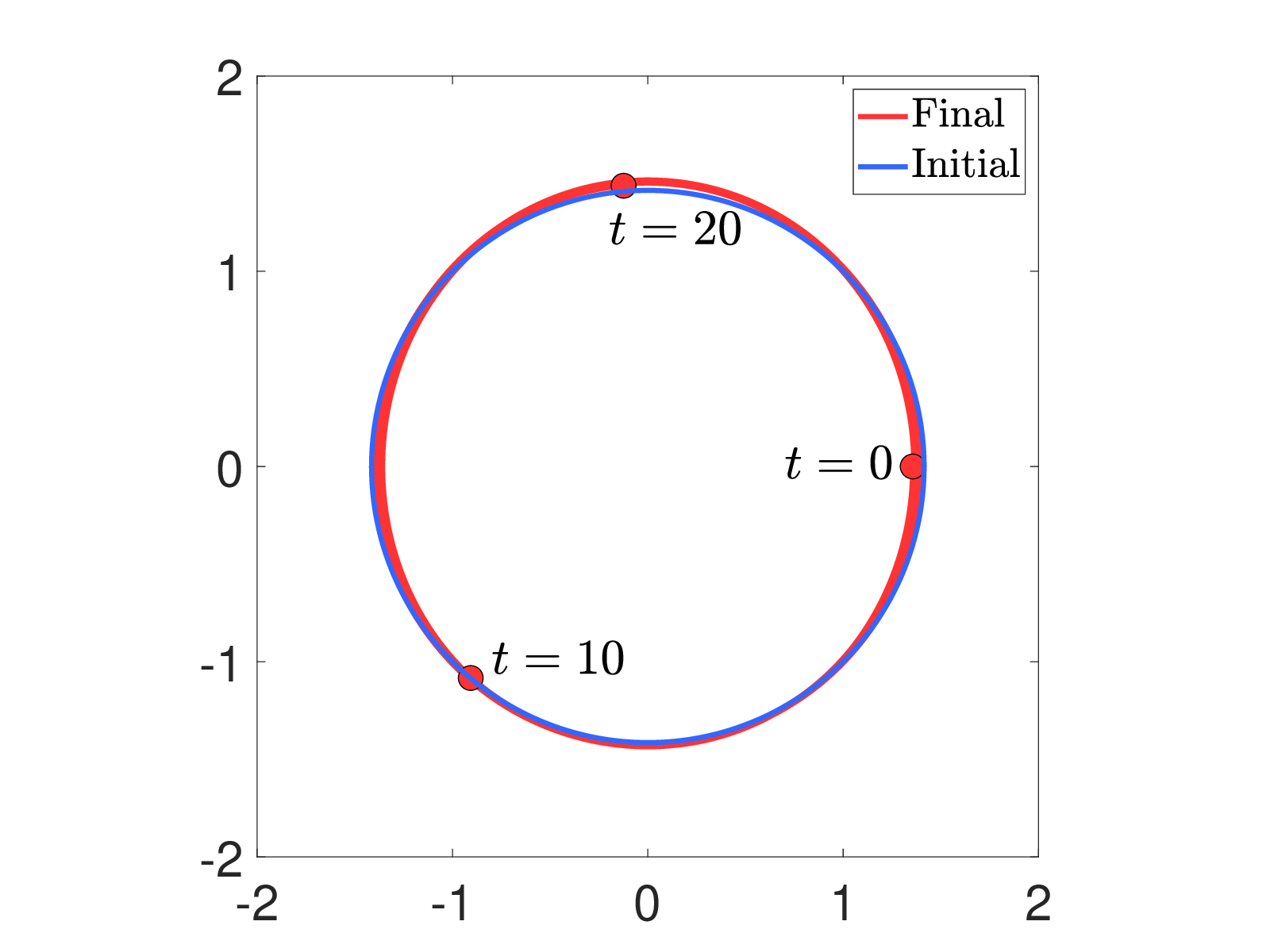}
    \end{subfigure}
    \begin{subfigure}[t]{0.325\linewidth}
    \centering
    \includegraphics[scale=0.19]{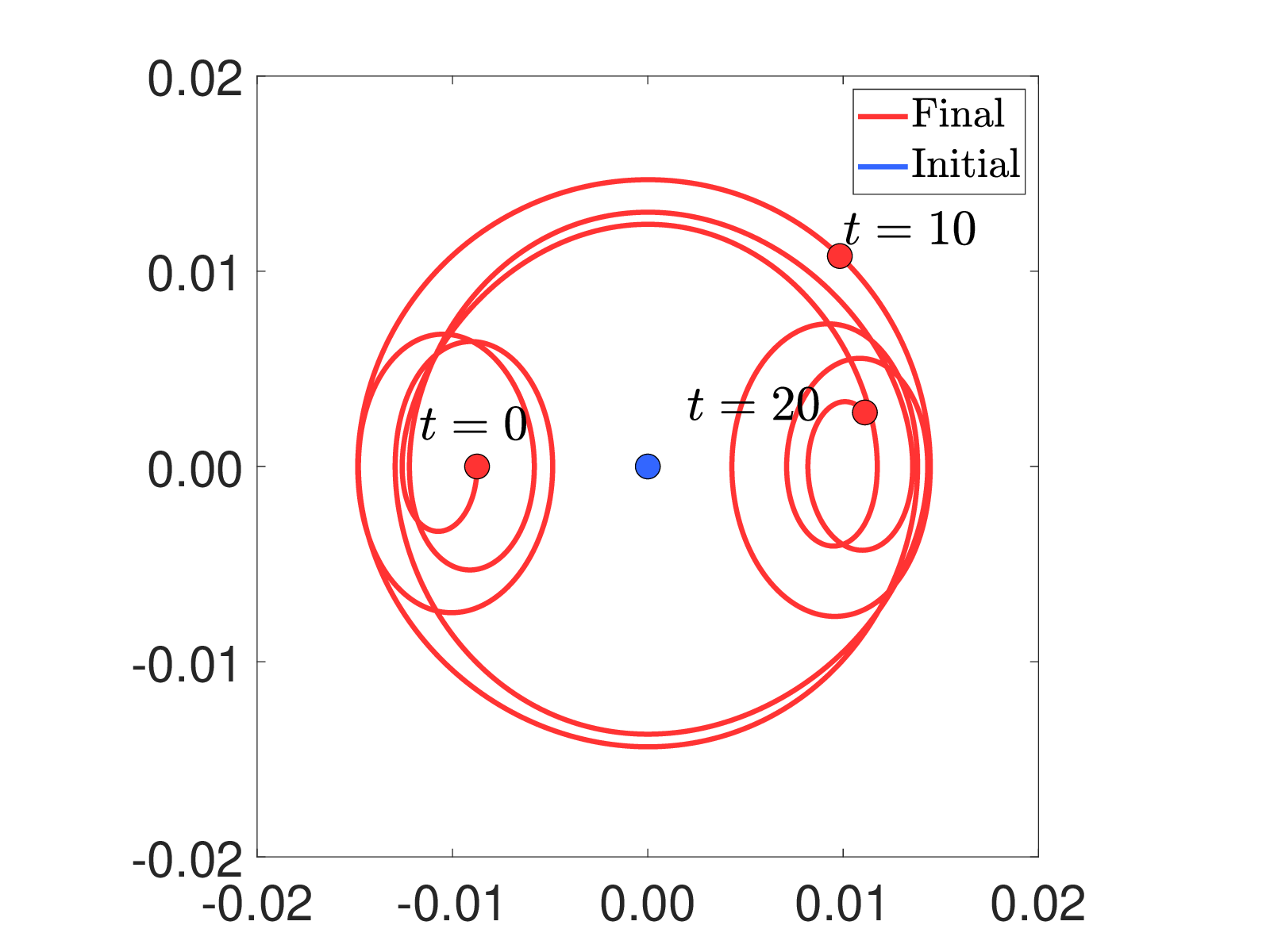}
    \end{subfigure} \\
    \begin{subfigure}[t]{0.325\linewidth}
    \centering
    \includegraphics[scale=0.19]{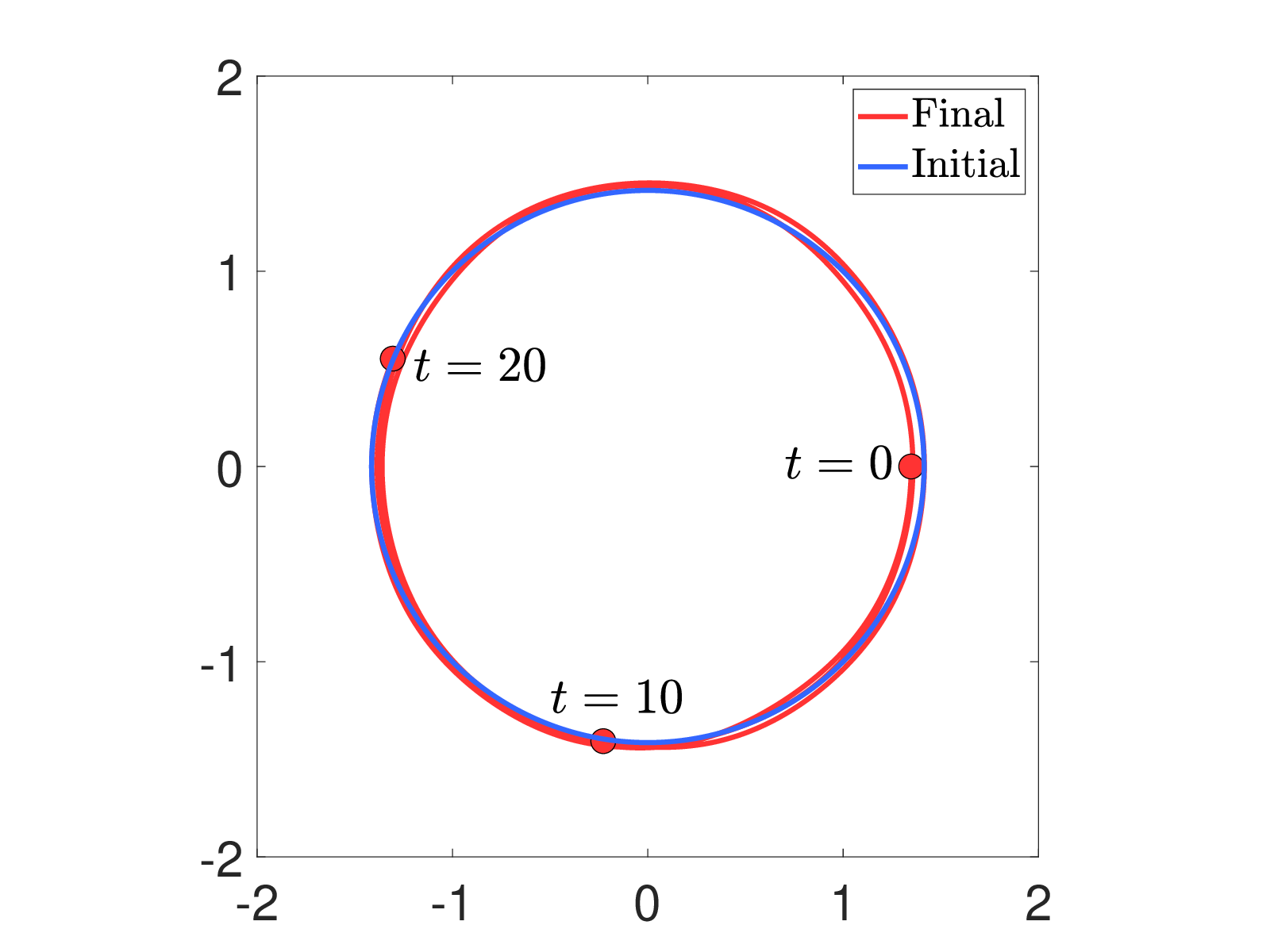}
    \end{subfigure}
    \begin{subfigure}[t]{0.325\linewidth}
    \centering
    \includegraphics[scale=0.19]{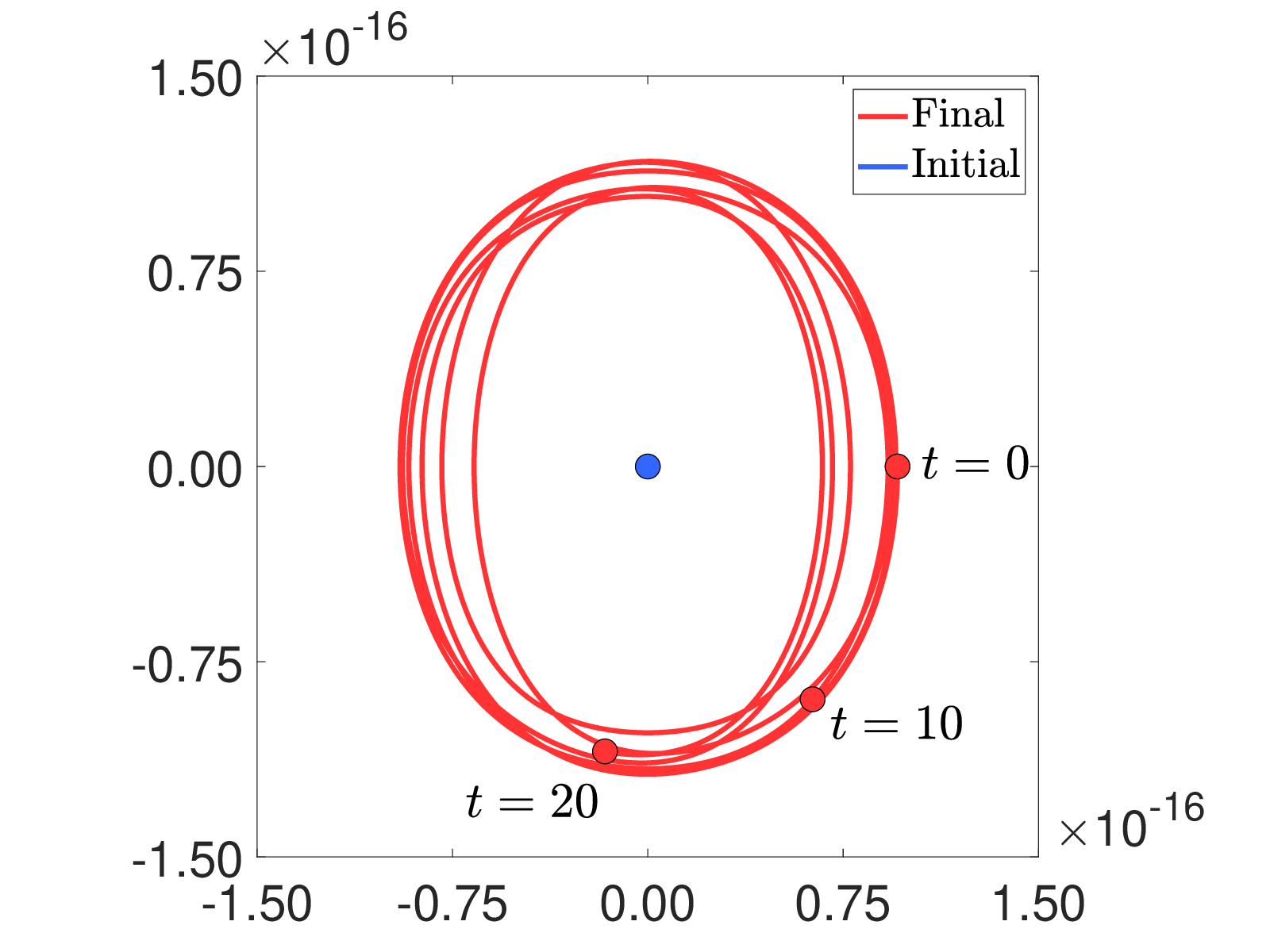}
    \end{subfigure}
    \begin{subfigure}[t]{0.325\linewidth}
    \centering
    \includegraphics[scale=0.19]{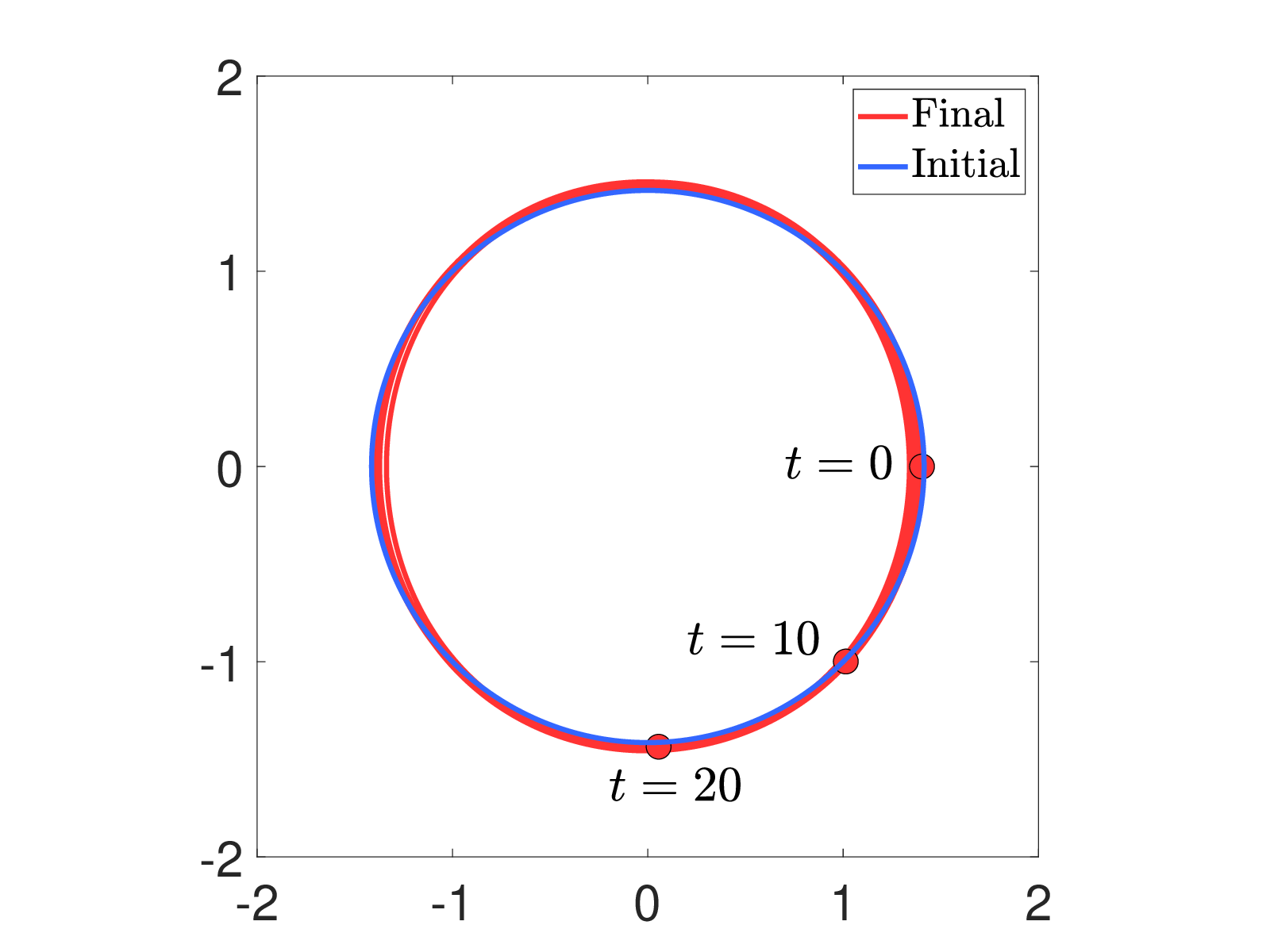}
    \end{subfigure} \\
    \begin{subfigure}[t]{0.325\linewidth}
    \centering
    \includegraphics[scale=0.19]{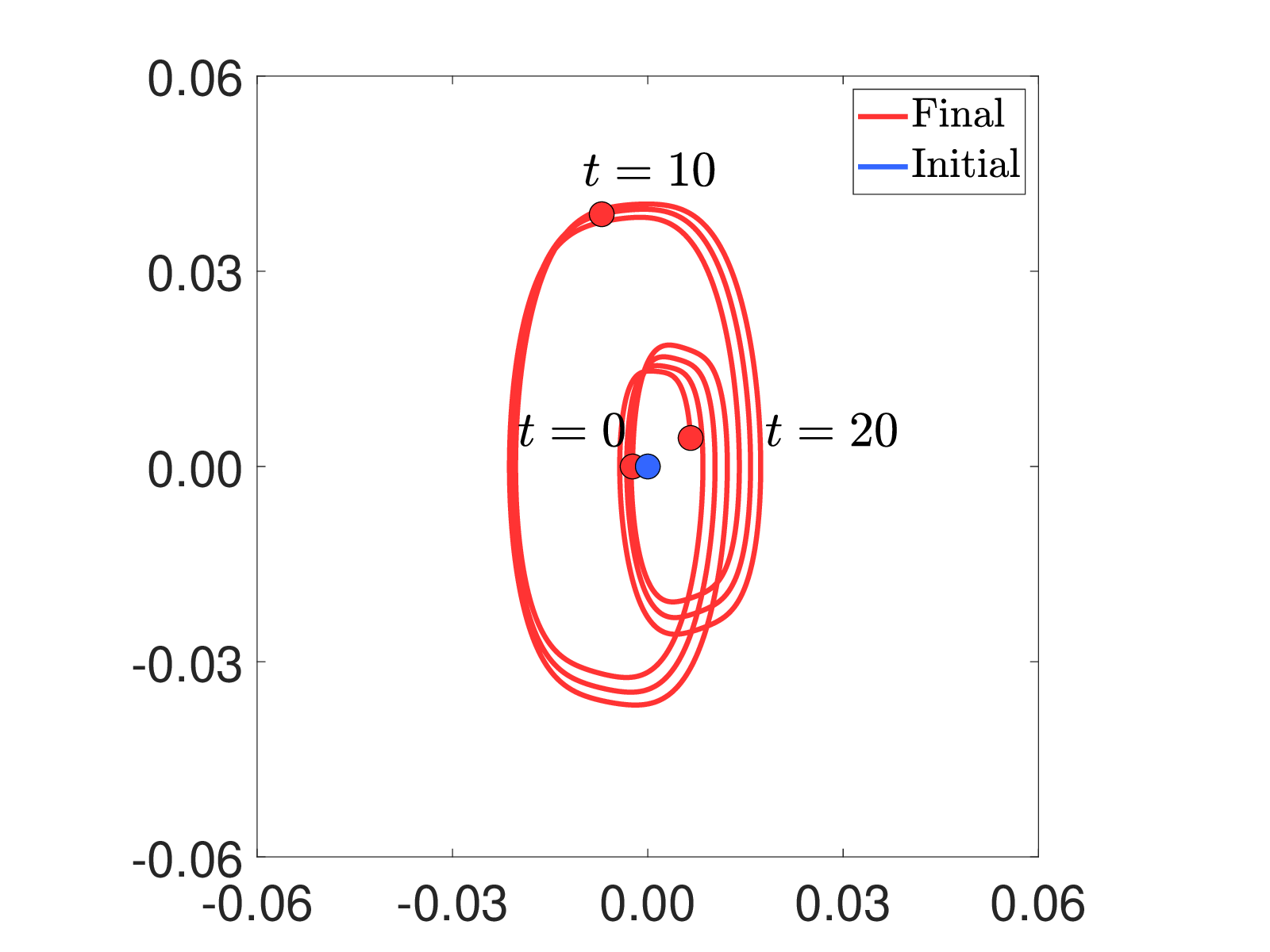}
    \caption{$(x_1, y_1)$}
    \end{subfigure}
    \begin{subfigure}[t]{0.325\linewidth}
    \centering
    \includegraphics[scale=0.19]{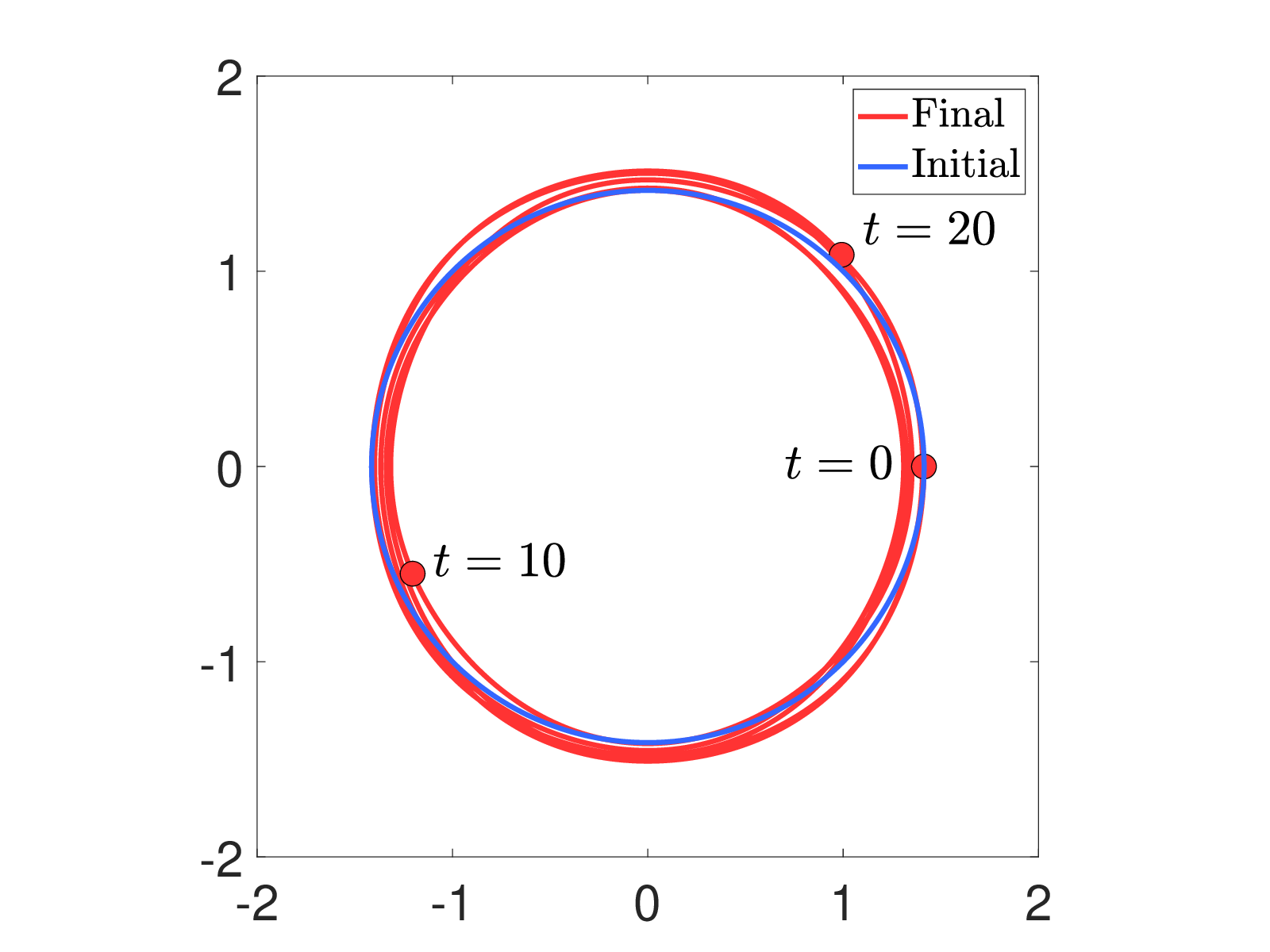}
    \caption{$(x_2, y_2)$}
    \end{subfigure}
    \begin{subfigure}[t]{0.325\linewidth}
    \centering
    \includegraphics[scale=0.19]{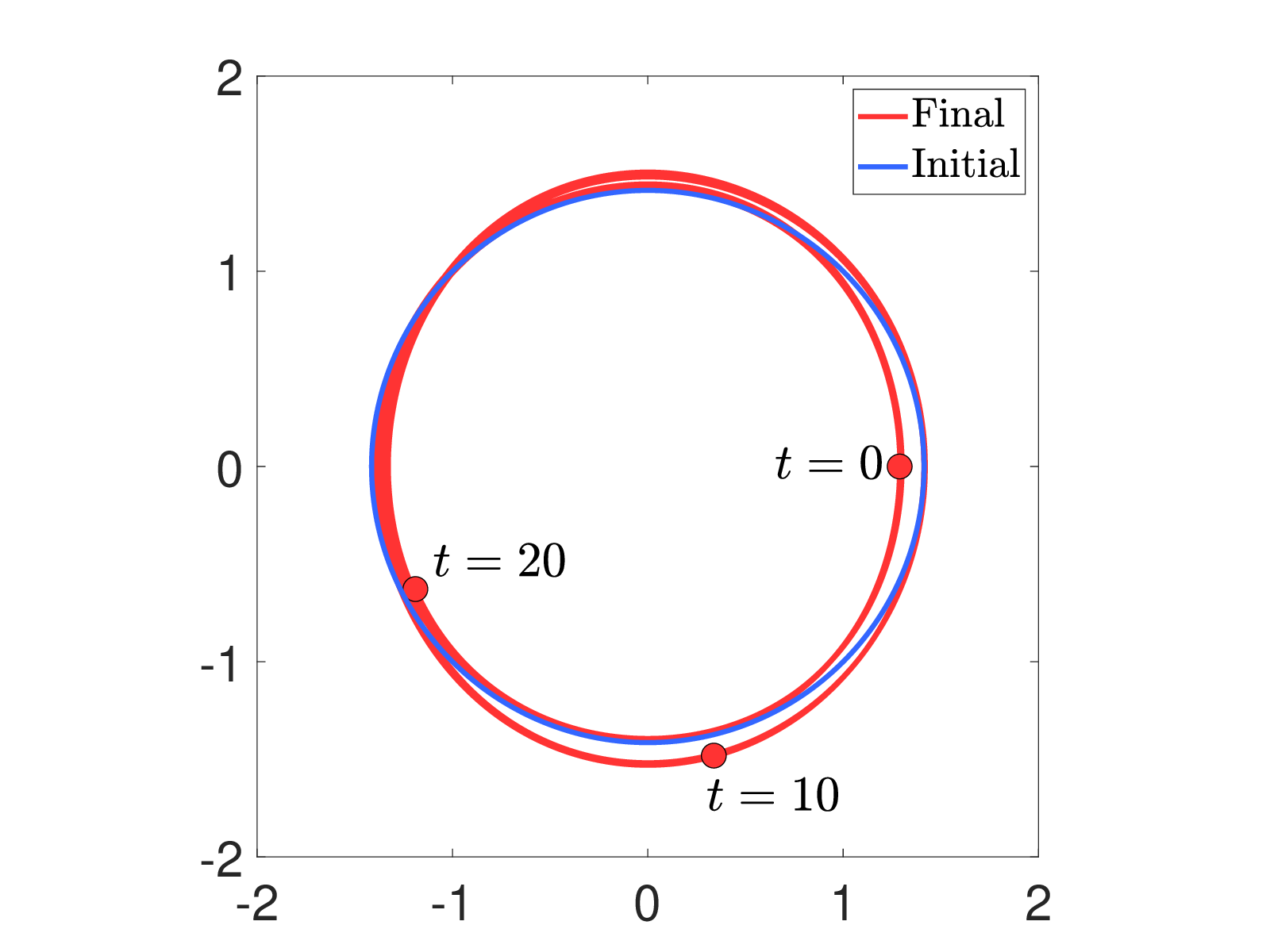}
    \caption{$(x_3, y_3)$}
    \end{subfigure}
    \caption{Phase space trajectories of the FPU model with markers indicating specific time points at $t = 0$, $10$,  and $20$. The panels represent different initial conditions: (top) $a = (1, 1, 0)$; (middle) $a=(1, 0, 1)$; (bottom) $a = (0, 1, 1)$.}
    \label{fig: FPU-2d-traj}
\end{figure}
A notable distinction emerges between the tangential and normal subspaces: while the 
macroscopic geometric features in the tangential subspaces remain nearly invariant 
across different active frequencies, indicating a degree of robustness in the 
dominant dynamics, the microscopic residual fluctuations in the normal subspaces 
display markedly different geometric patterns. 
This pronounced variation in the transverse directions suggests the presence of 
potentially unexplored mathematical or physical mechanisms.

\section{Conclusion and further work}
\label{sec: conclu}

In this study, we extend the alternating numerical procedure, originally proposed 
by~\citet{fu2026numerical}, to investigate elliptic lower-dimensional quasi-periodic 
solutions. 
By conducting a comparative analysis of two Newton-based iterative methods, the KAM 
scheme and the CWB scheme, we identify a potentially deeper underlying structure 
governing the quasi-periodic solutions of nearly integrable systems. 
For clarity and conciseness, we employ the full-dimensional case as an illustrative 
example. 
In an integrable system, quasi-periodic solutions are represented by the linear 
solution~\eqref{eqn: linear-near-soln}. 
The KAM scheme operates by subtracting this entire linear component from the total 
solution, effectively isolating the perturbation from the unperturbed torus 
$\{0\} \times \mathbb{T}^n$ associated with the 
Hamiltonian~\eqref{eqn: near-integrable-1}. 
Under the assumption that this perturbation contains no linear terms, the resulting 
solution is defined relative to the frequency $\omega(I_0)$, simplified as $\omega$. 
Consequently, the total quasi-periodic solution in complex coordinates is 
reconstructed as:
\begin{equation}
    \label{eqn: kam-soln}
    z(t) = \underbrace{\sum_{ j=1}^{n} \sqrt{I_{0,j}} e^{i \omega_j t}}_{\mathrm{Linear\; part}} + \underbrace{f(\omega t) }_{\mathrm{Perturbation}}. 
\end{equation}
In contrast, within the CWB scheme, the linear solution cannot be treated as the 
full initial condition. 
Instead, it must be considered as:
\[
    I(t) = I_0 - a^2, \quad \theta(t) = \omega(I_0 - a^2) t + \theta_0, 
\]
where $a^2 = (a_1^2, \ldots, a_n^2)^{\top} \neq 0$. 
The Hamiltonian for the perturbation solution becomes:
\[
    H = \langle \omega(I_0 - a^2), J + a^2 \rangle + \varepsilon H_1(J+a^2, \theta; I_0 - a^2). 
\]
Accordingly, the resulting quasi-periodic solution takes the form
\begin{equation}
    \label{eqn: CWB-solution}
    z(t) = \sum_{ j=1}^{n} \sqrt{I_{0,j} - a_j^2} e^{i \omega_j (I_0-a^2) t} + \sum_{ j=1}^{n} a_j \left( e^{i \omega_j'(I_0-a^2) t} - e^{i \omega_j(I_0-a^2) t} \right) + \sum_{k \notin \mathcal{S}} \hat{z}_k e^{i \langle k, \omega'(I_0-a^2) \rangle t}. 
\end{equation}
Intuitively, as the parameter $|a| \rightarrow 0$, the frequencies should satisfy 
$\omega(I_0 - a^2) \rightarrow \omega$ and $\omega' (I_0 - a^2) \rightarrow \omega$. 
If this convergence is rigorously established, the quasi-periodic solution from the 
CWB scheme~\eqref{eqn: CWB-solution} admits the simplified representation
\begin{equation}
    \label{eqn: cwb-soln-limit}
    z(t) = \underbrace{\sum_{ j=1}^{n} \sqrt{I_{0,j}} e^{i \omega_j  t}}_{\mathrm{Linear\; part}} + \underbrace{\sum_{k \notin \mathcal{S}} \hat{z}_k e^{i \langle k, \omega \rangle t}}_{\mathrm{Perturbation}}. 
\end{equation}
Comparing~\eqref{eqn: kam-soln} and~\eqref{eqn: cwb-soln-limit}, we can derive that 
the perturbation is given by 
\[
    f(\omega t) = \sum_{k \notin \mathcal{S}} \hat{z}_k e^{i \langle k, \omega \rangle t},
\]
which reveals that the perturbation involves only the coefficients of the 
non-resonant Fourier modes. 
Furthermore, as we transition from an integrable system to a nearly integrable one, 
the quasi-periodic solution evolves from a simple linear mode superposition into a 
complex periodic function of $\theta = \omega t$. 
An interesting direction for future research is to investigate the behavior of the 
coefficients of non-resonant Fourier modes across different perturbation Hamiltonians.

In the forthcoming work, we plan to demonstrate the application of our alternating 
numerical scheme to construct quasi-periodic solutions for nearly integrable systems 
involving infinite normal frequencies. 
Specifically, we first address the infinite-dimensional FPU model and weakly coupled 
lattice networks, as discussed by~\citet{aubry1997breathers}, for which the 
existence of invariant tori was established via the KAM technique 
by~\citet{yuan2002construction}. 
We then extend our investigation to nearly integrable systems governed by nonlinear 
partial differential equations (PDEs), including the nonlinear Schrödinger (NLS), 
nonlinear wave (NLW), and KdV equations~\citep{kuksin2000analysis, Kappeler2003}. 
Furthermore, we aim to present several previously undiscovered exact solutions, 
specifically for the NLS. It is also worth noting that for the NLS equation in 
higher spatial dimensions, a class of special $d$-dimensional quasi-periodic 
solutions was constructed via the KAM technique by~\citet{yuan2003quasi}.

\section*{Acknowledgements}
This work was partially supported by the NSFC (Grant No. 12241105) and by SIMIS (startup fund and cross-disciplinary research projects).

%We thank Bowen Li for his helpful and insightful discussions. Mingwei Fu acknowledges partial support from the Loo-Keng Hua Scholarship of CAS. Bin Shi was partially supported by the startup fund from SIMIS and Grant No.YSBR-034 of CAS.

\bibliographystyle{abbrvnat}
\bibliography{sigproc}

%===================================================================================================%
\appendix
\section{Proofs of results in~\Cref{sec: basic-setting}}
\label{sec: priori-error-app}

In this appendix, we provide detailed proofs for the results 
in~\Cref{sec: basic-setting}, especially 
for~\Cref{prop: vector-field} ---~\Cref{prop: B-operator} 
(see~\Cref{subsec: vector-field} ---~\Cref{subsec: b-operator}). 
To facilitate these proofs, we first establish an elementary inequality and show 
that the Gevrey decay property is preserved under convolution. 
These preliminary results are rigorously stated in the following lemmas.

%===============================================================%
\subsection{Preliminary lemmas}
\label{subsec: preliminary}

\begin{lemma}
    \label{lemma: App-1}
    Let $a, b \ge 0$ and $0 < s < 1$. 
    Then the following inequality holds:
    \begin{equation}
        \label{eqn: ineq_power}
        a^s + b^s - (a+b)^s \ge (2-2^s) \min\{a,b\}^s. 
    \end{equation}
\end{lemma}

\begin{proof}
    Without loss of generality, we assume $a \ge b$. 
    The case $b=0$ is trivial, so we consider $b > 0$. 
    Define $x = a/b \geq 1$ and consider the function $f(x) = x^s + 1 - (1+x)^s$. 
    Since $0 < s <1$, its derivative satisfies $f'(x) = s(x^{s-1} - (x+1)^{s-1}) > 0$ 
    for all $x \geq 1$. 
    Thus, $f$ is strictly increasing on $[1, \infty)$, which implies 
    $f \geq f(1)= 2 - 2^{s}$. 
    Taking $x=a/b$, we derive the inequality~\eqref{eqn: ineq_power}, so the proof 
    is complete. 
\end{proof}

Next, we use~\Cref{lemma: App-1} to show that the Gevrey decay property is preserved 
under convolution.

\begin{lemma}
    \label{lemma: App-2}
    Let the Gevrey decay set $\mathcal{K}(s)$ be defined in~\eqref{eqn: gevrey-l2}. 
    For any $\hat{a}, \hat{b} \in \mathcal{K}(s)$, there 
    exists a constant $C_2(m, s) > 0$ such that their convolution satisfies
    \begin{equation}    
        \label{eqn: convolution}   
        \sup_{k \in \mathbb{Z}^m} \left( |(\hat{a}*\hat{b})(k)| \exp\left\{| k |^{s} \right\} \right) \leq C_2(m,s).
    \end{equation}
    More generally, if $\hat{a}_j(k) \in \mathcal{K}(s)$ for 
    $j=1, \ldots, \ell$, then there exists a constant $C_{\ell}(m, s) > 0$ such that 
    the $\ell$-fold convolution satisfies: 
    \begin{equation}    
        \label{eqn: p-convolution}   
        \sup_{k \in \mathbb{Z}^m} \left( \big| \big(\hat{a}_1 * \hat{a}_2 * \cdots * \hat{a}_{\ell} \big)(k) \big| \exp\left\{|k|^s \right\} \right) \leq C_{\ell}(m,s).
    \end{equation}   
\end{lemma}

\begin{proof}
    Given that $\hat{a}, \hat{b} \in \mathcal{K}(s)$, we estimate their convolution as
    \begin{align*}
        |(\hat{a}*\hat{b})(k)| 
        & \le \sum_{k' \in \mathbb{Z}^m} |\hat{a}(k-k')| |\hat{b}(k')|                                    \\
        & \le \sum_{k' \in \mathbb{Z}^m} \exp\left\{-|k-k'|^{s}  \right\} \exp\left\{-|k'|^{s} \right\}                               \\
        & \le \sum_{k' \in \mathbb{Z}^m} \exp\left\{-| k |^{s} \right\} \exp\left\{- \left( |k-k'|^{s}  + |k'|^{s}  - \left( |k-k'| + |k'| \right)^{s}  \right) \right\},                                                                                                                                                                   % \label{eqn: conol-1}  
    \end{align*}
    where the last step uses the triangle inequality. 
    Applying~\Cref{lemma: App-1} to bound the sum involving the remaining exponential 
    terms, we have
    \[
        |(\hat{a}*\hat{b})(k)| \leq \exp\left\{-| k |^{s} \right\} \left( \sum_{k' \in \mathbb{Z}^m} \exp\left\{ - (2 - 2^{s}) \min\left\{ | k - k'|^{s}, \; | k' |^{s} \right\} \right\}  \right).     
    \]
    Since the sum of the Gevrey decay series converges in an $m$-dimensional 
    lattice $\mathbb{Z}^m$, we establish the inequality~\eqref{eqn: convolution}. 
    By induction, this same argument extends to any finite $\ell$-fold convolutional 
    product, thereby establishing the inequality~\eqref{eqn: p-convolution}.
\end{proof}

%===============================================================%
\subsection{Proof of~\Cref{prop: vector-field}}
\label{subsec: vector-field}

Since the perturbation $H_1$ is a polynomial with real coefficients, each component 
of its associated vector field, $\partial H_1/ \partial \overline{z}_{j}$ for 
$j=1,\ldots, n$, is likewise a real-coefficient polynomial. 
Given that $\hat{z} \in \mathcal{K}(s)$,~\Cref{lemma: App-2} implies that the 
Fourier coefficients $\hat{X}(k)$ exhibit Gevrey decay. 
By applying the Leibniz rule, it follows that the derivative 
$\partial_{\omega_T}\hat{X}(k)$ also exhibits Gevrey decay, which leads directly to 
the estimate provided in~\eqref{eqn: vector-field-comp}. 
By summing these Gevrey decay series, we derive a uniform bound for the vector field 
as expressed in~\eqref{eqn: vector-field-bound}. 
This concludes the proof.

%===============================================================%
\subsection{Proof of~\Cref{prop: convolution}}
\label{subsec: convolution}

We begin by considering the full tangent operator, denoted as 
$\mathcal{H} = \partial \hat{X} / \partial \hat{z}$. 
For any $k, k' \in \mathbb{Z}^m$, its component is given by
\[   
    \mathcal{H} (k - k') = \frac{2}{(2\pi)^{m}} \int_{\mathbb{T}^{m}} \frac{\partial^2 H_1}{\partial z \partial \overline{z}} \cos \left(\langle k - k', \theta \rangle \right)  d\theta. 
\]
Following a similar argument to that used for the vector field, we 
apply~\Cref{lemma: App-2} to derive uniform bounds for $\mathcal{H}$ across the full 
lattice $\mathbb{Z}^m$. 
Since $S$ represents this operator restricted specifically to non-resonant indices, 
the uniform bound naturally extends directly to $S$, thereby verifying the 
inequality~\eqref{eqn: tangent-comp}.

To bound the operator norm of $\mathcal{H}$, we evaluate the norm of 
$\mathcal{H} \hat{z}$ for any $ \hat{z} \in \ell_2$ as
\begin{align*}
    \| \mathcal{H} \hat{z} \|^2 
    &= \sum_{k \in \mathbb{Z}^m}\bigg\| \sum_{k' \in \mathbb{Z}^m}\mathcal{H}(k') \hat{z}(k - k') \bigg\|^2 \\
    &\leq \sum_{k \in \mathbb{Z}^m} \left( \sum_{k' \in \mathbb{Z}^m} \left\| \mathcal{H}(k') \right\|^{\frac12} \cdot \left\| \mathcal{H}(k') \right\|^{\frac12} \|\hat{z}(k - k')\| \right)^2.                                                                                               
\end{align*}
Given that the operator $\mathcal{H}$ exhibits Gevrey decay, we apply the 
Cauchy-Schwarz inequality to establish the following bound
\begin{equation}
    \label{eqn: h1-bound}
    \|\mathcal{H} \hat{z}\|^2 \leq \bigg[ \sum_{k \in \mathbb{Z}^m} \|\mathcal{H}(k)\| \bigg]^2  \bigg[ \sum_{k \in \mathbb{Z}^m} \|\hat{z}(k)\|^2 \bigg] < \infty.
\end{equation}
Note that these operators, $\partial \hat{X}_q / \partial \hat{z}_p$, $S$, and 
$\partial \hat{X} / \partial \hat{z}_p$, are the full tangent operator 
$\partial \hat{X} / \partial \hat{z}$ restricted to specific index sets. 
Since the $\ell_2$-norm of a restricted operator is bounded by that of the full 
operator, we arrive at the inequality~\eqref{eqn: tangent-bound}, which completes 
the proof.

%===============================================================%
\subsection{Proof of~\Cref{prop: tensor}}
\label{subsec: tensor}

Following the same procedure used for vector fields and tangent operators, we 
analyze the third-order tensor $\partial^2 \hat{X} / \partial \hat{z}^{2}$. 
For any $k, k', k'' \in \mathbb{Z}^m$, we derive the kernels as
\begin{align*}
    \frac{\partial^2 \hat{X}(k)}{ \partial \hat{z}(k') \partial \hat{z}(k'') } 
    = & \mathcal{T}_{11}(k - k' - k'') + \mathcal{T}_{21}(k + k' - k'') + \mathcal{T}_{12}(k - k' + k'') + \mathcal{T}_{22}(k + k' + k'')  \\ 
    = & \frac{1}{(2\pi)^{m}} \int_{\mathbb{T}^{m}} \frac{\partial^3 H_1}{\partial z^2 \partial \overline{z}} e^{-i \langle k - k' - k'', \theta \rangle } d\theta + \frac{1}{(2\pi)^{m}} \int_{\mathbb{T}^{m}} \frac{\partial^3 H_1}{\partial z \partial \overline{z}^2} e^{-i \langle k + k' - k'', \theta \rangle } d\theta  \\
    + & \frac{1}{(2\pi)^{m}} \int_{\mathbb{T}^{m}} \frac{\partial^3 H_1}{\partial z \partial \overline{z}^2} e^{-i \langle k - k' + k'', \theta \rangle } d\theta + \frac{1}{(2\pi)^{m}} \int_{\mathbb{T}^{m}} \frac{\partial^3 H_1}{\partial z^2 \partial \overline{z}} e^{-i \langle k + k' + k'', \theta \rangle } d\theta.  
\end{align*}
Given that $\hat{z} \in \mathcal{K}(s)$ and the perturbation $H_1$ is a 
real-coefficient polynomial, there exists some constant 
$\gamma_8 = \gamma_8(H_1, s) > 0$ such that 
\[
    \sup_{k, k', k'' \in \mathbb{Z}^m}\bigg( \| \mathcal{T}_{ij}(k + (-1)^ik' + (-1)^jk'') \| \exp\left\{ |k + (-1)^ik' + (-1)^jk''|^s \right\} \bigg) \leq \gamma_8,
\] 
which demonstrates that the tensors $\mathcal{T}_{ij}$ for $i, j \in \{ 1, 2 \}$ 
exhibit Gevrey decay.

To bound the tensors, consider two vectors, $\hat{z}_1$ and $\hat{z}_2$, with 
support contained within $\Lambda_N$.
Taking $\mathcal{T}_{11}$ as a representative example, we apply the Cauchy-Schwarz 
inequality to derive the following bound as
\begin{align*}
    &\| \mathcal{T}_{11} ( \hat{z}_1, \hat{z}_2 ) \|^2 \\
    =& \sum_{k \in \mathbb{Z}^m} \bigg\| \sum_{k'' \in \Lambda_N} \sum_{k' \in \Lambda_N} \mathcal{T}_{11} (k - k' - k'') \hat{z}_1(k') \hat{z}_2(k'' ) \bigg\|^2 \\
    \leq& \sum_{k \in \mathbb{Z}^m} \left[ \sum_{k', k'' \in \Lambda_N} \left\| \mathcal{T}_{11} (k - k' - k'') \right\| \right] \cdot \left[ \sum_{k', k'' \in \Lambda_N} \left\| \mathcal{T}_{11} (k - k' - k'') \right\| \|\hat{z}_1(k')\|^2 \|\hat{z}_2(k'')\|^2 \right]  \\
    \leq& \left[ \sum_{k', k'' \in \Lambda_N} \gamma_8 e^{-|k - k' - k''|^s} \right] \cdot \left[ \sum_{k \in \mathbb{Z}^m} \sum_{k', k'' \in \Lambda_N} \gamma_8 e^{-|k - k' - k''|^s} \|\hat{z}_1(k')\|^2 \|\hat{z}_2(k'')\|^2 \right]  \\
    \leq& 2\gamma_8^2 (2N+1)^m \left[ \sum_{k \in \mathbb{Z}^m} e^{-|k|^s} \right] \left[ \sum_{k' \in \Lambda_N} \|\hat{z}_1(k')\|^2 \right] \left[ \sum_{k'' \in \Lambda_N} \|\hat{z}_2(k'')\|^2 \right] \\
    =& 2\gamma_8^2 (2N+1)^m \left[ \sum_{k \in \mathbb{Z}^m} e^{-|k|^s} \right] \cdot \|\hat{z}_1\|^2 \|\hat{z}_2\|^2 < \infty.                                                               
\end{align*}
Similarly, we establish the corresponding bounds for $\mathcal{T}_{21}$, 
$\mathcal{T}_{12}$, and $\mathcal{T}_{22}$. 
These operators, $\partial^2 \hat{X}_q / \partial \hat{z}_p^2$, 
$\partial S/\partial \hat{z}_p$, and $\partial^2 \hat{X} / \partial \hat{z}_p^2$, 
are restrictions of the full tangent operator $ \partial^2 \hat{X} / \partial \hat{z}^2$ to specific index sets. 
Since the $\ell_2$-norm of a restricted operator is bounded by that of the full 
operator, we arrive at the inequality~\eqref{eqn: tensor}, which completes the proof.

%===============================================================%
\subsection{Proof of~\Cref{prop: B-operator}}
\label{subsec: b-operator}

Recalling the operator $B$ defined in~\eqref{eqn: b-operator}, we observe 
that for any $k, k' \notin \mathcal{S}$, its entries satisfy the following 
inequality: 
\[
    \| B(k, k') \| \leq \frac{1}{e} \left\| \frac{ \partial \langle k, \hat{X}_q \rangle  }{ \partial \hat{z}_p(k') }~\hat{z}_p (k) \right\|.
\]
For any $\hat{z} \in \mathcal{K}(s)$, we apply the decay properties 
established in~\Cref{prop: vector-field} to verify the Gevrey bound 
in~\eqref{eqn: b-comp}. 
Given that $B$ is a rank-one operator, we apply the Cauchy-Schwarz inequality to 
bound its norm as follows: 
\[
    \| B(k, k') \| \leq \frac{1}{e} \| k \hat{z}_p(k) \| \cdot \left\| \frac{ \partial \hat{X}_q }{ \partial \hat{z}_p(k') } \right\|. 
\]
By utilizing the Gevrey decay of the tangent operators derived 
in~\Cref{subsec: convolution}, we extend the bound over the lattice 
$\mathbb{Z}^m \times \mathbb{Z}^m$ as
\begin{equation*}
    \| B \| \leq \sum_{k \in \mathbb{Z}^m} \left( \frac{\| k \hat{z}_p(k) \|}{e} \right) \cdot \sum_{k' \in \mathbb{Z}^m} \left( \left\|  \frac{\partial  \hat{X}_q}{\partial \hat{z}_p} (k') \right\| \right).
\end{equation*}
Finally, since $\hat{z} \in \mathcal{K}(s)$, we 
invoke~\Cref{prop: vector-field} and~\Cref{prop: convolution} to ensure the 
convergence of the infinite sum over the lattice, thereby confirming the uniform 
bound as stated in~\eqref{eqn: b-bound}.

\section{Proofs of results in~\Cref{sec: implement}}
\label{sec: implement-app}

In this appendix, we provide the detailed proofs for the results presented 
in~\Cref{sec: implement}, specifically addressing~\Cref{lem: near-mes-initial} 
and~\Cref{thm: small-size}.

%=================================================%
\subsection{Proof of~\Cref{lem: near-mes-initial} }
\label{sec: near-mes-initial}

Let $\Omega \subseteq \mathbb{R}^m$ be a bounded domain, and denote its diameter as 
$d = \max_{x,y \in \Omega} \| x- y \|_2$. 
For each integer vector $k \in \mathbb{Z}^{m} \setminus \{0\} $, the corresponding 
nearly-resonant strip is enclosed by two parallel hyperplanes perpendicular to $k$. 
Consequently, the measure of the tangent nearly-resonant sets defined 
in~\eqref{eqn: diophantine-nearly} satisfies 
\[
    \mathrm{mes}\left( \Omega_{0,M}^{\tau} \right) \leq \sum_{k \in \mathbb{Z}^m \setminus \{0\} } \frac{2 d^{m-1}}{ (\| k \|_2)^{\tau + 1} }.
\]
Similarly, we estimate the measures for the total single-mode resonant 
set~\eqref{eqn: 1-melnikov} and the total difference resonant 
set~\eqref{eqn: 2-melnikov} as
\[
    \left\{ 
    \begin{aligned}
        &  \mathrm{mes}\left( \Omega_{1,M}^{\tau} \right) \leq \sum_{k \in \mathbb{Z}^m} \frac{2(n - m) d^{m-1} }{ (\| k \|_2+1)^{\tau + 1}}, \\
        &  \mathrm{mes}\left( \Omega_{2,M}^{\tau} \right) \leq \sum_{k \in \mathbb{Z}^m} \frac{4(n-m)^2 d^{m-1} }{ (\| k \|_2+2)^{\tau + 1}}.
    \end{aligned} 
    \right.
\]
Since the exponent satisfies $\tau > m - 1$, the lattice series 
$\sum_{k \in \mathbb{Z}^{m} \setminus \{ 0\} } \| k \|_2^{-\tau - 1}$ converges. 
By combining the above estimates, we establish the desired 
bound~\eqref{eqn: near-mes-initial}, which completes the proof.

%=================================================%
\subsection{Proof of~\Cref{thm: small-size}}
\label{sec: small-size-original}

Before proceeding to the proof of~\Cref{thm: small-size}, we establish a lemma 
regarding the preservation of the Gevrey decay property under operator multiplication. 
We say that an operator $A : \ell_2(\mathbb{Z}^{m}) \mapsto \ell_2(\mathbb{Z}^{m})$ 
exhibits \textit{proper decay} if its entries satisfy: 
\[
\sup_{k, k' \in \mathbb{Z}^m} \left( |A(k, k')| \exp\left\{ |k - k'|^s \right\} \right) \leq 1.
\]

\begin{lemma}
    \label{lemma: App-3}
    Let $A_1$ and $A_2$ be two operators with proper decay. 
    Then there exists a constant $C'_2(m, s) > 0$ such that their product 
    $\mathcal{A} = A_1A_2$ satisfies
    \begin{equation}    
        \label{eqn: convolution-mix}   
        \sup_{k, k' \in \mathbb{Z}^m} \left( |\mathcal{A}(k, k')| \exp\left\{ |k - k'|^s \right\} \right) \leq C'_2(m,s).
    \end{equation}
    More generally, for any finite collection of operators $\{ A_j \}_{j=1}^{\ell}$ 
    with proper decay, there exists a constant $C'_{\ell}(m, s) > 0$ such that the 
    $\ell$-fold product $\mathcal{A}=\prod_{j=1}^{\ell}A_j $ satisfies: 
    \begin{equation}    
        \label{eqn: p-convolution-mix}   
        \sup_{k,k' \in \mathbb{Z}^m} \left( |\mathcal{A}(k, k')| \exp\left\{ |k - k'|^s \right\} \right) \leq C'_{\ell}(m,s).
    \end{equation}   
\end{lemma}

\begin{proof}
    Given that $A_1$ and $A_2$ possess proper decay, the entries of their product 
    are bounded by:
    \begin{align*}
        |\mathcal{A}(k, k')| 
        & \le \sum_{k'' \in \mathbb{Z}^m} |A_1(k,k'')| |A_2(k'', k')|                                  \\
        & \le \sum_{k'' \in \mathbb{Z}^m}  \exp\left\{ - |k - k''|^s  - |k '' - k'|^s \right\}                    \\
        & \le \exp\left\{-| k - k' |^{s} \right\}  \sum_{k'' \in \mathbb{Z}^m} \exp\left\{ - (2 - 2^{s}) \min\left\{ | k - k'' |^{s}, \; | k'' - k' |^{s} \right\} \right\},      
    \end{align*}
    where the last inequality follows from~\Cref{lemma: App-1}. 
    Since the Gevrey decay series converges, the summation above is bounded by a 
    constant depending only on $m$ and $s$, thus we establish the 
    inequality~\eqref{eqn: convolution-mix}. 
    The general estimate~\eqref{eqn: p-convolution-mix} follows immediately by 
    induction on $\ell$.
\end{proof}

We now complete the proof of~\Cref{thm: small-size}. 
Since the diagonal operator $D_{N}$ admits the uniform lower 
bound~\eqref{eqn: uniform-lower-original-small}, it is invertible. 
This allows us to factor the restricted operator $(T + \varepsilon B)_{N}$ as follows: 
\[
    (T + \varepsilon B)_{N}= D_{N} + \varepsilon (S_{N}+B_{N}) = D_{N} \left[ I + \varepsilon D_{N}^{-1}(S_{N}+B_{N}) \right].
\]
Given the smallness of the perturbation parameter $\varepsilon$, the operator 
$(T + \varepsilon B)_{N}$ is invertible, and its inverse can be expanded as the
Neumann series: 
\begin{equation}
    \label{eqn: T-nu-expansion}
    (T+\varepsilon B)_{N}^{-1} = \left[ I + \varepsilon D_{N}^{-1}(S_{N}+B_{N}) \right]^{-1} D_{N}^{-1} = \left\{ \sum_{\ell=0}^{\infty} \left[ - \varepsilon D_{N}^{-1}(S_{N}+B_{N})  \right]^{\ell} \right\} D_{N}^{-1}.
\end{equation}
By applying the bound for $\| D_{N}^{-1} \|_2$ 
from~\eqref{eqn: uniform-lower-original-small} alongside the estimates 
from~\Cref{prop: convolution} and~\Cref{prop: B-operator}, we derive the following 
bound
\[
    \| (T+\varepsilon B)_{N}^{-1} \|_2 \leq \| D_{N}^{-1} \|_2 \left( \sum_{\ell =0}^{\infty} \| \varepsilon D_{N}^{-1}(S_{N} + B_{N})\|_2^{\ell} \right) \leq \frac{1}{\varepsilon_{N}},
\] 
which confirms the norm bound~\eqref{eqn: small-size-inverse}. 
To establish the off-diagonal decay, we rewrite the Neumann 
series~\eqref{eqn: T-nu-expansion} to isolate the higher-order terms:
\[
    (T+\varepsilon B)_{N}^{-1} = D_{N}^{-1} - \left[ \sum_{\ell=0}^{\infty} \left( - \varepsilon D_{N}^{-1}(S_{N} +B_N) \right)^{\ell} \right] \left(\varepsilon D_{N}^{-1}(S_{N}+B_{N}) D_{N}^{-1}\right).
\]
Utilizing the Gevrey decay properties for the entries of the operators $S$ and 
$B$ (\Cref{prop: convolution} and~\Cref{prop: B-operator}),~\Cref{lemma: App-3} 
implies that for any $k \neq k'$:
\[
    \left| \left( - \varepsilon D_{N}^{-1} (S_{N}+B_{N}) \right)^{\ell} \left(  k, k' \right) \right| \leq  \left(\frac{1}{2}\right)^{\ell+1} \exp\left\{- | k - k' |^{s} \right\}.
\]
By summing the corresponding Neumann series, we establish the inequality for 
the off-diagonal entries~\eqref{eqn: small-size-localization}. 
The proof of~\Cref{thm: small-size} is thus complete.

\section{Proofs of results in~\Cref{sec: iterative-scheme}}
\label{sec: completion-app}

In this appendix, we provide the full proof for~\Cref{lem: inverse-derivative}.

%==============================================================================%
\subsection{Proof of~\Cref{lem: inverse-derivative}}
\label{subsec: derivative-app}

From~\Cref{prop: convolution} and~\Cref{prop: B-operator}, it follows that for any 
$k \neq  k'$, the off-diagonal entries of $\partial T_N$ satisfy
\[
    |\partial L_N(k, k')| \leq \exp\left\{ - | k - k' |^{s} \right\}.
\]
Following the operator identity in~\eqref{eqn: t-n-derivative-inverse}, we 
expand each entry $\partial L_{N}^{-1}(k, k')$ as a double sum as:
\[
    \partial L_{N}^{-1}(k, k') = - \sum_{k_1 \in  \Lambda_N} \sum_{k_2 \in  \Lambda_N}L_{N}^{-1}(k, k_1) \partial L_N(k_1, k_2) L_{N}^{-1}(k_2, k')
\]
By applying the triangle inequality and substituting the bounds 
from~\Cref{lemma: App-3}, we obtain: 
\begin{align*}
|\partial L_N^{-1}(k, k')| 
    &\leq \sum_{k_1 \in  \Lambda_N} \sum_{k_2 \in  \Lambda_N } |L_{N}^{-1}(k, k_1)| \cdot |\partial L_{N} (k_1, k_2)| \cdot |L_{N}^{-1}(k_2, k')| \\
    &\leq \frac{4N^2 (2N+1)^{2m}}{\varepsilon_N^2}  \left(\exp\left\{ - \frac{\left( |k - k'| - 2N^{\frac12} \right)^{s}}{2} \right\} \right).
\end{align*}
Given that the spatial separation satisfies $|k - k'| \geq N^{\frac34}$, 
the off-diagonal entries of $\partial L_{N}^{-1}$ satisfy the Gevrey 
bound~\eqref{eqn: inverse-restricted-entries}. 
The proof thus is complete.

%\appendix
%\input{022_newton}
%\input{033_property}
%\input{044_small}
%\input{05_maltiscale}

\end{document}